\documentclass{amsart}
\usepackage{amsmath}
\usepackage{amssymb}
\usepackage{amsfonts}
\usepackage{graphicx}
\usepackage{comment}
\usepackage{subfigure}
\usepackage{epsfig}
\usepackage{epstopdf}
\usepackage{float}
\restylefloat{figure}

\newtheorem{theorem}{Theorem}[section]
\newtheorem{corollary}[theorem]{Corollary}
\newtheorem{lemma}[theorem]{Lemma}
\newtheorem{proposition}[theorem]{Proposition}
\newtheorem{definition}[theorem]{Definition}
\newtheorem{example}[theorem]{Example}
\newtheorem{notation}[theorem]{Notation}
\newtheorem{remark}[theorem]{Remark}

\begin{document}
\title[Local dimensions]{Local dimensions of measures of finite
type}
\author[K.~E.~Hare, K.~G.~Hare, K.~R.~Matthews]{Kathryn E. Hare, Kevin G.
Hare, and Kevin R. Matthews}
\thanks{ The research of all three authors is supported in part by NSERC.}

\maketitle

\begin{abstract}
We study the multifractal analysis of a class of equicontractive,
self-similar measures of finite type, whose support is an interval. Finite
type is a property weaker than the open set condition, but stronger than the
weak open set condition. Examples include Bernoulli convolutions with
contraction factor the inverse of a Pisot number and self-similar measures
associated with $m$-fold sums of Cantor sets with ratio of dissection $1/R$
for integer $R\leq m$.

We introduce a combinatorial notion called a loop class and prove that the
set of attainable local dimensions of the measure at points in a positive
loop class is a closed interval. We prove that the local dimensions at the
periodic points in the loop class are dense and give a simple formula for
those local dimensions. These self-similar measures have a distinguished
positive loop class called the essential class. The set of points in the
essential class has full Lebesgue measure in the support of the measure and
is often all but the two endpoints of the support. Thus many, but not all,
measures of finite type have at most one isolated point in their set of
local dimensions.

We give examples of Bernoulli convolutions whose sets of attainable local
dimensions consist of an interval together with an isolated point. As well,
we give an example of a measure of finite type that has exactly two distinct
local dimensions.
\end{abstract}

%\begin{classification}
%Primary 28A80; Secondary 28C10, 11K16
%\end{classification}
%\begin{keywords}
%local dimension, multifractal analysis, finite type, Bernoulli
%convolution, IFS
%\end{keywords}
%\thanks{This paper is in final form and no version of it will be submitted
%for publication elsewhere.}

\section{Introduction}

It is well known that if $\mu $ is a self-similar measure arising from an
IFS satisfying the open set condition, then the set of local dimensions of
the measure is a closed interval whose endpoints are easily computed.
Further, the Hausdorff dimension of the set of points whose local dimension
is a given value can be determined using the Legendre transform of the $%
L^{q} $-spectrum of the measure. This is known as the multifractal formalism
and we refer the reader to \cite{Fa} for more details.

For measures that do not satisfy the open set condition, the multifractal
analysis is more complicated and, in general, much more poorly understood.
In \cite{HL}, Hu and Lau discovered that the $3$-fold convolution of the
classical middle-third Cantor measure fails the multifractal formalism as
there is an isolated point in the set of local dimensions. Subsequently, in 
\cite{BHM, Sh, Te} further examples of this phenomena were explored and it
was shown, for example, that there is always an isolated point in the set of
local dimensions of the $m$-fold convolution of the Cantor measure
associated with a Cantor set with ratio of dissection $1/R,$ when the
integer $R\leq m$. More recently, it was proven in \cite{BH} that continuous
measures satisfying a weak technical condition have the property that a
suitably large convolution power admits an isolated point in its set of
local dimensions.

\ In \cite{NW}, Ngai and Wang introduced the notion of finite type (see
Section 2 for the definition). This property is stronger than the weak
separation condition introduced in \cite{LN}, but is satisfied by many
self-similar measures which fail to possess the open set condition. Examples
include Bernoulli convolutions, $\mu _{\varrho }$, with contraction factor $%
\varrho $ equal to the reciprocal of a Pisot number \cite{NW} and the
Cantor-like measures mentioned above.

Building on earlier work (c.f., \cite{FO,Hu,LN1,Po}), Feng undertook a study
of equicontractive, self-similar measures of finite type in \cite{F3,F1,F2}.
His main results were for Bernoulli convolutions. In particular, he proved
that despite the failure of the open set condition, the multifractal
formalism still holds for the Bernoulli convolutions whose contraction
factor was the reciprocal of a simple Pisot number (meaning, a Pisot number
whose minimal polynomial is of the form $x^{n}-x^{n-1}-\dots -x-1$). A
particularly interesting example is when the contraction factor is the
golden ratio with minimal polynomial $x^2-x-1$ (also called the golden mean).

In this\ paper we study the local dimension theory of equicontractive,
self-similar measures $\mu $ of finite type, whose support is a compact
interval and for which the underlying probabilities are regular. We first
give a simple formula for the value of the local dimension of $\mu $ at any
``periodic'' point of its support. The finite type condition leads naturally
to a combinatorial notion we call a ``loop class''. For a ``positive'' loop
class we prove that the set of attainable local dimensions of the measure is
a closed interval and that the set of local dimensions at periodic points in
the loop class is a dense subset of this interval.

These self-similar measures have a distinguished positive loop class that is
called the \textquotedblleft essential class\textquotedblright . Thus the
set of attainable local dimensions is the union of a closed interval
together with the local dimensions at points in loop classes external to the
essential class. We will say that a point is an essential point if it is in
the essential class. The set of essential points has full Lebesgue measure
on the support of the measure and in many interesting examples the set of
essential points is the interior of the support of the measure. This is the
case with many Bernoulli convolutions, $\mu _{\varrho }$, including when $%
\varrho ^{-1}$ is the golden ratio (c.f. Section \ref{ex:other Pisot}), and
with the $m$-fold convolution of the Cantor measure on a Cantor set with
ratio of dissection $1/R$ when $R \leq m$ (see Section \ref{sec:Cantor}).

When the essential set is the interior of the support of the measure $\mu ,$
then $\mu $ has no isolated point in its set of attainable local dimensions
if and only if $\dim \mu (0)$ coincides with the local dimension of $\mu $
at an essential point. In that case, the set of attainable local dimensions
of $\mu$ is a closed interval. The Bernoulli convolution $\mu _{\varrho }$,
with $\varrho ^{-1}$ a simple Pisot number has this property.

However, we construct other examples of Bernoulli convolutions (with
contraction factor a Pisot inverse) which do have an isolated point in their
set of attainable local dimensions (see Subsection \ref{BCisolatedpt}). As
far as we are aware, these are the first examples of Bernoulli convolutions
known to admit an isolated point. We also construct a Cantor-like measure of
finite type, whose set of local dimensions consists of (precisely) two
distinct points (see Example \ref{ex:isolated points}). In all of these
examples, the essential set is the interior of the support of the measure.

The convolution square of the Bernoulli convolution, $\mu _{\varrho }$, with 
$\varrho ^{-1}$ the golden ratio, is another example of a self-similar
measure to which our theory applies. It, too, has exactly one isolated point
in its set of attainable local dimensions, although in this case the set of
non-essential points is countably infinite (see Subsection \ref{ssec:golden
square}).

The computer was used to help obtain some of these results. In principle,
the techniques could be applied to other convolutions of Bernoulli
convolutions and other measures of finite type, however even with the simple
examples given here, the problem can become computationally difficult.

The paper is organized as follows: In Section \ref{sec:term}, we detail the
structure of self-similar measures of finite type, introduce terminology and
describe a number of examples that we will return to throughout the paper.
The notion of transition matrices and properties of local dimensions of
measures of finite type are discussed in Section \ref{sec:trans}. In Section %
\ref{sec:loop} we introduce the notion of loop class, essential class and
periodic points. A formula is given for the local dimension at a periodic
point and we prove that the essential class is always of positive type. In
Section \ref{sec:main} we prove that the set of local dimensions at periodic
points in a positive loop class is dense in the set of local dimensions at
all points in the loop class. We also show that the set of local dimensions
at the points of a positive loop class is a closed interval. In particular,
this implies that the set of local dimensions at the essential points is a
closed interval.

In Section \ref{sec:comp} we give a detailed description of our computer
algorithm by means of a worked example. We also explain our main techniques
for finding bounds on sets of local dimensions and illustrate these by
constructing a Cantor-like measure of finite type whose local dimension is
the union of two distinct points. In Section \ref{sec:Cantor} we show that
with our approach we can partially recover results from \cite{BHM, HL, Sh}
about the local dimensions of Cantor-like measures of finite type. We also
show that some facts about the endpoints of the interval portion of the
local dimension, that are known to be true for Cantor-like measures in the
\textquotedblleft small" overlap case, do not hold in general. Bernoulli
convolutions, $\mu _{\varrho }$, where $\varrho $ is the reciprocal of a
Pisot number of degree at most four, are studied in Section \ref{sec:Pisot}
and we see that two of these measures admit an isolated point. We also study
the convolution square of the Bernoulli convolution with the golden ratio in
this final section.

For the examples in this paper, we present only minimal information. 
A more detailed analysis of all of these examples can found as supplemental 
information appended to the arXiv version of the paper \cite{Homepage}.

\section{Terminology and examples}

\label{sec:term}

\subsection{Finite Type}

Consider the iterated function system (IFS) consisting of the contractions $%
S_{j}:\mathbb{R\rightarrow R}$, $j=0,\dots,m$, defined by 
\begin{equation}
S_{j}(x)=\varrho x+d_{j}
\end{equation}
where $0<\varrho <1$, $d_0<d_{1}<d_{2}<\cdot \cdot \cdot <d_{m}$ and $m\geq
1 $ is an integer. By the associated self-similar set, we mean the unique,
non-empty, compact set $K$ satisfying 
\begin{equation*}
K=\bigcup\limits_{j=0}^{m}S_{j}(K).
\end{equation*}

Suppose $p_{j}$, $j=0,\dots,m$ are probabilities, i.e., $p_{j}>0$ for all $j$
and $\sum_{j=0}^{m}p_{j}=1$. Our interest is in the self-similar measure $%
\mu $ associated to the family of contractions $\{S_{j}\}$ as above, which
satisfies the identity 
\begin{equation}
\mu =\sum_{j=0}^{m}p_{j}\mu \circ S_{j}^{-1}.  \label{ss}
\end{equation}%
These measures are sometimes known as \textit{equicontractive}, or \textit{$%
\varrho $-equicontractive} if we want to emphasize the contraction factor $%
\varrho $. They are non-atomic, probability measures whose support is the
self-similar set.

We put $\mathcal{A}=\{0,\dots,m\}$. Given an $n$-tuple $\sigma
=(j_{1},\dots,j_{n})$ $\in \mathcal{A}^{n}$, we write $S_{\sigma }$ for the
composition $S_{j_{1}}\circ \cdot \cdot \cdot \circ S_{j_{n}}$ and let 
\begin{equation*}
p_{\sigma }=p_{j_{1}}\cdot \cdot \cdot p_{j_{n}}\text{.}
\end{equation*}

\begin{definition}
The iterated function system, $\{S_{j}(x)=\varrho x+d_{j}:j=0,\dots,m\}$, is
said to be of \textbf{finite type}\textit{\ }if there is a finite set $%
F\subseteq \mathbb{R}^{+}$ such that for each positive integer $n$ and any
two sets of indices $\sigma =(j_{1},\dots,j_{n})$, $\sigma ^{\prime
}=(j_{1}^{\prime },\dots,j_{n}^{\prime })$ $\in \mathcal{A}^{n}$, either 
\begin{equation*}
\varrho ^{-n}\left\vert S_{\sigma }(0)-S_{\sigma ^{\prime }}(0)\right\vert
>c \text{ or }\varrho ^{-n}(S_{\sigma }(0)-S_{\sigma ^{\prime }}(0))\in F,
\end{equation*}
where $c=(1-\varrho )^{-1}(\max d_{j}-\min d_{j})$ is the diameter of $K$.

If $\{S_{j}\}$ is of finite type and $\mu $ is an associated self-similar
measure, we also say that $\mu $ is of finite type.
\end{definition}

It is worth noting here that the definition of finite type is independent of
the choice of probabilities.

In \cite{Ng}, Nguyen showed that any IFS satisfying the open set condition
is of finite type, but there are also examples of iterated function systems
of finite type that fail the open set condition. Indeed, recall that an
algebraic integer greater than $1$ is called a \textit{Pisot number} if all
its Galois conjugates are less than $1$ in absolute value. Examples include
integers greater than $1$ and the golden ratio, $(1+\sqrt{5})/2$. In \cite[%
Thm. 2.9]{NW}, Ngai and Wang showed that if $\varrho ^{-1}$ is a Pisot
number and all $d_{j}\in \mathbb{Q[\varrho }^{-1}]$, then the measure $\mu $
satisfying (\ref{ss}) is of finite type. This result allows us to produce
many examples of measures of finite type that do not satisfy the open set
condition.

The case when the IFS is generated by two contractions is of particular
interest.

\begin{notation}
We will use the notation $\mu _{\varrho }$ to denote the self-similar
measure 
\begin{equation*}
\mu _{\varrho }=\frac{1}{2}\mu _{\varrho }\circ S_{0}^{-1}+\frac{1}{2}%
\mu_{\varrho }\circ S_{1}^{-1},
\end{equation*}
where $S_{j}=\varrho x+j(1-\varrho )$ for $j=0,1$.
\end{notation}

\begin{example}
\label{ex:pisot} When $0<\varrho \leq 1/2,$ the measures, $\mu _{\varrho },$
are known as \textbf{Cantor measures} (or uniform Cantor measures). Their
support is the Cantor set with ratio of dissection $\varrho $ and they
satisfy the open set condition. When $1/2<\varrho <1,$ these measures are
called \textbf{Bernoulli convolutions}. They fail to satisfy the open set
condition, but are of finite type whenever $\varrho ^{-1}$ is a Pisot number.
\end{example}

Given two probability measures, $\mu ,\nu $, the convolution of $\mu $ and $%
\nu $ is defined as 
\begin{equation*}
\mu \ast \nu (E)=\mu \times \nu \{(x,y):x+y\in E\}.
\end{equation*}%
The name \textquotedblleft Bernoulli convolution\textquotedblright\ comes
from the fact that 
\begin{equation*}
\mu _{\varrho }=\ast _{n=1}^{\infty }\left( \frac{\delta _{0}+\delta
_{(1-\varrho )\varrho ^{n}}}{2}\right) ,
\end{equation*}%
where the infinite convolution is understood to converge in a weak sense.

Bernoulli convolutions, $\mu _{\varrho }$, with contraction factor $\varrho $
the inverse of a Pisot number, have been long studied. They have unusual
properties and are of interest in fractal geometry, number theory and
harmonic analysis. For example, although almost every Bernoulli convolution
is absolutely continuous with respect to Lebesgue measure, and even has an $%
L^{2}$ density function, those with a Pisot inverse as the contraction
factor are not only purely singular, but their Fourier transform, $\widehat{%
\mu _{\rho }}(y)$, does not even tend to zero as $y\rightarrow \pm \infty $.
We refer the reader to \cite{PSS} and \cite{So} for some of the interesting
history of these measures.

\begin{example}
\label{ex:convolution} Suppose $\mu $ and $\nu $ are $\varrho $%
-equicontractive measures, say $\mu =\sum_{i}p_{i}\mu \circ S_{i}^{-1}$ and $%
\nu =\sum_{j}q_{j}\mu \circ T_{j}^{-1}$ where $S_{i}(x)=\varrho x+d_{i}$ and 
$T_{j}(x)=\varrho x+e_{j}$. Then $\mu \ast \nu $ is the $\varrho $%
-equicontractive, self-similar measure satisfying 
\begin{equation*}
\mu \ast \nu =\sum_{t}r_{t}(\mu \ast \nu )\circ U_{t}^{-1}
\end{equation*}%
where $U_{t}(x)=\sum_{d_{i}+e_{j}=f_{t}}\varrho x+f_{t}$ and $%
r_{t}=\sum_{d_{i}+e_{j}=f_{t}}p_{i}q_{j}$. It follows directly from Ngai and
Wang's result \cite{NW} that any $m$-fold convolution power of the Bernoulli
convolution or Cantor measure, $\mu _{\varrho },$ is of finite type when $%
\varrho ^{-1}$ is Pisot.
\end{example}

\begin{example}
\label{ex:shmerkin} Another consequence of \cite{NW} is that the IFS 
\begin{equation*}
\left\{ S_{j}(x)=\frac{1}{R}x+\frac{j}{Rm}(R-1):j=0,\dots,m\right\} ,
\end{equation*}%
where $R\geq 2$ is an integer, is of finite type. The convex hull of the
self-similar set is $[0,1]$ and the self-similar set is the full interval $%
[0,1]$ when $m\geq R-1$. When $m\geq R$, the open set condition is not
satisfied. The $m$-fold convolutions of Cantor measures with contraction
factor $1/R$ are examples of self-similar measures associated with such an
IFS.

These Cantor-like measures were studied in \cite{BHM, Sh} using different
methods. In Section \ref{sec:Cantor} we will see how our approach relates to
some of their results.
\end{example}

\subsection{Standard Technical Assumptions}

We will refer to the following conditions on a self-similar measure $\mu $
as our standard technical assumptions.

\begin{enumerate}
\item The measure $\mu =\sum_{j}p_{j}\mu \circ S_{j}^{-1}$ is a $\varrho $%
-equicontractive, self-similar measure, as in (\ref{ss}), that is of finite
type.

\item The probabilities, $\{p_{j}\}_{j=0}^{m}$ satisfy $p_{0}=p_{m}=\min
p_{j}$ (we call these \textit{regular probabilities}).

\item The support of $\mu $ (equivalently, the underlying self-similar set)
is a closed interval. By rescaling the $d_{j}$ appropriately, we can assume
without loss of generality that this interval is $[0,1]$.
\end{enumerate}

We remark that supp$\mu =[0,1]$ if and only if (the rescaled) $\{d_{j}\}$
satisfy $d_{0}=0$, $d_{m}=1-\varrho $ and $d_{i+1}-d_{i}\leq \varrho $ for
all $i=0,\dots,m-1$. In this case, $c=1$ in the definition of finite type.

Although some of what we say is true more generally for self-similar
measures of finite type, we make use of the standard technical assumptions
at key points throughout the paper.

The Bernoulli convolutions $\mu _{\varrho }$ and the $m$-fold convolutions
of uniform Cantor measures $\mu _{\varrho }$ with $\varrho >1/(m+1)$ are
examples of measures satisfying the standard technical assumptions. (See
Examples \ref{ex:pisot} and \ref{ex:convolution}). The measures of Example %
\ref{ex:shmerkin}, where $m\geq R-1$, are also examples of such measures
when regular probabilities are chosen.

\subsection{Net intervals and Characteristic vectors}

As we have seen, measures that are of finite type need not satisfy the open
set condition. Our primary interest is in this case. The finite type
property is, however, stronger than the weak separation condition (see \cite%
{Ng} for a proof), and the multifractal analysis of self-similar measures of
finite type is somewhat more tractable because of their better structure.
This structure is explained in detail in \cite{F3, F1, F2}, but we will give
a quick overview here.

For each integer $n$, let $h_{1},\dots ,h_{s_{n}}$ be the collection of
elements of the set $\{S_{\sigma }(0),$ $S_{\sigma }(1):\sigma \in \mathcal{A%
}^{n}\}$, listed in increasing order. Put 
\begin{equation*}
\mathcal{F}_{n}=\{[h_{j},h_{j+1}]:1\leq j\leq s_{n}\}\text{.}
\end{equation*}%
Elements of $\mathcal{F}_{n}$ are called \textit{net intervals of level }$n$%
. For each $\Delta \in \mathcal{F}_{n}$, $n\geq 1$, there is a unique
element $\widehat{\Delta }\in \mathcal{F}_{n-1}$ which contains $\Delta $.
We call $\widehat{\Delta }$ the \textit{parent} of $\Delta $ and $\Delta $ a 
\textit{child} of $\widehat{\Delta }$. We denote the \textit{normalized
length} of $\Delta =[a,b]$ by 
\begin{equation*}
\ell _{n}(\Delta )=\varrho ^{-n}(b-a)\text{.}
\end{equation*}%
Note that by definition there is no $\sigma \in \mathcal{A}^{n}$ with $%
a<S_{\sigma }(0)<b$, nor can we have $a<S_{\sigma }(1)<b$. Furthermore,
there must be some $\sigma _{1},\sigma _{2}$ with $S_{\sigma
_{1}}(x)=a,S_{\sigma _{2}}(y)=b$ for suitable choices of $x,y$ $\in \{0,1\}$.

Next, we consider all $\sigma \in \mathcal{A}^{n}$ with $\Delta \subseteq
S_{\sigma }[0,1]$. As $S_{\sigma }[0,1]$ is a closed interval of length $%
\varrho ^{n},$ this is the same as the set of all $\sigma \in \mathcal{A}%
^{n} $ with $a-\varrho ^{n}<S_{\sigma }(0)\leq a$. We suppose 
\begin{equation*}
\{\varrho ^{-n}(a-S_{\sigma }(0)):\sigma \in \mathcal{A}^{n}\text{,
\thinspace }\Delta \subseteq S_{\sigma }[0,1]\}=\{a_{1},\dots ,a_{k}\}
\end{equation*}%
and assume $a_{1}<a_{2}<\cdot \cdot \cdot <a_{k}$. We define the \textit{%
neighbour set} of $\Delta $ as 
\begin{equation*}
V_{n}(\Delta )=(a_{1},\dots ,a_{k}).
\end{equation*}

Let $\widehat{\Delta }\in \mathcal{F}_{n-1}$ be the parent of $\Delta $, and 
$\Delta _{1},\dots,\Delta _{m}$ (listed in order from left to right) be all
the net intervals of level $n$ which are also children of $\widehat{\Delta }$
and have the same normalized length and neighbour set as $\Delta $. Define $%
r_{n}(\Delta )$ to be the integer $r$ with $\Delta _{r}=\Delta $. The 
\textit{characteristic} \textit{vector of }$\Delta $ is the triple 
\begin{equation*}
\mathcal{C}_{n}(\Delta )=(\ell _{n}(\Delta ),V_{n}(\Delta ),r_{n}(\Delta )).
\end{equation*}
We also speak of the pair of characteristic vectors, $\alpha ,\beta ,$ as
parent and child if $\alpha =\mathcal{C}_{n-1}(\widehat{\Delta })$ and $%
\beta =\mathcal{C}_{n}(\Delta )$ for a parent/child pair $\widehat{\Delta }
,\Delta $. The characteristic vector is important because it carries the
neighbourhood information about $\Delta $.

Put 
\begin{equation*}
\Omega =\{\mathcal{C}_{n}(\Delta ):n\in \mathbb{N}\text{, }\Delta \in 
\mathcal{F}_{n}\}\text{.}
\end{equation*}
If the measure is of finite type, then $\Omega $ will contain only finitely
many distinct characteristic vectors.

Suppose the net interval, $\widehat{\Delta },$ has two children, $\Delta
_{1} $ and $\Delta _{2},$ that differ only in the value of $r_{n}(\Delta
_{i})$, that is, they have the same length and the same neighbourhood set.
The characteristic vectors for the children of $\Delta _{1}$ and $\Delta _{2}
$, will be identical as they depend only on $\ell _{n}(\Delta _{i})$ and $%
V_{n}(\Delta _{i}),$ and not on $r_{n}(\Delta _{i})$. For notational
reasons, we find it convenient to take advantage of this when drawing the
directed graph relating parents to children by suppressing $r_{n}(\Delta
_{i}),$ equating these two children on the graph, and allowing multiple
edges from $\widehat{\Delta }$ to $\Delta _{1}$. We will call the
characteristic vectors where we have suppressed the numbers $r_{n}(\Delta
_{i})$ the reduced characteristic vectors, and we will call the resulting
graph the reduced transition graph.

\begin{example}
\label{ex:golden diagram} In \cite[Section 4]{F1} and \cite[Section 6]{F2},
Feng studied the Bernoulli convolution $\mu _{\varrho }$, with $\varrho
^{-1} $the golden ratio, and found that there were seven characteristic
vectors. Their normalized lengths and neighbourhood sets are given by:

\begin{itemize}
\item Characteristic vector 1: $(1, (0), 1)$

\item Characteristic vector 2: $(\varrho, (0), 1)$

\item Characteristic vectors 3a and 3b: $(1-\varrho, (0, \varrho), 1)$ and $%
(1-\varrho, (0, \varrho), 2)$

\item Characteristic vector 4: $(\varrho, (1-\varrho), 1)$

\item Characteristic vector 5: $(\varrho, (0, 1-\varrho),1)$

\item Characteristic vector 6: $(2 \varrho-1, (1-\varrho), 1)$
\end{itemize}

Notice there are only six reduced characteristic vectors; we label the two
characteristic vectors with identical length and neighbourhood set as 3a and
3b. In \cite{F1} these were labelled as 3 and 7. The directed graphs in
Figure \ref{fig:golden} show the parent/children relationships. The term
``essential class'', referred to in the Figure, is defined in Section 4.

\label{Ex:BC} 
\begin{figure}[tbp]
\begin{center}
\subfigure[Transition
graph]{\includegraphics[scale=0.65]{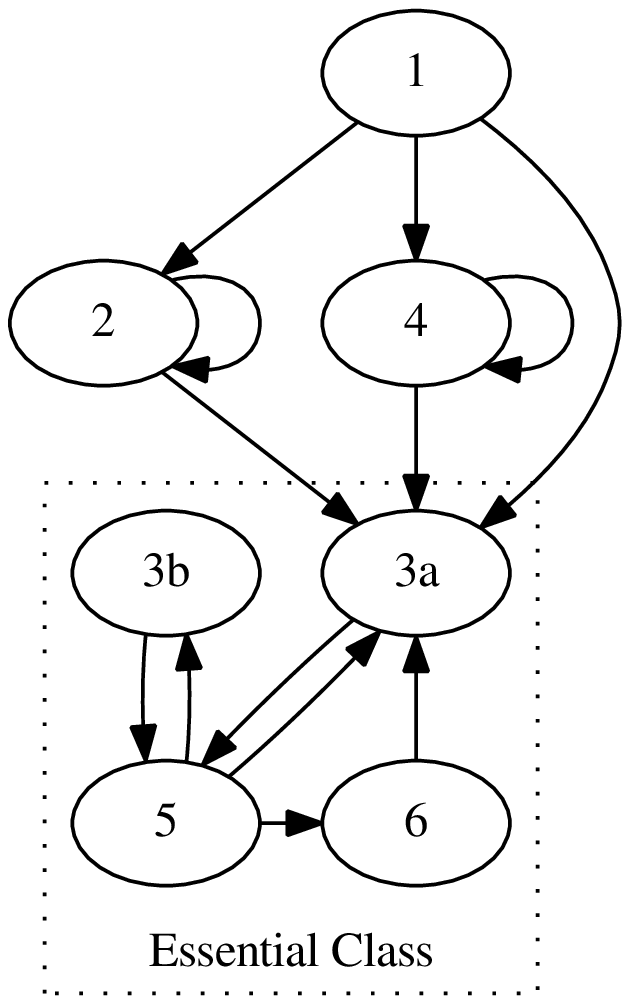}} 
\subfigure[Reduced
Transition graph]{\includegraphics[scale=0.65]{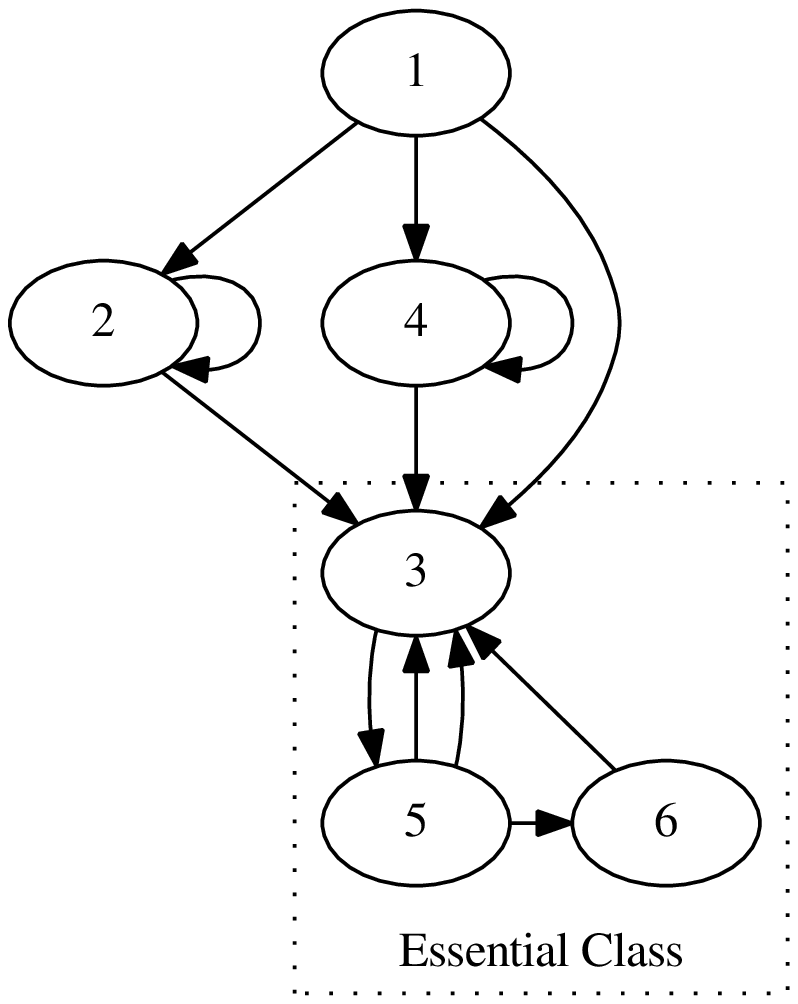}}
\end{center}
\caption{Transition graph for $\protect\mu_\protect\varrho$ with $\protect%
\varrho^{-1}$ the golden ratio}
\label{fig:golden}
\end{figure}
\end{example}

By an \textit{admissible path, }$\eta ,$\textit{\ of length }$L(\eta )=n,$
we will mean an ordered $n$-tuple, $\eta =(\gamma _{j})_{j=1}^{n},$ where $%
\gamma _{j}\in \Omega $ for all $j$ and the characteristic vector, $\gamma
_{j},$ is the parent of $\gamma _{j+1}$.

By the \textit{symbolic expression} of $\Delta \in \mathcal{F}_{n}$ we mean
an admissible path of length $n+1$, denoted 
\begin{equation*}
\lbrack \Delta ]=(\mathcal{C}_{0}(\Delta _{0}),\dots,\mathcal{C}_{n}(\Delta
_{n})),
\end{equation*}
where $\Delta =\Delta _{n}$ and for each $j<n$, $\Delta _{j}\in \mathcal{F}%
_{j}$. Here $\Delta _{0}$ is $[0,1]$. Feng \cite{F3} proved that the
symbolic expression uniquely determines $\Delta $.

For $x\in \lbrack 0,1]$, the symbolic representation for $x$, denoted $[x],$
will mean the sequence $(\mathcal{C}_{0}(\Delta _{0}),\mathcal{C}_{1}(\Delta
_{1}),\dots)$ of characteristic vectors where $x\in \Delta _{n}$ for all $n$
and $\Delta _{j}\in \mathcal{F}_{j}$ is the parent of $\Delta _{j+1}$. We
note that unless $x$ is an endpoint of a net interval (in which case there
are two representations of $x$), $[x]$ is unique. The notation $[x|N]$ will
mean the admissible path consisting of the first $N$ characteristic vectors
of $[x]$.

We frequently write $\Delta _{n}(x)$ for the net interval of level $n$
containing $x$. Thus $[x]$ is the sequence where the first $n+1$ terms gives
the symbolic representation of $\Delta _{n}(x)$ for each $n$.

\section{Transition matrices and local Dimensions}

\label{sec:trans}

\subsection{Local dimensions\label{sec:locdim} of measures of finite type}

\begin{definition}
Given a probability measure $\mu $, by the \textbf{upper local dimension} of 
$\mu $\textit{\ }at $x\in $ supp$\mu $, we mean the number 
\begin{equation*}
\overline{\dim }\mu (x)=\limsup_{r\rightarrow 0^{+}}\frac{\log \mu
([x-r,x+r])}{\log r}.
\end{equation*}%
Replacing the $\limsup $ by $\liminf $ gives the \textbf{lower local
dimension}, denoted \underline{$dim$}$\mu (x)$. If the limit exists, we call
the number the \textbf{local dimension} of $\mu $ at $x$ and denote this by $%
\dim \mu (x)$.
\end{definition}

Multifractal analysis refers to the study of the local dimensions of
measures.

For a $\varrho $-equicontractive measure $\mu $, it is easy to check that 
\begin{equation}
\dim \mu (x)=\lim_{n\rightarrow \infty }\frac{\log \mu ([x-\varrho
^{n},x+\varrho ^{n}])}{n\log \varrho }\text{ for }x\in \text{supp}\mu ,
\label{locdim}
\end{equation}
and similarly for the upper and lower local dimensions.

Our first several lemmas will enable us to show that we can replace the
interval $[x-\varrho ^{n},x+\varrho ^{n}]$ by $\Delta _{n}(x)$.

\begin{lemma}
Suppose $\mu $ satisfies the standard technical assumptions. Let $\Delta
=[a,b]\in \mathcal{F}_{n},$ with $V_{n}(\Delta )=(a_{1},\dots ,a_{k})$. Then 
\begin{equation*}
\mu (\Delta )=\sum_{i=1}^{k}\mu \lbrack a_{i},a_{i}+\ell _{n}(\Delta )]\sum 
_{\substack{ \sigma \in \mathcal{A}^{n}  \\ \varrho ^{-n}(a-S_{\sigma
}(0))=a_{i}}}p_{\sigma }.
\end{equation*}
\end{lemma}

\begin{proof}
This argument can basically be found in Feng \cite{F3}, but we give the
details here for completeness. Iterating (\ref{ss}) $n$ times gives
\begin{equation*}
\mu (\Delta )=\sum_{\sigma \in \mathcal{A}^{n}}p_{\sigma }\mu (S_{\sigma
}^{-1}(\Delta )).
\end{equation*}
Since $\mu $ is a non-atomic measure supported on $[0,1]$, we have
\begin{align*}
\mu (\Delta )
    &=\sum_{\substack{\sigma \in \mathcal{A}^{n}\\ S_{\sigma }(0,1)\cap\Delta \neq \emptyset}}
      p_{\sigma }\mu (S_{\sigma }^{-1}(\Delta ))
\end{align*}

Now, $S_\sigma (0,1) \cap \Delta \neq \emptyset$ implies that $\Delta \subset S_\sigma [0,1]$, hence by  definition of the neighbourhood set $\varrho^{-n}(a-S_{\sigma}(0))=a_{i}$ for some $i$. Thus $S_\sigma(0) = a-a_i \varrho^n$, so  $S_\sigma[0,1] = [a - a_i \varrho^n, a - a_i + \varrho^n]$. This implies that
\begin{align*}
\mu (\Delta )
    &=\sum_{i=1}^{k} \sum_{\substack{\sigma \in \mathcal{A}^{n}\\ \varrho ^{-n}(a-S_{\sigma}(0))=a_{i}}}
      p_{\sigma }\mu (S_{\sigma }^{-1}(\Delta ))
\end{align*}
We observe that $S_\sigma(x) = \varrho^n x + S_\sigma(0)$, and hence
\begin{align*}
S_\sigma([a_i, a_i + \ell_n(\Delta)]) 
    & = [a_i \varrho^n + a - a_i \varrho^n, a_i \varrho^n + a - a_i \varrho^n + \ell_n(\Delta)\varrho^n] \\
    &  = [a, a+\ell_n(\Delta) \varrho^n] = [a,b] = \Delta.
\end{align*}
Hence
\begin{align*}
\mu (\Delta ) 
    &=\sum_{i=1}^{k}  \sum_{\substack{\sigma \in \mathcal{A}^{n}\\ \varrho ^{-n}(a-S_{\sigma}(0))=a_{i}}}
      p_{\sigma }\mu [a_i, a_i + \ell_n(\Delta)] \\
    &=\sum_{i=1}^{k} \mu [a_i, a_i + \ell_n(\Delta)]
      \sum_{\substack{\sigma \in \mathcal{A}^{n}\\ \varrho^{-n}(a-S_{\sigma}(0))=a_{i}}} p_{\sigma }
\end{align*}
as claimed.
\end{proof}

\begin{notation}
For $i=1,2,\dots,$card$(V_{n}(\Delta ))=K$, put 
\begin{equation*}
P_{n}^{i}(\Delta )=p_{0}^{-n}\sum_{\sigma \in \mathcal{A}^{n}:\varrho
^{-n}(a-S_{\sigma }(0))=a_{i}}p_{\sigma }\text{ }
\end{equation*}%
and 
\begin{equation*}
P_{n}(\Delta )=\sum_{i=1}^{K}P_{n}^{i}(\Delta ).
\end{equation*}
\end{notation}

Here we have chosen to normalize by multiplying by $p_{0}^{-n}$. This is
done so that the minimal non-zero entry in the transition matrices (defined
in the next subsection) is at least $1$.

\begin{corollary}
\label{cor}There is a constant $c>0$ such that for any $n$ and any $\Delta
\in \mathcal{F}_{n},$ 
\begin{equation*}
cp_{0}^{n}P_{n}(\Delta )\leq \mu (\Delta )\leq p_{0}^{n}P_{n}(\Delta ).
\end{equation*}
\end{corollary}

\begin{proof}
The upper bound is clear from the lemma. For the lower bound we note that
each $\mu \left[ a_{i},a_{i}+\ell _{n}(\Delta )\right] >0$ as the support of 
$\mu $ is the full interval $[0,1]$. The finite type condition ensures there
are only finitely many choices for $[a_{i},a_{i}+\ell _{n}(\Delta )]$.
\end{proof}

\begin{lemma}
\label{Lem1} Suppose $\mu $ satisfies the standard technical assumptions.
There are constants $c_{1},c_{2}>0$ such that if $\Delta _{1}$, $\Delta _{2}$
are two adjacent net intervals of level $n$, then 
\begin{equation*}
c_{1}\frac{1}{n}P_{n}(\Delta _{2})\leq P_{n}(\Delta _{1})\leq
c_{2}nP_{n}(\Delta _{2}).
\end{equation*}
\end{lemma}

\begin{proof}
The proof is similar to that of \cite[Lemma 2.11]{F1} and proceeds by
induction on $n$. The base case holds as there are only finitely many
choices for $P_{1}(\Delta _{j})$ when $\Delta _{j}\in \mathcal{F}_{1}$. Now
assume the result for level $n-1$ and we will verify it holds for level $n$.

If $\Delta _{1},\Delta _{2}$ have the same parent $\widehat{\Delta }$, the
result follows easily from the observation that 
\begin{equation*}
P_{n-1}(\widehat{\Delta })\leq P_{n}(\Delta _{j})\leq mp_{0}^{-1}\max
p_{j}P_{n-1}(\widehat{\Delta })\text{.}
\end{equation*}

Otherwise, $\Delta _{1}$ and $\Delta _{2}$ are children of adjacent net
intervals of level $n-1$, $\widehat{\Delta _{1}},\widehat{\Delta _{2}}$
respectively, and we can suppose $\Delta _{1}$ is to the left of $\Delta
_{2} $. As in \cite{F1}, put
\begin{align*}
D_{1} &=\{\sigma \in \mathcal{A}^{n-1}:\widehat{\Delta _{1}}\subseteq
S_{\sigma }[0,1]\text{ and they share the same right endpoint\thinspace }\}
\\
D_{2} &=\{\sigma \in \mathcal{A}^{n-1}:\widehat{\Delta _{2}}\text{ }
\subseteq S_{\sigma }[0,1]\text{ and they share the same left
endpoint\thinspace }\} \\
E_{j} &=\{\sigma \in \mathcal{A}^{n-1}\diagdown D_{j}:\widehat{\Delta _{j}}
\subseteq S_{\sigma }[0,1]\}\text{, }j=1,2.
\end{align*}
The definitions ensure that $E_{1}=E_{2}$, 
\begin{equation*}
p_{0}^{n-1}P_{n-1}(\widehat{\Delta _{j}})=\sum_{\sigma \in D_{j}}p_{\sigma
}+\sum_{\sigma \in E_{j}}p_{\sigma }
\end{equation*}
and 
\begin{align*}
p_{0}^{n}P_{n}(\Delta _{1})
  &\leq \sum_{\sigma \in D_{1}}p_{\sigma}p_{m}+\sum_{\sigma \in E_{1}}p_{\sigma }\sum_{j=0}^{m}p_{j} \\
  &\leq p_m \sum_{\sigma \in D_{1}}p_{\sigma}+p_m \sum_{\sigma \in E_{1}}p_{\sigma } +
        \sum_{\sigma \in E_{2}}p_{\sigma } + \sum_{\sigma \in D_{2}}p_{\sigma} \\
  &\leq p_m p_0^{n-1} P_{n-1}(\widehat{\Delta_1}) + p_0^{n-1} P_{n-1}(\widehat{\Delta_2}) \\
\end{align*}
Applying the induction assumption gives
\begin{align*}
p_{0}P_{n}(\Delta _{1})
  &\leq p_m P_{n-1}(\widehat{\Delta_1}) + P_{n-1}(\widehat{\Delta_2}) \\
  &\leq p_m c_2(n-1) P_{n-1}(\widehat{\Delta_2}) + P_{n-1}(\widehat{\Delta_2}) \\
  &\leq (p_m c_2(n-1)+1) P_{n-1}(\widehat{\Delta_2})
\end{align*}
Taking $c_2 \geq 1/p_0 = 1/p_m \geq 1$ gives
\begin{align*}
p_{0}P_{n}(\Delta _{1})
  & \leq ((n-1)+1) P_{n-1} (\widehat{\Delta_2}) \\
  & \leq n P_{n-1}(\widehat{\Delta_2}) \\
  & \leq c_2 n P_{n}(\widehat{\Delta_2})
\end{align*}

The other inequality is similar.
\end{proof}

Note that in the proof the assumption that $\{p_{j}\}$ were regular
probabilities was important.

The following is immediate from the two previous results.

\begin{corollary}
There are constants $C_{1},C_{2}$ such that if $\Delta _{1},\Delta _{2}$ are
adjacent net intervals of level $n$, then 
\begin{equation*}
C_{1}\frac{1}{n}\mu (\Delta _{2})\leq \mu (\Delta _{1})\leq C_{2}n\mu
(\Delta _{2}).
\end{equation*}
\end{corollary}

Together these results yield the following useful approach to computing
local dimensions.

\begin{corollary}
\label{cor1}Suppose $\mu $ satisfies the standard technical assumptions. Let 
$x\in $supp$\mu $ and $\Delta _{n}(x)$ denote a net interval of level $n$
containing $x$. Then 
\begin{align}
\overline{\dim }\mu (x)& =\limsup_{n\rightarrow \infty }\frac{\log \mu
(\Delta _{n}(x))}{n\log \varrho }  \notag \\
& =\frac{\log p_{0}}{\log \varrho }+\limsup_{n\rightarrow \infty }\frac{\log
P_{n}(\Delta _{n}(x))}{n\log \varrho }.
\end{align}%
A similar statement holds for the (lower) local dimensions.
\end{corollary}

\begin{proof}
Since any net interval of level $n$ has length at most $\varrho ^{n}$, the
interval $[x-\varrho ^{n},x+\varrho ^{n}]$ contains $\Delta _{n}(x)$. The
finite type property ensures it is contained in a union of a bounded number
of $n$'th level net intervals, say $\bigcup\limits_{j=1}^{N}\Delta
_{n}(x_{j})$, where $\Delta _{n}(x_{j})$ is adjacent to $\Delta
_{n}(x_{j+1}) $ and for a suitable index $j$, $x_{j}=x$. Thus for constants 
$c$, $C$ (independent of the choice of $n$ and $x$),
\begin{align*}
cp_{0}^{n}P_{n}(\Delta _{n}(x)) &\leq \mu (\Delta _{n}(x))\leq \mu
([x-\varrho ^{n},x+\varrho ^{n}])\leq \sum_{j=1}^{N}\mu (\Delta _{n}(x_{j}))
\\
&\leq \sum_{j=1}^{N}p_{0}^{n}P_{n}(\Delta _{n}(x_{j}))\leq
Np_{0}^{n}C^{N}n^{N}P_{n}(\Delta _{n}(x)).
\end{align*}
Thus the limiting behaviour of the three expressions
\begin{equation*}
\frac{\log \mu ([x-\varrho ^{n},x+\varrho ^{n}])}{n\log \varrho },\text{ }
\frac{\log \mu (\Delta _{n}(x))}{n\log \varrho }
\end{equation*}
and 
\begin{equation*}
\frac{\log p_{0}}{\log \varrho }+\frac{\log P_{n}(\Delta _{n}(x))}{n\log
\varrho }
\end{equation*}
coincide.
\end{proof}

\begin{remark}
If $\Delta =[0,b]$, then $V_{n}(\Delta )=\{0\}$ and hence $P_{n}(\Delta )=1$%
. Consequently, \thinspace $\dim \mu (0)=\log p_{0}/\log \varrho $. More
generally, since $V_{n}(\Delta )$ is never empty and $p_{0}$ is the minimal
probability, it follows that $P_{n}(\Delta _{n}(x))\geq 1$ for all $n$ and $%
x $. Consequently, 
\begin{equation*}
\dim \mu (x)\leq \dim \mu (0)\text{ for all }x\in \text{supp}\mu .
\end{equation*}
\end{remark}

\subsection{Transition matrices}

The results of the previous subsection show that for studying the local
dimensions of these measures it will be helpful to make accurate estimates
of $P_{n}(\Delta _{n}(x))$. Towards this, slightly modifying \cite{F3} we
define \textit{primitive transition matrices}, $T(\mathcal{C}_{n-1}( 
\widehat{\Delta }),$ $\mathcal{C}_{n}(\Delta )),$ for a net interval $\Delta
=[a,b]$ of level $n$ and parent $\widehat{\Delta }=[c,d]$ as follows:

\begin{notation}
Suppose $V_{n}(\Delta )=(a_{1},\dots,a_{K})$ and $V_{n-1}(\widehat{\Delta }%
)=(c_{1},\dots,c_{J})$. For $j=1,\dots,J$ and $k=1,\dots,K,$ we set 
\begin{equation*}
T_{jk}:=\left( T(\mathcal{C}_{n-1}(\widehat{\Delta }),\mathcal{C}_{n}(\Delta
))\right) _{jk}=p_{0}^{-1}p_{\ell }\text{ }
\end{equation*}%
if $\ell \in \mathcal{A}$ and there exists $\sigma \in \mathcal{A}^{n-1}$
with $S_{\sigma }(0)=c-\varrho ^{n-1}c_{j}$ and $S_{\sigma \ell
}(0)=a-\varrho ^{n}a_{k}$. This is equivalent to saying 
\begin{equation*}
T_{jk}=p_{0}^{-1}p_{\ell }\text{ if }c-\varrho ^{n-1}c_{j}+\varrho
^{n-1}d_{\ell }=a-\varrho ^{n}a_{k}.
\end{equation*}%
We set $\left( T(\mathcal{C}_{n-1}(\widehat{\Delta }),\mathcal{C}_{n}(\Delta
))\right) _{jk}=0$ otherwise.
\end{notation}

As $\Omega $ is finite for a measure of finite type, there is an upper bound
on the size of these matrices. The entries are non-negative and all non-zero
entries are at least one. Each column has at least one non-zero entry
because $a_{k}\in V_{n}(\Delta )$ if and only if there is some $c_{j}\in
V_{n-1}(\widehat{\Delta })$ which ``contributes'' to it, in the sense
defined above.

It is also important to note that the standard technical assumption that supp%
$\mu =[0,1]$ guarantees that given $\sigma \in \mathcal{A}^{n-1}$ such that $%
S_{\sigma }(0)=c-\varrho ^{n-1}c_{j}$, there exists $\ell \in \mathcal{A}$
such that $S_{\sigma \ell }(0)=\varrho ^{n-1}d_{\ell }+S_{\sigma }(0)\in
(a-\varrho ^{n},a]$. This means that each row of the matrix $\left( T(%
\mathcal{C}_{n-1}(\widehat{\Delta }),\mathcal{C}_{n}(\Delta ))\right) $ also
has a non-zero entry.

For $K=$ card$(V_{n}(\Delta ))$, put 
\begin{equation*}
Q_{n}(\Delta )=(P_{n}^{1}(\Delta ),\dots,P_{n}^{K}(\Delta )).
\end{equation*}%
The same reasoning as in \cite[Theorem 3.3]{F3} shows that 
\begin{equation*}
Q_{n}(\Delta )=Q_{n-1}(\widehat{\Delta })\left( T(\mathcal{C}_{n-1}(\widehat{%
\Delta }),\mathcal{C}_{n}(\Delta ))\right) .
\end{equation*}%
Thus if $[\Delta ]=(\gamma _{0},\dots,\gamma _{n})$ (that is, $\gamma _{j}=%
\mathcal{C}_{j}(\Delta _{j})$ and $\gamma _{0}=\mathcal{C}_{0}([0,1])$),
then since $Q_{0}[0,1]=\left[ 1\right] $ we have 
\begin{equation*}
P_{n}(\Delta )=\left\Vert Q_{n}(\Delta )\right\Vert =\left\Vert T(\mathcal{%
\gamma }_{0},\gamma _{1})\cdot \cdot \cdot T(\gamma _{n-1},\gamma
_{n})\right\Vert ,
\end{equation*}%
where by the norm of matrix $M=(M_{jk})$ we mean 
\begin{equation*}
\left\Vert M\right\Vert :=\sum_{jk}\left\vert M_{jk}\right\vert .
\end{equation*}

Given an admissible path $\eta =(\gamma _{1},\dots,\gamma _{n})$, we write 
\begin{equation*}
T(\eta )=T(\gamma _{1},\dots,\gamma _{n})=T(\mathcal{\gamma }_{1},\gamma
_{2})\cdot \cdot \cdot T(\gamma _{n-1},\gamma _{n})
\end{equation*}
and refer to such a product as a \textit{transition matrix}.

With this notation, the results of the previous subsection can be stated as

\begin{corollary}
\label{corlocdim} Suppose $\mu $ satisfies the standard technical
assumptions. If $x\in $supp$\mu $, then 
\begin{equation}
\overline{\dim }\mu (x)=\frac{\log p_{0}}{\log \varrho }+\limsup_{n%
\rightarrow \infty }\frac{\log \left\Vert T([x|n])\right\Vert }{n\log
\varrho }  \label{F2}
\end{equation}%
and similarly for the (lower) local dimension.
\end{corollary}

\begin{example}
\label{ex:golden transition} Again, consider the Bernoulli convolution, $\mu
_{\varrho }$, with $\varrho ^{-1}$ the golden ratio. Feng \cite{F1} showed
that $0$ has symbolic representation $[1,2,2,\dots]$ and that $%
T(1,2,2,\dots)=\left[ 1\right] $. Applying Corollary \ref{corlocdim} gives
another proof that $\dim \mu _{\varrho }(0)=\log 2/\left\vert \log \varrho
\right\vert $.
\end{example}

Next, we give a useful simple lemma.

\begin{lemma}
\label{compare} Suppose $\mu $ satisfies the standard technical assumptions.
Let $A$ and $B$ be transition matrices. Then $\left\Vert B\right\Vert \leq
\left\Vert AB\right\Vert$ and $\left\Vert B\right\Vert \leq \left\Vert
BA\right\Vert$.
\end{lemma}

\begin{proof}
We have 
\begin{equation*}
\left\Vert AB\right\Vert =\sum_{ij}\left( \sum_{k}A_{ik}B_{kj}\right)
=\sum_{jk}\left( \sum_{i}A_{ik}\right) B_{kj}.
\end{equation*}%
Since all the entries of each of the matrices is nonnegative and each column
of $A$ has an entry $\geq 1$, it follows that $\left\Vert AB\right\Vert \geq
\left\Vert B\right\Vert .$

The argument for the other inequality is similar, noting that each row of $A$ has an 
    entry $\geq 1$ as a consequence of the standard technical assumptions.
\end{proof}

An important consequence of this result is that the local dimension of $\mu $
at $x$ depends only on the tail of the symbolic representation of $x$.

\begin{corollary}
Suppose $[x]=(\gamma _{0},\gamma _{1},\dots)$. For any $N$, 
\begin{equation*}
\overline{\dim }\mu (x)=\frac{\log p_{0}}{\log \varrho }+\limsup_{n%
\rightarrow \infty }\frac{\log \left\Vert T(\gamma _{N},\gamma
_{N+1},\dots,\gamma _{N+n})\right\Vert }{n\log \varrho }
\end{equation*}%
and similarly for the (lower) local dimension.

If $[y]=(\gamma _{0},\beta _{1},\dots,\beta _{n},\gamma _{N},\gamma
_{N+1},\dots) $, then the (upper or lower) local dimensions of $\mu $ at $x$
and $y$ agree.
\end{corollary}

\begin{proof}
This holds since 
\begin{align*}
\left\Vert T(\gamma _{N},\gamma _{N+1},\dots,\gamma _{N+n})\right\Vert &\leq
\left\Vert T(\gamma _{0},\dots,\gamma _{N})T(\gamma _{N},\gamma
_{N+1},\dots,\gamma _{N+n})\right\Vert \\
&=\left\Vert T(\gamma _{0},\dots,\gamma _{N},\gamma _{N+1},\dots,\gamma
_{N+n})\right\Vert \\
&\leq \left\Vert T(\gamma _{0},\dots\gamma _{N})\right\Vert \left\Vert
T(\gamma _{N},\gamma _{N+1},\dots,\gamma _{N+n})\right\Vert .
\end{align*}
\end{proof}

\begin{notation}
By $sp(M)$ we mean the spectral radius of the square matrix $M,$ the largest
eigenvalue of $M$ in absolute value. Recall that 
\begin{equation*}
sp(M)=\lim_{n}\left\Vert M^{n}\right\Vert ^{1/n}.
\end{equation*}
\end{notation}

We will call a matrix $M$ \textit{positive }if all its entries are strictly
positive and write $M>0$. We record here some elementary facts about
positive matrices that will be useful later.

\begin{lemma}
\label{normsprod} Suppose $\mu $ satisfies the standard technical
assumptions. Assume $A,B,C$ are transition matrices and $B$ is positive.

\begin{enumerate}
\item Then $\left\Vert ABC\right\Vert \geq \left\Vert A\right\Vert
\left\Vert C\right\Vert $. \label{it2:1}

\item There is a constant $C_{1}=C_{1}(B)$ such that if $AB$ is a square
matrix, then 
\begin{equation*}
\left\Vert AB\right\Vert \leq C_{1}sp(AB)\text{.}
\end{equation*}
\label{it2:2}

\item Suppose $B$ is a square matrix. There is a constant $C_{2}=C_{2}(B)$
such that 
\begin{equation*}
sp(B^{n})\leq \left\Vert B^{n}\right\Vert \leq C_{2}sp(B^{n})\text{ for all }%
n\text{.}
\end{equation*}
\label{it2:3}
\end{enumerate}
\end{lemma}

\begin{proof}
To see (\ref{it2:1}), let $B=(B_{jk})$. As all $B_{jk}\geq 1,$ a simple calculation gives 
\begin{equation*}
\left\Vert ABC\right\Vert =\sum_{ijkl}A_{ij}B_{jk}C_{kl}\geq
\sum_{ijkl}A_{ij}C_{kl}=\left\Vert A\right\Vert \left\Vert C\right\Vert .
\end{equation*}

For (\ref{it2:2}), assume $A=(A_{jk})$ is a $q\times p$ matrix. Let $b=\max B_{jk}$.
As the entries of $A$
are non-negative and the entries of $B$ are at least 1,  it is easy to see that 
\begin{align*}
\left\Vert AB\right\Vert &=\sum_{j,k}(AB)_{jk}=\sum_{j,k,l}A_{jl}B_{lk}\leq
b\sum_{j,k=1}^{q}\sum_{l=1}^{p}A_{jl} \leq bq\sum_{j=1}^{q}\sum_{l=1}^{p}A_{jl}B_{lj} \\
&= bq\sum_{j=1}^{q}(AB)_{jj}\leq bq^2  sp(AB),
\end{align*}%
with the final inequality holding because the sum of the diagonal entries of 
$AB$ is the sum of the eigenvalues of $AB$, counted by multiplicity.

For (\ref{it2:3}), let $B=PJP^{-1}$ be the Jordan decomposition of $B$ and let $\beta
=sp(B)$. By the Perron-Frobenius theory, $\beta $ is a simple root of the
characteristic polynomial of $B$ and all other eigenvalues of $B$ are
strictly less than $\beta $ in modulus. Since all entries of $B$ are at
least $1$, it can be easily seen that $\beta >1$. As $\left\Vert
B^{n}\right\Vert \leq \left\Vert P\right\Vert \left\Vert J^{n}\right\Vert
\left\Vert P^{-1}\right\Vert $, it is enough to prove $\left\Vert
J^{n}\right\Vert \leq C_{1}\beta ^{n}$ where $C_{1}$ depends on $B$, but not 
$n$.

Assume $B$ is of size $d\times d$. Since the Jordan block for $\beta $ is $%
1\times 1$, all entries of $J^{n}$, other than the $(1,1)$ entry which is $%
\beta ^{n}$, are either $0$ or of the form $\binom{n}{j}\alpha ^{n-j}$ where 
$j\leq \max (d,n)$ and $\alpha $ is an eigenvalue of $B$ with $%
\left\vert \alpha \right\vert <\beta $. Thus 
\begin{equation*}
\left\Vert J^{n}\right\Vert \leq \beta ^{n}+d^{2}n^{d}\beta _{0}^{n}
\end{equation*}%
where $\beta _{0}<\beta $ is the maximum of $1$ and the modulus of the
eigenvalues of $B$ other than $\beta $. As $(\beta /\beta _{0})^{n}\geq
d^{2}n^{d}$ for all $n$ sufficiently large depending on $d,\beta ,\beta _{0}$%
, it follows that for some constant $C_{2},$ depending on $\beta ,\beta
_{0},d,$ (and hence depending only on $B$) we have $\left\Vert
J^{n}\right\Vert \leq C_{2}\beta ^{n}$ for all $n$. This proves the right
hand inequality.

The left hand inequality is obvious.
\end{proof}

\section{Loop classes and periodic points}

\label{sec:loop}

\subsection{Essential and Loop classes}

Feng in \cite{F2} also introduced the notion of an essential class for
measures of finite type. Here we introduce the more general definition of a
loop class, of which the essential class is a special case.

\begin{definition}
(i) A non-empty subset $\Omega ^{\prime }\subseteq \Omega $ is called a%
\textbf{\ loop class} if whenever $\alpha ,\beta \in \Omega ^{\prime }$,
then there are characteristic vectors $\gamma _{j}$, $j=1,\dots,n$, such
that $\alpha =\gamma _{1}$, $\beta =\gamma _{n}$ and $(\gamma
_{1},\dots,\gamma _{n}) $ is an admissible path with all $\gamma _{j}\in
\Omega ^{\prime }$.

(ii) A loop class $\Omega ^{\prime }\subseteq \Omega $ is called an\textbf{\
essential class} if, in addition, whenever $\alpha \in \Omega ^{\prime }$
and $\beta \in \Omega $ is a child of $\alpha $, then $\beta \in \Omega
^{\prime }$.

(iii) We call a loop class \textbf{maximal }if it is not properly contained
in any other loop class.
\end{definition}

The finite type property ensures that every element in the support of $\mu$
is contained in a loop class. Clearly every loop class is contained in a
unique maximal loop class. Feng in \cite[Lemma 6.4]{F2}, proved there is
always precisely one essential class. Of course, the essential class is a
maximal loop class.

\begin{notation}
We will denote the essential class by $\Omega _{0}$ and here-after speak of
``the'' essential class.
\end{notation}

\begin{definition}
If $[x]=(\gamma _{0},\gamma _{1},\gamma _{2},\dots)$ with $\gamma _{j}\in
\Omega _{0}$ for all large $j$, we will say that\textit{\ }$x$ is an \textbf{%
essential point} (or is \textbf{in the essential class}) and call $x$ a 
\textbf{non-essential point} otherwise. The phrase, $x$ is \textbf{in the
loop class }$\Omega ^{\prime },$ will have a similar meaning. An admissible
path will be said to be in a given loop class if all its members are in that
class.
\end{definition}

\begin{remark}
Note that if $[x]=(\gamma _{j})$ is non-essential, then none of the
characteristic vectors $\gamma _{j}$ belong to $\Omega _{0}$. A
non-essential point necessarily has its tail in some loop class external to
the essential class.
\end{remark}

We remark that the essential class is dense in $[0,1]$. This is because the
uniqueness of the essential class ensures that every net interval contains a
net subinterval of higher level whose characteristic vector is in the
essential class. In fact, we show next that the set of essential points has
full Lebesgue measure in $[0,1]$.

\begin{proposition}
Suppose $\mu $ satisfies the standard technical assumptions. Then the set of
non-essential points is a subset of a closed set of Lebesgue measure $0$.
\end{proposition}

\begin{proof}
As we already observed, every net interval contains a descendent net
subinterval whose characteristic vector is in the essential class. (We will
abuse notation slightly and call such a net subinterval ``essential''.) The
finite type property ensures we can find such a net subinterval in a bounded
number of generations and that there exists some $\lambda >0$ such that the
proportion of the length of the net subinterval to the length of the
original interval is $\geq \lambda $.

We now exhibit a Cantor-like construction. We begin with $[0,1]$. Consider
the first level at which there is a net subinterval that is essential.
Remove the interiors of all the net subintervals of this level that are
essential. The resulting closed subset of $[0,1]$ is a finite union of
closed intervals, say $\mathcal{U}_{1}$, whose lengths total at most $%
1-\lambda $. We repeat the process of removing the interiors of the
essential net subintervals at the next level at which there are essential
net subintervals in each of the intervals of $\mathcal{U}_{1}$. The
resulting closed subset now has length at most $(1-\lambda )^{2}$.

After repeating this procedure $k$ times one can see that the non-essential
points are contained in a finite union of closed intervals, $C_{k},$ whose
is total length is at most $(1-\lambda )^{k}$. It follows that the
non-essential points are contained in the closed set $\bigcap\limits_{k=1}^{%
\infty }C_{k},$ and this set has measure $0$.
\end{proof}

\begin{remark}
It is worth remarking that the construction above may leave some essential
points within the Cantor-like construction. As the resulting set is measure $%
0$, the smaller set of just the non-essential points will also be measure $0 
$.
\end{remark}

\begin{example}
\label{ex:golden essential} From Figure \ref{fig:golden} one can see that
the Bernoulli convolution $\mu _{\varrho }$, with $\varrho ^{-1}$ the golden
ratio, has seven distinct loop classes: $\{3a,3b,5,6\}$, $\{3a,3b,5\}$, $%
\{3a,5,6\}$, $\{3a,5\}$, $\{3b,5\}$, $\{2\}$, and $\{4\}$. Of these, $\{2\}$%
, $[4\}$ and the essential class, $\{3a,3b,5,6\}$ (with 3 reduced elements)
are maximal. The two loop classes external to the essential class, $\{2\}$
and $\{4\},$ are associated to the two endpoints, $0$ and $1$. These are the
only two non-essential points, in other words, the set of essential points
is $(0,1)$.
\end{example}

\begin{definition}
We will say the loop class $\Omega ^{\prime }$ is of \textbf{positive type }%
if there is an admissible path $\eta $ in $\Omega ^{\prime }$ such that $%
T(\eta )$ is a positive matrix.
\end{definition}

\begin{remark}
We note that as there is a non-zero entry in each row and column of each
primitive transition matrix, then any loop class $\Omega ^{\prime }$ of
positive type has the property that for every $\delta ,\delta ^{\prime }\in
\Omega ^{\prime }$ there is an admissible path $\eta =(\delta ,\delta
_{1},\dots ,\delta _{r},\delta ^{\prime })$ in $\Omega ^{\prime }$ such that 
$T(\eta )$ is positive.
\end{remark}

\begin{example}
\label{ex:golden positive type} The loop classes $\{2\}$ and $\{4\}$ of
Example \ref{ex:golden essential} are of positive type since (as shown in 
\cite[Section 4]{F1}) $T(2,2)=T(4,4)=[1]$. As $T(5,6,3a)=\left[ 
\begin{array}{cc}
1 & 1 \\ 
1 & 1%
\end{array}%
\right] ,$ the essential class, $\{3a,3b,5,6\},$ is also of positive type.
However, the loop class, $\{3a,5\},$ is not of positive type since all
transition matrices from this loop class are of the form $\left[ 
\begin{array}{cc}
1 & 0 \\ 
n & 1%
\end{array}%
\right] $ for some $n$.
\end{example}

\begin{remark}
In\ Example \ref{ex:$x^3-x-1$)} we give a Bernoulli convolution where a
maximal loop class is not of positive type.
\end{remark}

However, the essential class is always of positive type under our standard
technical assumptions.

\begin{proposition}
\label{prop:pos} Suppose $\mu $ satisfies the standard technical
assumptions. Then the essential class is of positive type.
\end{proposition}

\begin{proof}
Fix $\delta \in \Omega _{0}$. We recall that Feng in \cite[Lemma 6.4]{F2}
showed that given any positive integer $k\leq $ cardinality$(V(\delta ))$,
there is an admissible path $\gamma $ in the essential class, going from $%
\delta $ to $\delta ,$ such that all entries of the $k$'th row of $T(\gamma
) $ are non-zero.

Choose such a path, $\eta _{1},$ with all the entries of row $1$ of $T(\eta
_{1})$ non-zero. Row $2$ of $T(\eta _{1})$ has a non-zero entry in some
column, say $k,$ and, of course, the $(1,k)$ entry is also non-zero. Choose
an essential, admissible path $\eta _{2}$ from $\delta $ to $\delta ,$ with
all entries of row $k$ non-zero. Matrix multiplication shows that all
entries of both rows 1 and 2 of $T(\eta _{1})T(\eta _{2})=T(\eta _{1},\eta
_{2})$ are non-zero. By repeated application of this reasoning we can
construct an admissible path $\eta $ in $\Omega _{0},$ from $\delta $ to $%
\delta ,$ such that all entries of $T(\eta )$ are strictly positive.
\end{proof}

\subsection{Periodic points and their local dimensions}

\begin{definition}
We call $x\in $supp$\mu $ a \textbf{periodic point} if $x$ has symbolic
representation%
\begin{equation*}
\lbrack x]=(\gamma _{0},\dots,\gamma _{s},\theta ^{-},\theta ^{-},\dots)
\end{equation*}%
where $\theta $ is an admissible cycle (a non-trivial path with the same
first and last letter) and $\theta ^{-}$ is the path with the last letter of 
$\theta $ deleted. We refer to $\theta $ as a period of $x$.

We will say that periodic $x$ is\textit{\ }\textbf{positive} if the square
transition matrix $T(\theta )$ is positive.
\end{definition}

It is worth noting that $\theta $ and $\theta ^{-}$ are not uniquely
defined. For example, the path $(1,3a,5,3a,5,3a,\dots)$ from Example \ref%
{ex:golden essential} can be decomposed as $(\gamma _{0},\gamma _{1})=(1,3a)$
and $\theta ^{-}=(5,3a)$, or as $(\gamma _{0})=(1)$ and $\theta ^{-}=(3a,5)$%
, or as $(\gamma _{0})=(1)$ and $\theta ^{-}=(3a,5,3a,5)$, etc. In what
follows, it will not make any significant difference as to which choice is
made.

Observe that 
\begin{equation*}
T([x|N])=T(\gamma _{0},\dots ,\gamma _{s},\delta _{1})(T(\theta ))^{n}
\end{equation*}%
when $N=s+1+L(\theta ^{-})n$ and $\delta _{1}$ is the first letter of $%
\theta $. (Recall $L(\theta ^{-})$ is the length of path $\theta ^{-}$.) A
periodic $x$ is in the essential class if it has a period that is a path in
the essential class.

There is a simple formula for the local dimensions at periodic points.

\begin{proposition}
\label{periodic} Suppose $\mu $ satisfies the standard technical
assumptions. Let $x$ be the periodic point with period $\theta $. Then%
\begin{equation*}
\dim \mu (x)=\frac{\log p_{0}}{\log \varrho }+\frac{\log (sp(T(\theta )))}{%
L(\theta ^{-})\log \varrho }.
\end{equation*}
\end{proposition}

\begin{proof}
Suppose $[x]=(\gamma _{0},\dots,\gamma _{s},\theta ^{-},\theta ^{-},\dots)$ and $%
\theta ^{-}=(\delta _{1},\dots,\delta _{L(\theta ^{-})})$. Let $S=T(\gamma
_{0},\dots,\gamma _{s},\delta _{1})$. According to Lemma \ref{compare}, there
is a constant $c>0$ such that 
\begin{equation*}
\left\Vert (T(\theta ))^{n}\right\Vert \leq \left\Vert S(T(\theta
))^{n}\right\Vert \leq \left\Vert S(T(\theta ))^{n}T(\delta _{1},\dots,\delta
_{r})\right\Vert \leq c \left\Vert (T(\theta ))^{n}\right\Vert
\end{equation*}%
for any $r<L(\theta ^{-})$. Thus Corollary \ref{corlocdim} implies that 
\begin{equation*}
\dim \mu (x)=\frac{\log p_{0}}{\log \varrho }+\lim_{n\rightarrow \infty }%
\frac{\log \left\Vert (T(\theta ))^{n}\right\Vert }{nL(\theta ^{-})\log
\varrho }.
\end{equation*}%
Since $\frac{1}{n}\log \left\Vert (T(\theta ))^{n}\right\Vert \rightarrow $ $%
\log (sp(T(\theta ))),$ the result follows.
\end{proof}

\section{Local dimensions of positive loop classes}

\label{sec:main}

In \cite{F1} and \cite{F2}, Feng showed that the set of local dimensions for
the Bernoulli convolution $\mu _{\varrho },$ with $\varrho ^{-1}$ the golden
ratio, was an interval and determined its endpoints. This is not true, in
general, for measures of finite type. For instance, it is known that the set
of local dimensions of the $m$-fold convolution of uniform Cantor measures
with contraction factor $1/R,$ for integer $R\leq m,$ is the union of an
interval and an isolated point (see \cite{BHM, HL, Sh}).

Feng also proved that the set of attainable local dimensions of $\mu
_{\varrho }$ was the closure of the set of local dimensions at periodic
points. In this section we will prove that if a loop class, $\Omega ^{\prime
},$ is of positive type, then the set of local dimensions at points in $%
\Omega ^{\prime }$ is a closed interval. Moreover, this interval is the
closure of the set of local dimensions at the periodic points in $\Omega
^{\prime }$. These statements hold, in particular, for the essential class, $%
\Omega _{0},$ since the essential class is a positive loop class according
to Proposition \ref{prop:pos}.

In Sections \ref{sec:comp}, \ref{sec:Cantor} and \ref{sec:Pisot} we give
examples to illustrate that the set of local dimensions attained at points
of a loop class external to the essential class can overlap, or may be
disjoint from the local dimensions of the essential class.

\subsection{Local dimensions at periodic points are dense}

\begin{theorem}
\label{periodicdense} Suppose $\mu $ satisfies the standard technical
assumptions. Assume that $\Omega ^{\prime }$ is a loop class of positive
type. The set of local dimensions of $\mu $ at positive, periodic points in
the loop class $\Omega ^{\prime }$ is dense in the set of all local
dimensions at points in $\Omega ^{\prime }$. It is also dense in the set of
all upper (or lower) local dimensions at points in $\Omega ^{\prime }$.
\end{theorem}

\begin{remark}
We remark that in this theorem (and other results of this section) the
assumption that the loop class is of positive type may not be necessary for
particular IFS and particular loops.
\end{remark}

\begin{proof}
We will prove denseness in the set of lower local dimensions at points in
the loop class $\Omega ^{\prime }$. The arguments for the (upper) local
dimensions are the same.

Fix $x$ in $\Omega ^{\prime },$ say $[x]=(\gamma _{0},\gamma _{1},\gamma
_{2},\dots)$ with $\gamma _{k}\in \Omega ^{\prime }$ for all $k\geq M$. Choose
a subsequence $(n_{k})$ such that 
\begin{equation*}
\underline{\dim }\mu (x)=\frac{\log p_{0}}{\log \varrho }+\lim_{k\rightarrow
\infty }\frac{\log \left\Vert T(\gamma _{M},\dots,\gamma _{n_{k}})\right\Vert 
}{n_{k}\log \varrho }.
\end{equation*}%
As $\Omega ^{\prime }$ is finite, by passing to a further subsequence if
necessary (not renamed) we can assume all $\gamma _{n_{k}}=\delta \in \Omega
^{\prime }$. Put 
\begin{equation*}
D=\lim_{k\rightarrow \infty }\frac{\log \left\Vert T(\gamma _{M},\dots,\gamma
_{n_{k}})\right\Vert }{n_{k}\log \varrho }.
\end{equation*}

Let $\eta $ be an admissible path in $\Omega ^{\prime }$ going from $\delta $
to $\gamma _{M},$ such that $T(\eta )>0$. (We remark that it is important
that this transition matrix is independent of the choice of $k$.) Of course,
then $\theta _{k}=(\gamma _{M},\dots,\gamma _{n_{k}-1})\eta $ is an admissible
cycle in $\Omega ^{\prime }$.

As $T(\eta )$ is a positive matrix, Lemmas \ref{compare} and \ref{normsprod}%
(2) imply that there are constants, $K_{k}$, bounded above and bounded
below away from $0$ such that 
\begin{equation*}
\left\Vert T(\gamma _{M},\dots,\gamma _{n_{k}})\right\Vert =K_{k}sp(T(\gamma
_{M},\dots,\gamma _{n_{k}})T(\eta ))=K_{k}sp(T(\theta _{k})).
\end{equation*}

Consider the local dimension at the periodic point $y_{k}$ in $\Omega
^{\prime }$ given by 
\begin{equation*}
\lbrack y_{k}]=(\gamma _{0},\dots,\gamma _{M-1},\theta _{k}^{-},\theta
_{k}^{-},\dots).
\end{equation*}%
Since $T(\theta _{k})=T(\gamma _{M},\dots,\gamma _{n_{k}})T(\eta )$ is
positive, $y_{k}$ is a positive point. Furthermore, we have%
\begin{align*}
\dim \mu (y_{k}) &=\frac{\log p_{0}}{\log \varrho }+\frac{\log (sp(T(\theta
_{k}))}{L(\theta _{k}^{-})\log \varrho } \\
&=\frac{\log p_{0}}{\log \varrho }-\frac{\log K_{k}}{L(\theta _{k}^{-})\log
\varrho }+\frac{\log \left\Vert T(\gamma _{M},\dots,\gamma
_{n_{k}})\right\Vert }{L(\theta _{k}^{-})\log \varrho }.
\end{align*}%
Fix $\varepsilon >0$. The choice of $n_{k}$ and boundedness of $K_{k}$
ensures we can choose $k$ large enough so that
\begin{equation*}
\left\vert \frac{\log \left\Vert T(\gamma _{M},\dots,\gamma
_{n_{k}})\right\Vert }{n_{k}\log \varrho }-D\right\vert <\varepsilon ,
\end{equation*}%
\begin{equation*}
\left\vert \frac{\log K_{k}}{L(\theta _{k}^{-})\log \varrho }\right\vert
<\varepsilon \text{ and }\left\vert \frac{n_{k}}{L(\theta _{k}^{-})}%
-1\right\vert <\varepsilon \text{.}
\end{equation*}%
It is now easy to see that%
\begin{equation*}
\left\vert \dim \mu (x)-\dim \mu (y_{k})\right\vert <(D+2)\varepsilon .
\end{equation*}
\end{proof}

Since the essential class is of positive type, we immediately deduce the
following important fact about these measures of finite type.

\begin{corollary}
The set of local dimensions at essential periodic points is dense in the set
of local dimensions at all essential points.
\end{corollary}

Here is another needed elementary fact whose proof is an exercise for the
reader.

\begin{lemma}
\label{lem:cyclicshift} Let $\theta =(\delta _{1},\delta _{2},\dots,\delta
_{L},\delta _{1})$ be a cycle and let $\theta ^{\ast }=(\delta
_{k},\dots,\delta _{L},\delta _{1},\dots,\delta _{k})$ be any cyclic shift
of $\theta $. Then 
\begin{equation*}
sp(T(\theta ))=sp(T(\theta ^{\ast }))
\end{equation*}
\end{lemma}

\begin{theorem}
\label{LimitPeriodic} Suppose $\mu $ satisfies the standard technical
assumptions. Assume that the loop class, $\Omega ^{\prime },$ is of positive
type and suppose $(x_{n})$ is a sequence of positive, periodic points in $%
\Omega ^{\prime }$. Then there is some $x$ in $\Omega ^{\prime }$ such that 
\begin{equation*}
\overline{\dim }\mu (x)=\limsup_{n}\dim \mu (x_{n})\ \ \mathrm{and}\ \ 
\underline{\dim }\mu (x)=\liminf_{n}\dim \mu (x_{n}).
\end{equation*}
\end{theorem}

\begin{proof}
Assume $x_{n}$ has period $\theta _{n}$ in $\Omega ^{\prime }$ and that $%
T(\theta _{n})>0$. We put 
\begin{equation}
S:=\limsup_{n}\frac{\log (sp(T(\theta _{n}))}{L(\theta _{n}^{-})}\ \ \mathrm{%
and}\ \ I:=\liminf_{n}\frac{\log (sp(T(\theta _{n}))}{L(\theta _{n}^{-})},
\label{P6}
\end{equation}%
so 
\begin{equation*}
\limsup_{n}\dim \mu (x_{n})=\frac{\log p_{0}}{\log \varrho }+\frac{\log I}{%
\log \varrho }.
\end{equation*}%
and 
\begin{equation*}
\liminf_{n}\dim \mu (x_{n})=\frac{\log p_{0}}{\log \varrho }+\frac{\log S}{%
\log \varrho }.
\end{equation*}%
By passing to a subsequence, not renamed, we can assume all 
\begin{equation*}
\left\vert \frac{\log (sp(T(\theta _{2n})))}{L(\theta _{2n}^{-})}%
-S\right\vert <\varepsilon _{2n}\ \ \mathrm{and}\ \ \left\vert \frac{\log
(sp(T(\theta _{2n+1})))}{L(\theta _{2n+1}^{-})}-I\right\vert <\varepsilon
_{2n+1},
\end{equation*}%
where $\varepsilon _{n}$ is a decreasing sequence tending to 0. Further, we
can assume that all the even labelled paths, $\theta _{n}^{-}$, have the
same first letter, say $\alpha _{S}$, and the same last letter, $\beta _{S}$%
. Similarly, we can assume all the odd labelled paths have a common first
letter $\alpha _{I}$ and common last letter $\beta _{I}$.

Since $\Omega ^{\prime }$ is of positive type, we can certainly choose two
admissible paths in $\Omega ^{\prime }$, $\lambda _{SI}$ going from $\beta
_{S}$ to $\alpha _{I}$ and $\lambda _{IS}$ going from $\beta _{I}$ to $%
\alpha _{S}$, so that $T(\lambda _{SI})$ and $T(\lambda _{IS})$ are positive.

We want to inductively define a rapidly increasing subsequence $(k_{n})$
such that 
\begin{equation*}
\lbrack x]=(\eta ,\underbrace{\theta _{1}^{-},\dots,\theta _{1}^{-}}%
_{k_{1}},\lambda _{IS},\underbrace{\theta _{2}^{-},\dots,\theta _{2}^{-}}%
_{k_{2}},\lambda _{SI},\underbrace{\theta _{3}^{-},\dots,\theta _{3}^{-}}%
_{k_{3}},\dots)
\end{equation*}%
has the desired property, where $\eta $ is an admissible path beginning with 
$\mathcal{C}_{0}([0,1])$ and ending with the parent of $\alpha _{I}$.

Temporarily fix $n$. With abuse of notation we will write the truncated
product as 
\begin{equation*}
\lbrack x]_{n}=\eta \prod_{i=1}^{n}(\theta _{i}^{-})^{k_{i}}\lambda _{i}
\end{equation*}%
where $\lambda _{i}=\lambda _{SI}$ if $i$ is even, and $\lambda _{IS}$ if $i$
is odd.

Define $K_{n}$ as the maximal $K_{n}:=C_{2}(T(\theta _{n+1}^{\ast }))$ coming
from Lemma \ref{normsprod}(3) taken over all cyclic shifts of $\theta
_{n+1}^{\ast }$ of $\theta _{n+1}$. Let $\lambda _{n}^{\prime }$ be any
prefix of $\lambda _{n}$ and $\theta _{n+1}^{\prime }$ any prefix of $\theta
_{n+1}^{-}$. The elementary lemmas of the previous section imply that 
\begin{equation}
\lim_{k\rightarrow \infty }\left\vert \frac{\log (\left\Vert
T([x]_{n-1}(\theta _{n}^{-})^{k}\lambda _{n}^{\prime })\right\Vert )}{%
L([x]_{n-1}(\theta _{n}^{-})^{k}\lambda _{n}^{\prime })}-\frac{\log
(sp(T(\theta _{n})))}{L(\theta _{n}^{-})}\right\vert =0  \label{2-1}
\end{equation}%
and%
\begin{equation}
\lim_{k\rightarrow \infty }\left\vert \frac{\log (\left\Vert
T([x]_{n-1}(\theta _{n}^{-})^{k}\lambda _{n}\theta _{n+1}^{\prime
})\right\Vert )}{L([x]_{n-1}(\theta _{n}^{-})^{k}\lambda _{n}\theta
_{n+1}^{\prime })}-\frac{\log (sp(T(\theta _{n})))}{L(\theta _{n}^{-})}%
\right\vert =0.  \label{2-2}
\end{equation}%
Obviously, 
\begin{equation}
\lim_{k\rightarrow \infty }\frac{\log (K_{n})}{L([x]_{n-1}(\theta
_{n}^{-})^{k}\lambda _{n}\theta _{n+1}^{\prime })}=0  \label{2-3}
\end{equation}%
and%
\begin{equation}
\lim_{k\rightarrow \infty }\frac{L(\lambda _{n}\theta _{n+1})}{k}\cdot \frac{%
\log (sp(T(\theta _{n})))}{L(\theta _{n}^{-})}=0.  \label{2-4}
\end{equation}%
We pick $k_{n}$ such that the left hand side of these four limits are each
less than $\varepsilon _{n}$ for all choices of $\lambda _{n}^{\prime }$ and 
$\theta _{n+1}^{\prime }$.

Certainly this process defines an $x$ belonging to $\Omega ^{\prime }$. We
need to check that the upper and lower dimensions of $\mu $ at $x$ are
correct. By construction, the limiting behaviour (as $n\rightarrow \infty $)
of 
\begin{equation*}
\frac{\log (\left\Vert T([x]_{n-1}(\theta _{n}^{-})^{k_{n}}\lambda
_{n}^{\prime }))\right\Vert )}{L([x]_{n-1}(\theta _{n}^{-})^{k_{n}}\lambda
_{n}^{\prime }))}\ \ \mathrm{and}\ \ \frac{\log (\left\Vert
T([x]_{n-1}(\theta _{n}^{-})^{k_{n}}\lambda _{n}\theta _{n+1}^{\prime
})\right\Vert )}{L([x]_{n-1}(\theta _{n}^{-})^{k_{n}}\lambda _{n}\theta
_{n+1}^{\prime })}
\end{equation*}%
approach the values $I$ and $S$ in the infimum/supremum respectively. Hence it
remains to consider the case 
\begin{equation*}
\frac{\log (\left\Vert T([x]_{n-1}(\theta _{n}^{-})^{k_{n}}\lambda
_{n}(\theta _{n+1}^{-})^{p_{n+1}}\theta _{n+1}^{\prime })\right\Vert )}{%
L([x]_{n-1}(\theta _{n}^{-})^{k_{n}}\lambda _{n}(\theta
_{n+1}^{-})^{p_{n+1}}\theta _{n+1}^{\prime })}
\end{equation*}%
for $0<p_{n+1}<k_{n+1}$, and $\theta _{n+1}^{\prime }$ some prefix of $%
\theta _{n+1}$. This is equivalent to 
\begin{equation*}
\frac{\log (\left\Vert T([x]_{n-1}(\theta _{n}^{-})^{k_{n}}\lambda
_{n}\theta _{n+1}^{\prime }(\theta _{n+1}^{\ast -})^{p_{n+1}})\right\Vert )}{%
L([x]_{n-1}(\theta _{n}^{-})^{k_{n}}\lambda _{n}\theta _{n+1}^{\prime
}(\theta _{n+1}^{\ast -})^{p_{n+1}})}
\end{equation*}%
for some cyclic shift $\theta _{n+1}^{\ast }$ of $\theta _{n+1}$.

For ease of notation, let $L_{0}(n)=L_{0}:={L([x]_{n-1}(\theta
_{n}^{-})^{k_{n}}\lambda _{n}\theta _{n+1}^{\prime }(\theta _{n+1}^{\ast
-})^{p_{n+1}})}$. First, observe that:
\begin{align*}
E_{n} &:=\frac{\log (\left\Vert T([x]_{n-1}(\theta
_{n}^{-})^{k_{n}}\lambda _{n}\theta _{n+1}^{\prime }(\theta _{n+1}^{\ast
-})^{p_{n+1}})\right\Vert )}{L_{0}} \\ 
& \leq   \frac{\log (\left\Vert T([x]_{n-1}(\theta _{n}^{-})^{k_{n}}\lambda
_{n}\theta _{n+1}^{\prime }\right\Vert )}{L_{0}}+\frac{\log (\left\Vert
T(\theta _{n+1}^{\ast })^{p_{n+1}})\right\Vert )}{L_{0}} \\ 
& \leq   \frac{L([x]_{n-1}(\theta _{n}^{-})^{k_{n}}\lambda _{n}\theta
_{n+1}^{\prime })}{L_{0}}\cdot \frac{\log (sp(T(\theta _{n})))}{L(\theta
_{n}^{-})}+\frac{\log (K_{n}sp((T(\theta _{n+1}))^{p_{n}+1}))}{L_{0}}%
+\varepsilon _{n} \\ 
& \leq   \frac{L_{0}-L((\theta _{n+1}^{-})^{p_{n}+1})}{L_{0}}\cdot \frac{%
\log (sp(T(\theta _{n})))}{L(\theta _{n}^{-})}\\&\hspace{10pt} +\frac{L((\theta
_{n+1}^{-})^{p_{n}+1})}{L_{0}}\cdot \frac{\log (sp(T(\theta
_{n+1}))^{p_{n}+1})}{L((\theta _{n+1}^{-})^{p_{n}+1})}+\frac{\log K_{n}}{%
L_{0}}+\varepsilon _{n} \\ 
& \leq   \frac{L_{0}-L((\theta _{n+1}^{-})^{p_{n}+1})}{L_{0}}\cdot \frac{%
\log (sp(T(\theta _{n})))}{L(\theta _{n}^{-})}\\&\hspace{10pt} +\frac{L((\theta
_{n+1}^{-})^{p_{n}+1})}{L_{0}}\cdot \frac{\log (sp(T(\theta _{n+1})))}{%
L(\theta _{n+1}^{-})}+2\varepsilon _{n} \\ 
& =  (1-t_{n})\left( \frac{\log (sp(T(\theta _{n})))}{L(\theta _{n}^{-})}%
\right) +t_{n}\left( \frac{\log (sp(T(\theta _{n+1})))}{L(\theta _{n+1}^{-})}%
\right) +2\varepsilon _{n}.%
\end{align*}%
for 
\begin{equation*}
t_{n}=\frac{L((\theta _{n+1}^{-})^{p_{n}+1})}{L_{0}}.
\end{equation*}%
Here the second inequality comes from (\ref{2-2}), Lemma \ref{normsprod}%
(3) and Lemma \ref{lem:cyclicshift}, and the final inequality from (\ref{2-3}).

For the opposite inequality, we use the fact that $T(\lambda _{n}\theta
_{n+1}^{\prime })>0$ and Lemma \ref{normsprod}(1) for the first inequality
below, and (\ref{2-1}) and (\ref{2-4}) for the second and third inequalities
to obtain 
\begin{align*}
E_{n}&=\frac{\log (\left\Vert T([x]_{n-1}(\theta
_{n}^{-})^{k_{n}}\lambda _{n}\theta _{n+1}^{\prime }(\theta _{n+1}^{\ast
-})^{p_{n+1}})\right\Vert )}{L_{0}} \\ 
& \geq  \frac{\log (\left\Vert T([x]_{n-1}(\theta
_{n}^{-})^{k_{n}}\right\Vert )}{L_{0}}+\frac{\log (\left\Vert T((\theta
_{n+1}^{\ast })^{p_{n+1}})\right\Vert )}{L_{0}} \\ 
& \geq  \frac{L([x]_{n-1}(\theta _{n}^{-})^{k_{n}})}{L_{0}}\cdot \frac{\log
(sp(T(\theta _{n})))}{L(\theta _{n}^{-})}\\
& \hspace{10pt} +\frac{L((\theta
_{n+1}^{-})^{p_{n}+1})}{L_{0}}\cdot \frac{\log (sp(T(\theta
_{n+1}))^{p_{n}+1})}{L((\theta _{n+1}^{-})^{p_{n}+1})}-\varepsilon _{n} \\ 
& \geq  (1-t_{n})\left( \frac{\log (sp(T(\theta _{n})))}{L(\theta _{n}^{-})}%
\right) +t_{n}\left( \frac{\log (sp(T(\theta _{n+1})))}{L(\theta _{n+1}^{-})}%
\right) -2\varepsilon _{n}%
\end{align*}%
Together these estimates prove $E_{n}$ lies within $2\varepsilon _{n}$ of
the same convex combination of %\begin{equation*}
$\frac{\log (sp(T(\theta _{n})))}{L(\theta _{n}^{-})}$ %\text{ and }
and $\frac{\log (sp(T(\theta _{n+1})))}{L(\theta _{n+1}^{-})}$, 
%\end{equation*}%
and hence $\lim \sup \,E_{n}$ and $\lim \inf E_{n}$ both belong to the
interval $[I,S]$. That completes the proof.
\end{proof}

Combining Theorems \ref{periodicdense} and \ref{LimitPeriodic}, it follows
that these measures satisfy the following.

\begin{corollary}
The set of local dimensions at essential points coincides with the set of
lower (or upper) local dimensions at essential points. Moreover, 
\begin{equation*}
\{\dim \mu (x):x\text{ essential }\}=closure\{\dim \mu (x):x\text{
essential, positive periodic}\}.
\end{equation*}
\end{corollary}

\subsection{Set of local dimensions at points in a positive loop class is an
interval}

The final result of this section will be to show that the set of local
dimensions at points in a loop class of positive type is a closed interval.
Of course, in particular, this applies to the essential class.

\begin{theorem}
\label{Interval} Suppose $\mu $ satisfies the standard technical
assumptions. Further, suppose that the loop class $\Omega ^{\prime }$ is of
positive type. Assume $y$ and $z$ are periodic, positive points in $\Omega
^{\prime }$. Then the set of local dimensions of $\mu $ contains the closed
interval with endpoints $\dim \mu (y)$ and $\dim \mu (z)$.
\end{theorem}

\begin{proof}
Suppose $y$ has period $\varphi $ and $z$ has period $\theta $ where $%
A=T(\varphi )$ and $B=T(\theta )$ are positive matrices with spectral radii $%
\alpha $ and $\beta $, respectively. With this notation, 
\begin{equation*}
\dim \mu (y)=\frac{\log p_{0}}{\log \varrho }+\frac{\log \alpha }{L(\varphi
^{-})\log \varrho }\text{ and }\dim \mu (z)=\frac{\log p_{0}}{\log \varrho }+%
\frac{\log \beta }{L(\theta ^{-})\log \varrho }.
\end{equation*}%
Let $0<t<1$. We want to prove that there exists a $x$ such that 
\begin{equation*}
   \dim \mu (x) = t\dim \mu (y)+(1-t)\dim \mu (z).
\end{equation*}

Appealing to Theorem \ref{LimitPeriodic}, we see it will be enough to show
that there is some sequence of periodic, positive points $x_{k},$ in the
loop class $\Omega ^{\prime },$ such that 
\begin{equation}
\lim_{k}\dim \mu (x_{k})=\frac{\log p_{0}}{\log \varrho }+\frac{t\log \alpha 
}{L(\varphi ^{-})\log \varrho }+\frac{(1-t)\log \beta }{L(\theta ^{-})\log
\varrho }.  \label{key}
\end{equation}

To do this, we start by choosing admissible paths in $\Omega ^{\prime }$, $%
\eta _{1}$ joining the last letter of $\theta $ to the first letter of $%
\varphi $ and $\eta _{2}$ doing the opposite, such that $T(\eta _{j})>0$.
Then for any positive integers, $n,m$, $B^{m}T(\eta _{1})A^{n}T(\eta _{2})$
is a square transition matrix.

Select two sequences of integers, $(n_{k})_{k=1}^{\infty }$, $%
(m_{k})_{k=1}^{\infty },$ tending to infinity, with 
\begin{equation*}
\frac{L(\varphi ^{-})n_{k}}{L(\theta ^{-})m_{k}+L(\varphi ^{-})n_{k}}%
\rightarrow t.
\end{equation*}%
We will prove that the periodic points with period $\theta _{k}$ satisfying%
\begin{equation*}
T(\theta _{k})=B^{m_{k}}T(\eta _{1})A^{n_{k}}T(\eta _{2})
\end{equation*}
work.

As $T(\eta _{j})>0$, Lemmas \ref{compare} and \ref{normsprod} combine to
imply that for all $k,$%
\begin{equation*}
sp\left( B^{m_{k}}T(\eta _{1})A^{n_{k}}T(\eta _{2})\right) \leq \left\Vert
T(\eta _{1})\right\Vert \left\Vert T(\eta _{2})\right\Vert \left\Vert
B^{m_{k}}\right\Vert \left\Vert A^{n_{k}}\right\Vert \leq C(A,B,\eta
_{1},\eta _{2})\alpha ^{n_{k}}\beta ^{m_{k}},
\end{equation*}%
and%
\begin{align*}
sp\left( B^{m_{k}}T(\eta _{1})A^{n_{k}}T(\eta _{2})\right) &\geq C(\eta
_{2})\left\Vert B^{m_{k}}T(\eta _{1})A^{n_{k}}T(\eta _{2})\right\Vert \\
&\geq C(\eta _{2})\left\Vert B^{m_{k}}\right\Vert \left\Vert
A^{n_{k}}\right\Vert \geq C(\eta _{2})\alpha ^{n_{k}}\beta ^{m_{k}}.
\end{align*}

Thus 
\begin{align*}
\lim_{k}\frac{\log (sp(B^{m_{k}}T(\eta _{1})A^{n_{k}}T(\eta _{2})))}{%
L(\theta ^{-})m_{k}+L(\varphi ^{-})n_{k}} &=\lim_{k}\frac{n_{k}\log \alpha
+m_{k}\log \beta }{L(\theta ^{-})m_{k}+L(\varphi ^{-})n_{k}} \\
&=\frac{t\log \alpha }{L(\varphi ^{-})}+\frac{(1-t)\log \beta }{L(\theta
^{-})}.
\end{align*}%
It follows that (\ref{key}) is satisfied.
\end{proof}

Combining the three theorems we deduce that the set of local dimensions at
the essential points of such measures is a closed interval whose endpoints
are given by the infimum and the supremum of the local dimensions at
essential, periodic points.

\begin{corollary}
\label{IntEnd}Let $I=\inf \{\dim \mu (x):x$ essential, positive periodic$\}$
and $S=\sup \{\dim \mu (x):x$ essential and positive periodic$\}.$ Then%
\begin{equation*}
\{\dim \mu (x):x\text{ essential}\}=[I,S]\text{.}
\end{equation*}
\end{corollary}

Another immediate corollary is that if the set of essential points is $(0,1)$%
, then the measure admits at most one isolated point.

\begin{corollary}
If the set of essential points of $\mu $ is $(0,1)$, then the set of local
dimensions of $\mu $ consists of a closed interval together with $\dim \mu
(0)=\dim \mu (1)$.
\end{corollary}

\begin{example}
\label{ex:golden local dimension} As observed in Example \ref{ex:golden
essential}, this is the situation for the Bernoulli convolution $\mu
_{\varrho }$, with $\varrho ^{-1}$ the golden ratio. Further, since $%
sp(T(5,3a,5))=1=sp(T(2,2))=sp(T(4,4)),$ and $\{3a,5\}$ is in the essential
class, it follows that $\dim \mu (0)$ coincides with the local dimension at
an essential point. Consequently, the set of local dimensions of $\mu
_{\varrho }$ is equal to the closed interval $[I,S]$ consisting of the local
dimensions at the essential points of $\mu _{\varrho }$. In \cite{F1, Hu} it
is shown that $I=\log 2/|\log \rho |$ and $S=I+1/2$.
\end{example}

More generally, we have the following.

\begin{corollary}
If every maximal loop class is of positive type, then the set of local
dimensions of $\mu $ is a finite union of closed intervals.
\end{corollary}

These intervals can be disjoint. Indeed, in Example \ref{ex:isolated points}
we construct a measure whose set of local dimensions consists of two points
(degenerate intervals).

\section{Algorithm}

\label{sec:comp}

\subsection{Algorithm for finding characteristic vectors, the essential set
and transition matrices}

Given a family of contractions, $S_{j}(x)=\varrho x+d_{j}$, and
probabilities $p_{j}$, we have implemented an algorithm to find $\Omega $
and all of the associated transition matrices. 
%The description in Section \ref{sec:term} for $\Omega $, although accurate, is not the most efficient
%for computing $\Omega $ and the associated transition matrices. 
We use a modified version of \cite{NW} to perform these calculations. We
give an overview here, by means of a worked example.

For this we consider the six contractions, the maps $S_{j}(x)=\frac{1}{3}%
x+d_{j}$ for $d_{j}=2j/15$, for $j=0,1,\dots,5.$ The corresponding
self-similar set is the $5$-fold sum of the middle-third Cantor set rescaled
to $[0,1]$. Suppose the maps have normalized probabilities $%
1,p_{1},p_{2},p_{3},p_{4},1$.

One can see that $\Delta =[0,1]\in \mathcal{F}_{0}$ has characteristic
vector $(1,(0),1)$. We will call this characteristic vector 1. Here the
first $1$ is the normalized length of the interval. The sequence $(0)$ is $%
V_{0}(\Delta )=(a_{1},a_{2},\dots,a_{k})$ is the neighbourhood set. The last 
$1 $ is $r_0(\Delta)$. For the questions we are interested in, $r_{n}(\Delta
) $ is not needed, hence we will suppress it in the future, instead allowing
multiple edges between nodes in the graph. This gives a reduced
characteristic vector of $(1, (0))$.

Instead of considering all intervals in $\mathcal{F}_{1}$, we only consider
those that arise from new reduced characteristic vectors found to be in $%
\Omega $. So, initially, we look at those children of reduced characteristic
vector 1. In this case, these two things are the same, but we will see later
on that this is not always the case. We first subdivide $[0,1]$ by
considering the maps $S_{j}(0)-a_{i}$ and and $S_{j}(1)-a_{i}$ for all $%
a_{i}\in V(1)$ and $j=0,1,2,3,4,5$. This partitions the interval at the
points $\{k/15:k=0,2,4,5,6,7,8,9,10,11,13,15\}$.

Consider first $\Delta =[0,2/15]$ coming from this subdivision. We see that
its normalized length is $3\cdot 2/15=2/5$ and it has neighbourhood set $(0)$%
. We label $(2/5,(0))$ as the reduced characteristic vector 2. The $d_{j}$
that contributes to $0$ is $d_{0}$ and is associated to the normalized
probability of $1$. Hence the transition matrix is $T(1,2)=\left[ 
\begin{array}{c}
1%
\end{array}%
\right] $.

Now, consider $\Delta =[2/15,4/15]$. This again has normalized length $2/5$
and neighbourhood set $(0,2/5)$. We label $(2/5,(0,2/5))$ as the reduced
characteristic vector 3. The $0$ comes from $d_{1}$ and the $2/5$ from $%
d_{0} $, hence the transition matrix is $T(1,3)=\left[ 
\begin{array}{cc}
p_{1} & 1%
\end{array}%
\right] $.

We continue in this fashion, noting that for $\Delta =[2j/15,(2j+1)/15]$,
where $j=2,\dots,5$, the characteristic vector is $(1/5,(0,2/5,5/4))$, which
we label as 4. In Feng's notation, we would distinguish these as four
different characteristic vectors by use of their third component, $%
r_{1}(\Delta )$. We will not distinguish these, but will instead allow
multiple maps from the reduced characteristic vector 1 to the reduced
characteristic vector 4. These are 
\begin{align*}
T(1,4)& =\left[ 
\begin{array}{ccc}
p_{2} & p_{1} & 1%
\end{array}%
\right] \text{ }\mathrm{or}\ \left[ 
\begin{array}{ccc}
p_{3} & p_{2} & p_{1}%
\end{array}%
\right] \\
& \mathrm{or}\ \left[ 
\begin{array}{ccc}
p_{4} & p_{3} & p_{2}%
\end{array}%
\right] \mathrm{or}\ \left[ 
\begin{array}{ccc}
p_{5} & p_{4} & p_{3}%
\end{array}%
\right] .
\end{align*}

Continuing in this fashion, we can compute the finite set of reduced
characteristic vectors obtaining

\begin{itemize}
\item Reduced characteristic vector 1: $(1, (0))$

\item Reduced characteristic vector 2: $(2/5, (0))$

\item Reduced characteristic vector 3: $(2/5, (0, 2/5))$

\item Reduced characteristic vector 4: $(1/5, (0, 2/5, 4/5))$

\item Reduced characteristic vector 5: $(1/5, (1/5, 3/5))$

\item Reduced characteristic vector 6: $(2/5, (1/5, 3/5))$

\item Reduced characteristic vector 7: $(2/5,(3/5))$
\end{itemize}

The characteristic vectors 4 and 5 comprise the essential set.

\begin{figure}[tbp]
\begin{center}
\includegraphics[scale=0.7]{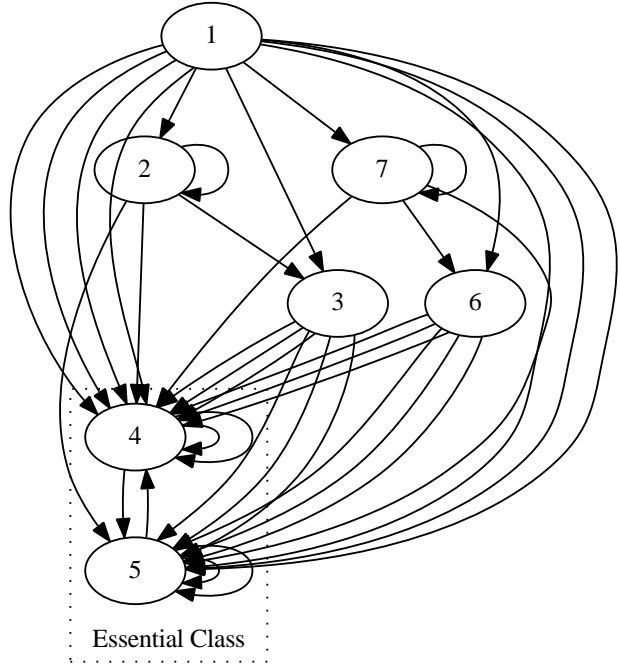}
\end{center}
\caption{Transition graph for $3x-1$, with $d_j \in \{2j/15: j=0,1,\dots,5\} 
$}
\label{fig:3x-1}
\end{figure}

For a complete list of transition matrices, see the supplementary information \cite{Homepage}.

\subsection{Algorithm for finding bounds on local dimensions\label{algdim}}

Given a loop class of positive type, there are two main ways we obtain good
estimates on the set of local dimensions associated to this loop class. The
first is to find explicit examples of periodic points within the loop class
and calculate their local dimensions by determining the spectral radius of
the transition matrix of the cycle. Applying Theorem \ref{periodicdense},
this will produce an interval contained in the set of local dimensions of
the loop class.

To find upper and lower bounds on the set of local dimensions we use the
family of pseudo-norms: 
\begin{align*}
\left\Vert T\right\Vert _{C,\min }& =\min \left\{ \sum_{j\in C}\left\vert
T_{jk}\right\vert :k\in C\right\} \\
\left\Vert T\right\Vert _{\min }& =\min \left\{ \sum_{k}\left\vert
T_{jk}\right\vert :j\right\} \\
\left\Vert T\right\Vert _{\max }& =\max \left\{ \sum_{k}\left\vert
T_{jk}\right\vert :j\right\} .
\end{align*}%
These are the sub-norm on indices $C$, the total sub-norm and the total
sup-norm respectively. The total sup-norm is actually a norm. Here $C$ is a
subset of the indices of the column vectors. Care must be take here that the
subset $C$ is valid for all matrices within the loop class one is
considering, as different transition matrices may have different dimensions.
For all matrices, $T_{1},T_{2}\geq 0$ we have 
\begin{align*}
\left\Vert T_{1}T_{2}\right\Vert _{C,\min }& \geq \left\Vert
T_{1}\right\Vert _{C,\min }\left\Vert T_{2}\right\Vert _{C,\min } \\
\left\Vert T_{1}T_{2}\right\Vert _{\min }& \geq \left\Vert T_{1}\right\Vert
_{\min }\left\Vert T_{2}\right\Vert _{\min } \\
\left\Vert T_{1}T_{2}\right\Vert _{\max }& \leq \left\Vert T_{1}\right\Vert
_{\max }\left\Vert T_{2}\right\Vert _{\max }
\end{align*}%
and 
\begin{equation*}
\left\Vert T\right\Vert _{\min },\left\Vert T\right\Vert _{C,\min }\leq
||T||\leq \eta \left\Vert T\right\Vert _{\max }
\end{equation*}%
where $\eta $ is the number of columns in $T$. We can thus obtain upper and
lower bounds for the local dimensions at points in the loop class by
calculating these pseudo-norms for all admissible products of primitive
transition matrices from the loop class up to some fixed length and using
formula (\ref{F2}).

We remark that pseudo-norms based on a subset of the indices of the row
vectors would serve, as well.

\begin{example}
\label{ex:isolated points} A measure whose local dimension is two isolated
points: We continue with the example described above, but now take the
normalized uniform weights $p_{i}=1$. Let $\mu $ denote the associated
self-similar measure.

The essential class, as mentioned above, is \{4, 5\} and its primitive
transition matrices are: 
\begin{align*}
T(4,4)& =\left[ 
\begin{array}{ccc}
1 & 0 & 0 \\ 
\noalign{\medskip}1 & 1 & 1 \\ 
\noalign{\medskip}0 & 1 & 1%
\end{array}%
\right] & T(4,5)& =\left[ 
\begin{array}{cc}
1 & 0 \\ 
\noalign{\medskip}1 & 1 \\ 
\noalign{\medskip}0 & 1%
\end{array}%
\right] \\
T(4,4)& =\left[ 
\begin{array}{ccc}
1 & 1 & 0 \\ 
\noalign{\medskip}1 & 1 & 1 \\ 
\noalign{\medskip}0 & 0 & 1%
\end{array}%
\right] & T(5,5)& =\left[ 
\begin{array}{cc}
1 & 1 \\ 
\noalign{\medskip}1 & 1%
\end{array}%
\right] \\
T(5,4)& =\left[ 
\begin{array}{ccc}
1 & 1 & 1 \\ 
\noalign{\medskip}1 & 1 & 1%
\end{array}%
\right] & T(5,5)& =\left[ 
\begin{array}{cc}
1 & 1 \\ 
\noalign{\medskip}1 & 1%
\end{array}%
\right] \\
& & &
\end{align*}

As the total column sub and sup-norms for all these matrices is $2$, it
follows that all $k$-fold products of primitive transition matrices in the
essential class will have their (usual) norm in the interval $[2^{k},3\cdot
2^{k}]$. It follows that the local dimension at any point in the essential
class is $(\log 6-\log 2)/\log 3=1$.

As the local dimension at $0$ and $1$ is equal to $\log 6/\log 3$, the set
of local dimensions of $\mu $ consists of the two distinct points, 
\begin{equation*}
\{\dim \mu (x):x\in \lbrack 0,1]\}=\{1\}\cup \left\{ 1+\log 2/\log 3\right\}
.
\end{equation*}
\end{example}

\section{Cantor-like measures}

\label{sec:Cantor}

In Example \ref{ex:shmerkin} we considered the IFS 
\begin{equation}
\left\{ S_{j}(x)=\frac{1}{R}x+\frac{j}{mR}(R-1):j=0,\dots,m\right\}
\label{Cantor}
\end{equation}%
for integers $R\geq 2,$ and the self-similar Cantor-like measures $\mu $
associated with this IFS and probabilities $\{p_{j}:j=0,\dots,m\}$. This
class of measures includes, for example, the $m$-fold convolution of the
uniform Cantor measure with contraction factor $1/R,$ rescaled to $[0,1]$.
The self-similar set is the full interval $[0,1]$ when $m\geq R-1$.

The results in \cite{BHM} and \cite{Sh}, which extend \cite{HL}, together
imply that if $m\geq R\geq 3$ and $p_{0},p_{m}<\min p_{j}$, $j\neq 0,m$,
then the set of local dimensions of $\mu $ consists of a closed interval and
one (or two) isolated points, the local dimensions at $0$ and $1$. These
facts can be recovered by our methods as well, when, in addition, $%
p_{0}=p_{m}$.

\begin{proposition}
Let $\mu $ be the self-similar measure associated with the IFS (\ref{Cantor}%
) with $m\geq R\geq 2$ and probabilities satisfying $p_{0}=p_{m}<\min_{j\neq
0,m}p_{j}$. Then the set of essential points is the open interval $(0,1)$
and $\dim \mu (0)$ is an isolated point in the set of all local dimensions.
Indeed, for $x\neq 0,1$, 
\begin{equation*}
\dim \mu (x)\leq \frac{\left\vert \log (\min \{2p_{0},p_{j}:j\neq
0,m\})\right\vert }{\log R}<\frac{\left\vert \log p_{0}\right\vert }{\log R}%
=\dim \mu (0).
\end{equation*}
\end{proposition}

\begin{proof}
We note, first, that the iterates of $0$ at level $n$ (meaning the real
values $S_{\sigma }(0)$ for $\sigma \in \mathcal{A}^{n}$) occur every $%
(R-1)/(R^{n}m)$, beginning at $0$ and ending at $1-R^{-n}$. Similarly, the
iterates of $1$ are spaced the same distance apart, but start at $R^{-n}$
and end at $1$. In the subinterval $[R^{-n},1-R^{-n}]$ they alternate,
except in the special case that $m\equiv 0$ mod$(R-1)$, when they coincide.
We will assume $m\not\equiv 0$ mod$(R-1)$ and leave the easier case for the reader.

Adjacent $n$'th level net intervals contained in $[R^{-n},1-R^{-n}]$ are of
the form $[a,b]$, $[b,c]$ where if $a$ is an iterate of $0$, then 
\begin{equation*}
a=\frac{j(R-1)}{R^{n}m},b=\frac{1}{R^{n}}+\frac{k(R-1)}{R^{n}m}\text{ and }c=%
\frac{(j+1)(R-1)}{R^{n}m}
\end{equation*}%
for suitable integers $j,k$. There is a similar formula when $a$ is an
iterate of $1$. If $m=L(R-1)+r$ for integer $L$ and $r\in \{1,\dots,R-2\}$,
then it is easily seen that $L+k=j$. Thus $n$'th level net intervals have
either (non-normalized) length 
\begin{equation*}
\frac{r}{R^{n}m}\text{ or }\frac{R-1-r}{R^{n}m},
\end{equation*}%
and, in either case, have length between $1/(R^{n}m)$ and $1/R^{n}$.

Let $x\in (0,1)$ and assume $n$ is chosen so large that $3R^{-n}<x<1-3R^{-n}$%
. Let $[A,B]$ denote the $n$'th level net interval containing $x$. The
choice of $n$ ensures that all the numbers of the form $i(R-1)/(R^{n}m)$ for
integer $i$ with 
\begin{equation*}
0<A-i(R-1)/(R^{n}m)<R^{-n}
\end{equation*}
are iterates of $0$ at level $n$ and hence comprise the neighbour set of $%
[A,B]$. Thus the neighbour set depends only upon whether $A$ is an iterate
of $0$ or an iterate of $1$. It follows that there are only two reduced
characteristic vectors associated with these net intervals.

Furthermore, the length of these net intervals is sufficiently large to
ensure that that the interior of each such interval contains at least one $%
n+1$-iterate of $0$ and one $n+1$-iterate of $1$. Consequently, each such $n$%
'th level net interval contains both styles of net intervals of level $n+1$
(and no other children). It follows that their characteristic vectors belong
to the essential class and that proves the set of essential points is $(0,1)$%
.

To prove the upper bound given in the statement of the proposition for $\dim
\mu (x)$ with $0<x<1,$ we will show that if $[x]=(\gamma _{0},\gamma
_{1},\dots)$ and $n$ is sufficiently large, then $\left\Vert T(\gamma
_{n-1},\gamma _{n})\right\Vert _{\min }\geq \min \{2,p_{j}p_{0}^{-1}:j\neq
0,m\}$ where $\left\Vert \cdot \right\Vert _{\min }$ is the total column
sub-norm introduced in the preceding section.

We can assume $3R^{-N}<x<1-3R^{-N}$ and take $n>N$. Suppose $\gamma _{n}=%
\mathcal{C(}\Delta _{n})$, $\Delta _{n}=[a,b]$, $\widehat{\Delta _{n}}=[c,d]$%
, and the neighbour sets are $V_{n}(\Delta _{n})=(a_{1},\dots,a_{K})$ and $%
V_{n-1}(\Delta _{n-1})=(c_{1},\dots,c_{J})$ respectively, where $c_{1}<\cdot
\cdot \cdot <c_{J}$.

Temporarily fix $k$. By definition, $T(\gamma _{n-1},\gamma _{n})_{jk}=p_{\ell }p_{0}^{-1}$ when
there is some $\sigma \in \mathcal{A}^{n}$ and $\ell \in \mathcal{A}$ such that $%
S_{\sigma }(0)=c-R^{-(n-1)}c_{j}$ and $S_{\sigma \ell }(0)=a-R^{-n}a_{k},$
and $T(\gamma _{n-1},\gamma _{n})_{jk}=$ $0$ otherwise.

 Of course, there is some choice of $j$ with a valid
choice of $\ell $. If $\ell \neq 0,m$ we are done. So assume this $\ell =0$.
That means $a-R^{-n}a_{k}=c-R^{-(n-1)}c_{j}$. The bounds on $x$ ensure that
for some $\tau \in \mathcal{A}^{n-1}$, 
\begin{equation*}
S_{\tau }(0)=S_{\sigma }(0)-(R-1)/(R^{n-1}m).
\end{equation*}%
Furthermore, 
\begin{equation*}
0<c-S_{\tau }(0)\leq a-S_{\sigma }(0)+\frac{R-1}{R^{n-1}m}<\frac{1}{R^{n}}+%
\frac{R-1}{R^{n-1}m}\leq \frac{1}{R^{n-1}}.
\end{equation*}
This implies there is some $i$ such that $c_{i}=R^{-(n-1)}(c-S_{\tau }(0))$;
indeed, $i=j+1$.

Another routine calculation shows $S_{\tau R}(0)=a-R^{-n}a_{k}$, thus $%
T(\gamma _{n-1},\gamma _{n})_{j+1,k}=p_{R}p_{0}^{-1}$. 

If, instead, $\ell =m$, similar arguments prove $T(\gamma _{n-1},\gamma
_{n})_{j-1,k}=p_{m-R}p_{0}^{-1}$. This completes the proof.
\end{proof}

\begin{remark}
We are not claiming that these bounds on $\dim \mu (x)$ for $x\neq 0,1$ are
sharp, merely illustrating that we can recover the property that $\dim \mu
(0)$ is an isolated point with our approach.
\end{remark}

In \cite[Thm. 6.1]{BHM}, the minimum and maximum local dimensions, other
than at $0,1$, were investigated for the case of the $m$-fold convolution of
the uniform Cantor measure for small $m$. We show that there exists $m$
outside of these ranges where these formula do not hold.

\begin{example}
\label{ex:m-fold convolution} Consider the $m$-fold convolution of the
uniform Cantor measure with contraction factor $1/R$, for integer $R\geq 3$.
It is shown in \cite{BHM} that if $R\leq m\leq 2R-2$, then 
\begin{equation}
\min_{x}\dim \mu (x)=\dim \mu (x_{\min })=\frac{m\log 2-\log \binom{m}{%
\left\lfloor \frac{m}{2}\right\rfloor }}{\log R}  \label{eq:min}
\end{equation}%
where 
\begin{equation*}
x_{\min }=\frac{1}{m}\left\lfloor \frac{m}{2}\right\rfloor .
\end{equation*}%
It was also shown there that for all even $m$, $\min_{x}\dim \mu (x)=\dim
\mu (x_{\min })$.

Using the computer, we have checked these formulas for $R=3$ and $3\leq
m\leq 10$. We have found that the right hand side of (\ref{eq:min}) is not $%
\dim \mu (x_{\min })$ and is not the minimal local dimension for $m=5,\dots
,10.$ Moreover, for $m=5,7,9$ the minimal local dimension does not occur at
the point $x_{\min }$. In fact, the predicted value of the minimum local
dimension is greater than the maximum local dimension other than for $x=0,1$%
. See Tables \ref{tab:min} and \ref{tab:max}. We note that in Tables \ref%
{tab:min} and \ref{tab:max}, when the formula is known to hold for
theoretical reasons, we put the precise value, otherwise we put a range,
coming from the techniques of Section \ref{sec:comp}.

In \cite{BHM}, there was also a formula given for the maximum local
dimension other than at $x=0,1$. Let $r=\left\lfloor (m-R)/2\right\rfloor $, 
$\ell _{2j}=r+1$ and $\ell _{2j+1}=m-r-R$. Put 
\begin{equation*}
x_{\max }=\frac{1}{m}\sum_{j=1}^{\infty }(R-1)R^{-j}\ell _{j}.
\end{equation*}%
It was shown that for $R\leq m\leq 2R-1$, $\max_{x\neq 0,1}\dim \mu (x)=\dim
\mu (x_{\max })$ and 
\begin{equation}
\dim \mu (x_{\max })=-\frac{\log ((p_{r+R+1}+p_{r}+\sqrt{%
(p_{r+R+1}-p_{r})^{2}+4p_{r+1}p_{r+R}})/2)}{\log R}.  \label{eq:max}
\end{equation}%
Here $p_{j}$ should be understood as $0$ if $j\notin \{0,1,\dots ,m\}$.

Using our methods, we can show that for $R=3$, $m=6$, these statements
continue to be true, that is, $\dim \mu (x_{\max })=$ $\max_{x\neq 0,1}\dim
\mu (x)$ and $\dim \mu (x_{\max })$ is the value specified by the right hand
side of formula (\ref{eq:max}). But for $R=3$ and $m=7,\dots,10$, the
predicted maximum local dimension from the right hand side of \eqref{eq:max}
is too big. We have not been able to determine if $x_{\max }$ is the point
where the maximum dimension occurs.

We refer the reader to the supplementary information \cite{Homepage} where this example is worked out in
full detail.

\begin{table}[tb]
\begin{center}
\begin{tabular}{llll}
$m$ & Left hand side          & Actual min & $dim \mu(x_{min} $) \\
    & of Formula \ref{eq:min} &            &                     \\ \hline
3   & .892790 & .892790 & .892790 \\ 
4   & .892790 & .892790 & .892790 \\ 
5   & 1.05875 & [.972382, .972639] & .984145 \\ 
6   & 1.05875 & .976628 & .976628 \\ 
7   & 1.18029 & [.993576, .993848] & .997991 \\ 
8   & 1.18029 & .995246 & .995246 \\ 
9   & 1.27620 & [.998541, .998658] & .999739 \\ 
10  & 1.27620 & .999022 & .999022%
\end{tabular}%
\end{center}
\caption{Minimal local dimensions}
\label{tab:min}
\end{table}

\begin{table}[tb]
\begin{center}
\begin{tabular}{llll}
$m$ & Left hand side          & Actual max & $\dim\mu(x_{max})$ \\ 
    & of Formula \ref{eq:max} &            &                    \\ \hline
3  & 1.13355 & 1.13354 & 1.13354 \\ 
4  & 1.05875 & 1.05874 & 1.05874 \\ 
5  & 1.02757 & 1.02757 & 1.02757 \\ 
6  & 1.01434 & 1.01434 & 1.01434 \\ 
7  & 1.01434 & [1.00605, 1.00736] & 1.00605 \\ 
8  & 1.01434 & [1.00342, 1.00346] & 1.00343 \\ 
9  & 1.02721 & [1.00133, 1.00171] & 1.00133 \\ 
10 & 1.03074 & [1.00079, 1.00082] & 1.00079 \\ 
&  &  & 
\end{tabular}%
\end{center}
\caption{Maximal local dimensions}
\label{tab:max}
\end{table}
\end{example}

\section{Bernoulli convolutions with contraction factors Pisot inverses}

\label{sec:Pisot}

\subsection{Bernoulli convolutions with contraction factors Pisot inverses}

In Table \ref{tab:Pisot}, we list all Pisot numbers in the open interval $%
(1,2)$ of degree less than or equal to 4. For each of these, we give the
number of vertices of the reduced transition graph and the size of the
reduced essential class for the Bernoulli convolution with contraction
factor the inverse of this Pisot number. In the case where the size is
listed as 'Unknown', there are more than 10000 reduced characteristic
vectors.

\subsubsection{Bernoulli convolutions with no isolated point\label{ex:other
Pisot}}

It can be shown for Bernoulli convolutions that whenever there are precisely
three more elements in the reduced transition graph than in the essential
set, then the non-essential set consists of the characteristic vector of $%
[0,1]$ and the characteristic vectors of the right-most and left-most net
intervals of level $1$. The latter two are maximal loop classes
corresponding to the two endpoints of $[0,1]$, the only two non-essential
points. Thus, with the exception of $x^{3}-x^{2}-1,$ $x^{3}-x-1$ and
possibly $x^{4}-x^{3}-1$, for the examples listed in Table \ref{tab:Pisot}
that the open interval $(0,1)$ is the set of essential points. We have
checked that in all of these examples (where the essential set is known to
be $(0,1))$, the value of the local dimension of the measure at $0$ is also
the local dimension at an essential point, hence there is no isolated point.

\subsubsection{Bernoulli convolutions with an isolated point\label%
{BCisolatedpt}}

\begin{example}
\label{ex:$x^3-x^2-1$)}Minimal polynomial $P_{1}(x)=x^{3}-x^{2}-1$: The
uniform Bernoulli convolution $\mu _{\varrho },$ with $\varrho ^{-1}$ the
Pisot number with minimal polynomial $P_{1},$ has five maximal loop classes.
In addition to the essential class, there are two singletons (corresponding
to the points $0,1$), one doubleton and one of size 23. All are of positive
type. The set of local dimensions of $\mu _{\varrho }$ consists of an
isolated point ($\dim \mu _{\varrho }(0)$) and a closed interval which is
the union of the closed intervals generated by the three non-trivial maximal
loop classes. In particular, we have 
\[
\lbrack .970222,1.07770] \subset \{\dim \mu (x):x\text{ essential}\}
 \subset \lbrack .848302,1.53266]
\]
and 
\begin{align*}
\lbrack .970221,1.07771]\cup \{1.81336\}& \subset \{\dim \mu
(x):x\in \lbrack 0,1]\} \\
& \subset \lbrack .848302,1.53265]\cup \{1.81336\}
\end{align*}%
See the supplementary information \cite{Homepage} for this example worked out in full.
\end{example}

\begin{example}
\label{ex:$x^3-x-1$)}Minimal polynomial $P_{2}(x)=x^{3}-x-1$: The uniform
Bernoulli convolution $\mu _{\varrho },$ with $\varrho ^{-1}$ the Pisot
number with minimal polynomial $P_{2},$ has six maximal loop classes. In
addition to the essential class, there are four singletons (two
corresponding to the points $0,1$) and one of size 6. The two corresponding
to the points $0$ and $1$, as well as the maximal loop of size 6, are of
positive type. The other two singletons, although not positive type, are
easy to handle as each has only one transition matrix. The set of local
dimensions of $\mu _{\varrho }$ consists of an isolated point ($\dim \mu
_{\varrho }(0)$) and a closed interval which is the union of the closed
intervals generated by the four non-trivial maximal loop classes. In
particular, we have 
\[
\lbrack .997949,1.00853] \subset \{\dim \mu (x):x\text{ essential}\}
 \subset \lbrack .747924,1.97198]
\]
and 
\begin{align*}
\lbrack .997949,1.00853]\cup \{2.46497\}& \subset \{\dim \mu
(x):x\in \lbrack 0,1]\} \\
& \subset \lbrack .747923,1.97198]\cup \{2.46497\}
\end{align*}%
See the supplementary information \cite{Homepage} for this example worked out in full.
\end{example}

\begin{table}[tbp]
\begin{center}
\begin{tabular}{llll}
Minimal Polynomial of & Approx value of & Size of reduce & Size of Essential
\\ 
Pisot Number & $\varrho$ & Transition graph & set (reduced) \\ \hline
$x^2-x-1 $ & .618034 & 6 & 3 \\ 
$x^3-x^2-1 $ & .682328 & 152 & 46 \\ 
$x^3-x-1 $ & .754878 & 1809 & 1207 \\ 
$x^3-2 x^2+x-1 $ & .569840 & 30 & 27 \\ 
$x^3-x^2-x-1 $ & .543689 & 11 & 8 \\ 
$x^4-x^3-1 $ & .724492 & Unknown &  \\ 
$x^4-x^3-2 x^2+1 $ & .524889 & 538 & 535 \\ 
$x^4-2 x^3+x-1 $ & .535687 & 190 & 187 \\ 
$x^4-x^3-x^2-x-1 $ & .518790 & 14 & 11 \\ 
&  &  & 
\end{tabular}%
.
\end{center}
\caption{Facts about Bernoulli convolutions with $\protect\varrho^{-1}$ a
small degree Pisot number}
\label{tab:Pisot}
\end{table}

\subsection{The $2$-fold convolution of $\protect\mu_\protect\varrho$}

\label{ssec:golden square} In this subsection we study the rescaled measure $%
\nu _{\varrho }=\mu _{\varrho }\ast \mu _{\varrho },$ where $\mu _{\varrho }$
is the Bernoulli convolution with $\varrho ^{-1}$ the golden ratio. This is
the self-similar measure associated with the IFS of contractions $%
S_{1}(x)=\varrho x$, $S_{2}(x)=\varrho x+1/2-\varrho /2$ and $%
S_{3}(x)=\varrho x+1-\varrho $, with corresponding (regular) probabilities
\thinspace $(1/4,1/2,1/4)$. It is of finite type and has support $[0,1]$.

The reduced transition diagram has 40 reduced characteristic vectors. The
essential class can be naturally identified with those labelled by $%
\{28,29,30,33,34,\dots,40\}$. Two cycles in the essential class are $\eta
_{1}=(29,35,39,29)$ and $\eta _{2}=(28,33,28)$. The spectral radius of $%
T(\eta _{1})$ is approximately $2.46916$, while the spectral radius of $%
T(\eta _{2})$ is approximately $2.48119$. This shows that%
\begin{equation*}
\lbrack .992400,1.00250]\subseteq \{\dim \nu _{\rho }(x):x\text{essential}\}.
\end{equation*}

We have also been able to find upper and lower bounds on the local
dimensions from the essential class using the method described in Subsection %
\ref{algdim}.\ We obtain an upper bound by taking the column sub-norm with
the subset $C=$ $\{3,4\}$ and taking admissible products of up to 20
primitive transition matrices. We obtain a lower bound by using the total
column sup-norm with products of up to 10 primitive transition matrices.
These calculations give%
\begin{equation*}
\{\dim \nu _{\rho }(x):x\text{ essential}\}\subseteq \lbrack
.815721,1.40091].
\end{equation*}

There are four non-essential maximal loops, each of which is a singleton.
The maximal loop classes $\{2\}$ and \thinspace $\{6\}$ correspond to the
two endpoints of the support, $0,1$. The transition matrix in both cases is
the $1\times 1$ identity matrix and the points have local dimension 
\begin{equation*}
\dim \nu _{\rho }(0)=\dim \nu _{\rho }(1)=\frac{\log 4}{\left\vert \log \rho
\right\vert }\sim 2.88084.
\end{equation*}

The other two maximal loop classes are $\{25\}$ and $\{19\}$. The
characteristic vector of 25 is $(\varrho -1/2,(1-3/2\varrho ,1/2-1/2\varrho
,1-\varrho ,3/2-3/2\varrho ,1-1/2\varrho ,3/2-\varrho ))$. Its transition
matrix has the same spectral radius as $T(\eta _{2})$, hence the local
dimension at any point in the loop class \thinspace $\{25\}$ coincides with
the local dimension at some essential point. Similar statements hold for the
loop class $\{19\}$.

Thus the set of local dimensions of $\nu _{\rho }$ consists of an interval
and an isolated point, $\dim \nu _{\varrho }(0)$.

We recall that a point $x$ can have a most two symbolic representations, and
that this will occur only if $x\in \mathcal{F}_{n}$. Let $x^{(n)}$ be the
point with symbolic representation 
\begin{equation*}
x^{(n)}=(1,\underbrace{2,\dots ,2}_{n},7,8,10,19,19,19,\dots ).
\end{equation*}%
This point also has symbolic representation 
\begin{equation*}
x^{(n)}=(1,\underbrace{2,\dots ,2}_{n+1},7,9,12,25,25,25,\dots ).
\end{equation*}%
We see that both representations of these points are external to the loop
class, and hence we have a countable number of non-essential points.

Not all points with a symbolic representation in a loop class external to
the essential set need be non-essential. To see this we observe that 
\begin{equation*}
(1, 4, 14, 22, 30, 37, 30, 37, 30, 37, \dots) = (1, 6, 18, 16, 13, 19, 19,
19, 19, \dots)
\end{equation*}
are two symbolic representations for the same point, one of which is in the
essential class, and one of which is not. Hence this point is an essential
point.

\begin{remark}
The transition matrices of the cycles $\eta _{1}$ and $\eta _{2}$ give the
extreme values of spectral radii over all transition matrices of essential
cycles of length up to 10. It would be interesting to know if these give the
endpoints of the interval portion of the set of local dimensions.
\end{remark}

\section*{Acknowledgements}

We thank M. Ng for helping with some of the calculations on the Cantor-like
measures.

\vspace{2 in}

Dept. of Pure Mathematics,
University of Waterloo,
Waterloo, Ont., N2L 3G1,
Canada.
\texttt{kehare@uwaterloo.ca} \\ 

Dept. of Pure Mathematics,
University of Waterloo,
Waterloo, Ont., N2L 3G1,
Canada.
\texttt{kghare@uwaterloo.ca} \\ 

Dept. of Pure Mathematics,
University of Waterloo,
Waterloo, Ont., N2L 3G1,
Canada.
\texttt{kevinmatthews12@hotmail.com}

\newpage
\clearpage

\allowdisplaybreaks
\begin{center}
{\textbf{\large LOCAL DIMENSIONS OF MEASURES OF FINITE TYPE -- SUPPLEMENTARY INFORMATION}}

\vspace{1 in}

KATHRYN E. HARE, KEVIN G. HARE, AND KEVIN R. MATTHEWS
\end{center}

%\title[Supplementary Info]
%\author[K.~E.~Hare, K.~G.~Hare, K.~R.~Matthews]{Kathryn E. Hare, Kevin G. Hare, and Kevin R. Matthews}
%\thanks{The research of all three authors is supported in part by NSERC.} 
%\maketitle

\thispagestyle{empty}
\renewcommand{\rightmark}{SUPPLEMENTARY INFO}
\setcounter{equation}{0}
\setcounter{section}{0}
\setcounter{figure}{0}
\setcounter{table}{0}
\setcounter{page}{1}
\makeatletter
\renewcommand{\thefigure}{S\arabic{figure}}
\renewcommand{\thesection}{S\arabic{section}}

\section{Outline}
This is the supplementary information for the manuscript
    {\em Local Dimensions of Measures of Finite Type}
    by Kathryn E. Hare, Kevin G. Hare and Kevin R. Matthews.
This contains the details of the examples mentioned in this 
    manuscript.

By ``the probabilities are uniform'' we mean that 
    $p_0 = p_1 = \dots = p_m$.

The spectral range of a loop class is
    \[ \mathrm{closure}\{sp(T(\theta)^{1/L(\theta)}:  \theta \text{ is a cycle in the loop class}\}. \]
Recalling that the dimension of a periodic point with cycle 
    $\theta$ is 
\[ \dim \mu (x) = \frac{\log p_0}{\varrho} +
    \frac{\log(T(\theta))}{L(\theta) \log \varrho} \]
    allows an easy translation from the spectral range of a loop
    and the set of local dimensions of a loop class.

The outline of the examples are as follows.

\begin{itemize}
\item Examples 2.6, 3.11, 4.7, 4.10.
\begin{itemize}
\item Section \ref{sec:1} is {for the minimal polynomial $x^2+x-1$ with $d_i \in [0, 1-\varrho]$} (page \pageref{sec:1}).
\end{itemize}

\item Examples from Section 7.
\begin{itemize}
\item Section \ref{sec:2} is {for the minimal polynomial $3 x-1$ with 
    $d_i = \frac{i}{3} \cdot \frac{2}{3}, i=0,\dots,3$} (page \pageref{sec:2}).
\item Section \ref{sec:3} is {for the minimal polynomial $3 x-1$ with 
    $d_i = \frac{i}{4} \cdot \frac{2}{3}, i=0,\dots,4$} (page \pageref{sec:3}).
\item Section \ref{sec:4} is {for the minimal polynomial $3 x-1$ with 
    $d_i = \frac{i}{5} \cdot \frac{2}{3}, i=0,\dots,5$} (page \pageref{sec:4}).
\item Section \ref{sec:5} is {for the minimal polynomial $3 x-1$ with 
    $d_i = \frac{i}{6} \cdot \frac{2}{3}, i=0,\dots,6$} (page \pageref{sec:5}).
\item Section \ref{sec:6} is {for the minimal polynomial $3 x-1$ with 
    $d_i = \frac{i}{7} \cdot \frac{2}{3}, i=0,\dots,7$} (page \pageref{sec:6}).
\item Section \ref{sec:7} is {for the minimal polynomial $3 x-1$ with 
    $d_i = \frac{i}{8} \cdot \frac{2}{3}, i=0,\dots,8$} (page \pageref{sec:7}).
\item Section \ref{sec:8} is {for the minimal polynomial $3 x-1$ with 
    $d_i = \frac{i}{9} \cdot \frac{2}{3}, i=0,\dots,9$} (page \pageref{sec:8}).
\item Section \ref{sec:9} is {for the minimal polynomial $3 x-1$ with 
    $d_i = \frac{i}{10} \cdot \frac{2}{3}, i=0,\dots,10$} (page \pageref{sec:9}).
\end{itemize}

\item Examples from Section 8, including Table 3 and Examples 8.1 and 8.2
\begin{itemize}
\item Section \ref{sec:10} is {for the minimal polynomial $x^3+x-1$ with $d_i \in [0, 1-\varrho]$} (page \pageref{sec:10}).
\item Section \ref{sec:11} is {for the minimal polynomial $x^3+x^2-1$ with $d_i \in [0, 1-\varrho]$} (page \pageref{sec:11}).
\item Section \ref{sec:12} is {for the minimal polynomial $x^3-x^2+2 x-1$ with $d_i \in [0, 1-\varrho]$} (page \pageref{sec:12}).
\item Section \ref{sec:13} is {for the minimal polynomial $x^3+x^2+x-1$ with $d_i \in [0, 1-\varrho]$} (page \pageref{sec:13}).
\item Section \ref{sec:14} is {for the minimal polynomial $x^4-2 x^2-x+1$ with $d_i \in [0, 1-\varrho]$} (page \pageref{sec:14}).
\item Section \ref{sec:15} is {for the minimal polynomial $x^4-x^3+2 x-1$ with $d_i \in [0, 1-\varrho]$} (page \pageref{sec:15}).
\item Section \ref{sec:16} is {for the minimal polynomial $x^4+x^3+x^2+x-1$ with $d_i \in [0, 1-\varrho]$} (page \pageref{sec:16}).
\end{itemize}

\item Example in Section 8.2.
\begin{itemize}
\item Section \ref{sec:17} is {for the minimal polynomial $x^2+x-1$ with $d_i \in [0, 1/2-1/2 \varrho, 1-\varrho]$} (page \pageref{sec:17}).
\end{itemize}

\item Example 6.1
\begin{itemize}
\item Section \ref{sec:18} is {for the minimal polynomial $3 x-1$ with 
    $d_i = \frac{i}{3} \cdot \frac{2}{3}, i=0,\dots,3$} (page \pageref{sec:18}).
\item Section \ref{sec:19} is {for the minimal polynomial $3 x-1$ with 
    $d_i = \frac{i}{4} \cdot \frac{2}{3}, i=0,\dots,4$} (page \pageref{sec:19}).
\item Section \ref{sec:20} is {for the minimal polynomial $3 x-1$ with 
    $d_i = \frac{i}{5} \cdot \frac{2}{3}, i=0,\dots,5$} (page \pageref{sec:20}).
\end{itemize}
\end{itemize}

\section{Minimal polynomial $x^2+x-1$ with $d_i \in [0, 1-\varrho]$} 
\label{sec:1}
 
Consider $\varrho$, the root of $x^2+x-1$ and the maps $S_i(x) = \varrho x  + d_i$ with $d_{0} = 0$, and $d_{1} = 1-\varrho$.
The probabilities are uniform.
The reduced transition diagram has 6 reduced characteristic vectors.
The reduced characteristic vectors are:
\begin{itemize}
\item Reduced characteristic vector 1: $(1, (0))$ 
\item Reduced characteristic vector 2: $(\varrho, (0))$ 
\item Reduced characteristic vector 3: $(1-\varrho, (0, \varrho))$ 
\item Reduced characteristic vector 4: $(\varrho, (1-\varrho))$ 
\item Reduced characteristic vector 5: $(\varrho, (0, 1-\varrho))$ 
\item Reduced characteristic vector 6: $(2 \varrho-1, (1-\varrho))$ 
\end{itemize}
See Figure \ref{fig:Pic1} for the transition diagram.
\begin{figure}[H]
\includegraphics[scale=0.5]{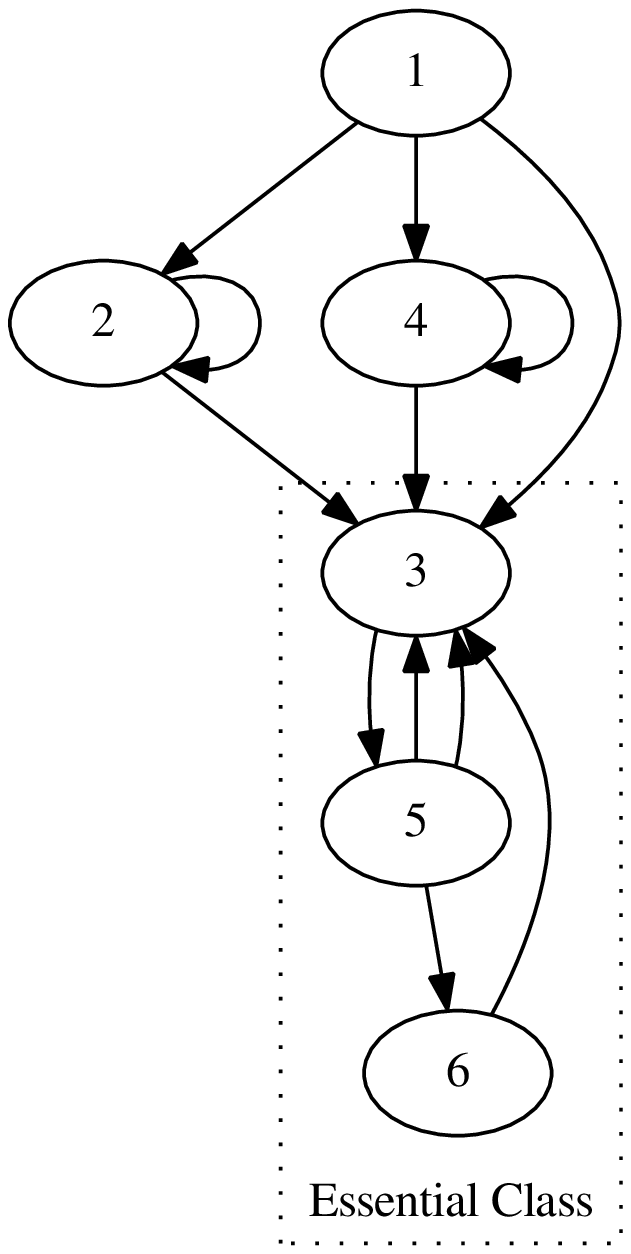}
\caption{$x^2+x-1$ with $d_i \in [0, 1-\varrho]$, Full set + Essential class}
\label{fig:Pic1}
\end{figure}
This has transition matrices:
\begin{align*}
T(1,2) & =  \left[ \begin {array}{c} 1\end {array} \right] & 
 T(1,3) & =  \left[ \begin {array}{cc} 1&1\end {array} \right] \\ 
T(1,4) & =  \left[ \begin {array}{c} 1\end {array} \right] & 
 T(2,2) & =  \left[ \begin {array}{c} 1\end {array} \right] \\ 
T(2,3) & =  \left[ \begin {array}{cc} 1&1\end {array} \right] & 
 T(3,5) & =  \left[ \begin {array}{cc} 1&0\\ 0&1\end {array} \right] \\ 
T(4,3) & =  \left[ \begin {array}{cc} 1&1\end {array} \right] & 
 T(4,4) & =  \left[ \begin {array}{c} 1\end {array} \right] \\ 
T(5,3) & =  \left[ \begin {array}{cc} 1&0\\ 1&1\end {array} \right] & 
 T(5,6) & =  \left[ \begin {array}{c} 1\\ 1\end {array} \right] \\ 
T(5,3) & =  \left[ \begin {array}{cc} 1&1\\ 0&1\end {array} \right] & 
 T(6,3) & =  \left[ \begin {array}{cc} 1&1\end {array} \right] \\ 
\end{align*}

The essential class is: [3, 5, 6].
The essential class is of positive type.
An example is the path [5, 6].
The essential class is not a simple loop.
This spectral range will include the interval $[1., 1.272019649]$.
The minimum comes from the loop $[3, 5, 3]$.
The maximum comes from the loop $[3, 5, 3, 5, 3]$.
These points will include points of local dimension [.9404200909, 1.440420090].
The Spectral Range is contained in the range $[1., 1.319507911]$.
The minimum comes from the total row sub-norm of length 10. 
The maximum comes from the total row sup-norm of length 10. 
These points will have local dimension contained in [.8642520535, 1.440420090].

There are 2 additional maximal loops.

Maximal Loop Class: [4].
The maximal loop class is a simple loop.
It's spectral radius is an isolated points of 1.
These points have local dimension 1.440420090.

Maximal Loop Class: [2].
The maximal loop class is a simple loop.
It's spectral radius is an isolated points of 1.
These points have local dimension 1.440420090.

\section{Minimal polynomial $3 x-1$ with $d_i \in [0, 2/9, 4/9, 2/3]$} 
\label{sec:2}
 
Consider $\varrho$, the root of $3 x-1$ and the maps $S_i(x) = \varrho x  + d_i$ with $d_{0} = 0$, $d_{1} = 2/9$, $d_{2} = 4/9$, and $d_{3} = 2/3$.
The probabilities are given by $p_{0} = 1/8$, $p_{1} = 3/8$, $p_{2} = 3/8$, and $p_{3} = 1/8$.
The reduced transition diagram has 5 reduced characteristic vectors.
The reduced characteristic vectors are:
\begin{itemize}
\item Reduced characteristic vector 1: $(1, (0))$ 
\item Reduced characteristic vector 2: $(2/3, (0))$ 
\item Reduced characteristic vector 3: $(1/3, (0, 2/3))$ 
\item Reduced characteristic vector 4: $(1/3, (1/3))$ 
\item Reduced characteristic vector 5: $(2/3, (1/3))$ 
\end{itemize}
See Figure \ref{fig:Pic2} for the transition diagram.
\begin{figure}[H]
\includegraphics[scale=0.5]{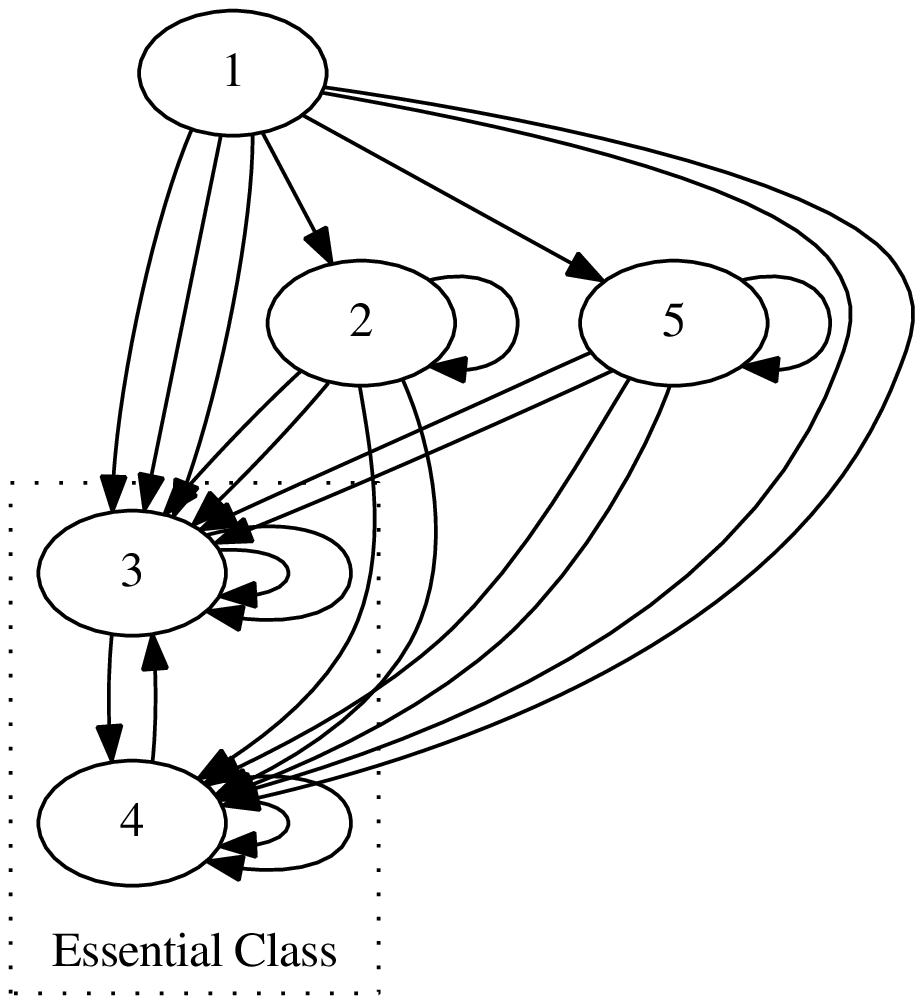}
\caption{$3 x-1$ with $d_i \in [0, 2/9, 4/9, 2/3]$, Full set + Essential class}
\label{fig:Pic2}
\end{figure}
This has transition matrices:
\begin{align*}
T(1,2) & =  \left[ \begin {array}{c} 1\end {array} \right] & 
 T(1,3) & =  \left[ \begin {array}{cc} 3&1\end {array} \right] \\ 
T(1,4) & =  \left[ \begin {array}{c} 3\end {array} \right] & 
 T(1,3) & =  \left[ \begin {array}{cc} 3&3\end {array} \right] \\ 
T(1,4) & =  \left[ \begin {array}{c} 3\end {array} \right] & 
 T(1,3) & =  \left[ \begin {array}{cc} 1&3\end {array} \right] \\ 
T(1,5) & =  \left[ \begin {array}{c} 1\end {array} \right] & 
 T(2,2) & =  \left[ \begin {array}{c} 1\end {array} \right] \\ 
T(2,3) & =  \left[ \begin {array}{cc} 3&1\end {array} \right] & 
 T(2,4) & =  \left[ \begin {array}{c} 3\end {array} \right] \\ 
T(2,3) & =  \left[ \begin {array}{cc} 3&3\end {array} \right] & 
 T(2,4) & =  \left[ \begin {array}{c} 3\end {array} \right] \\ 
T(3,3) & =  \left[ \begin {array}{cc} 1&0\\ 1&3\end {array} \right] & 
 T(3,4) & =  \left[ \begin {array}{c} 1\\ 1\end {array} \right] \\ 
T(3,3) & =  \left[ \begin {array}{cc} 3&1\\ 0&1\end {array} \right] & 
 T(4,4) & =  \left[ \begin {array}{c} 3\end {array} \right] \\ 
T(4,3) & =  \left[ \begin {array}{cc} 3&3\end {array} \right] & 
 T(4,4) & =  \left[ \begin {array}{c} 3\end {array} \right] \\ 
T(5,4) & =  \left[ \begin {array}{c} 3\end {array} \right] & 
 T(5,3) & =  \left[ \begin {array}{cc} 3&3\end {array} \right] \\ 
T(5,4) & =  \left[ \begin {array}{c} 3\end {array} \right] & 
 T(5,3) & =  \left[ \begin {array}{cc} 1&3\end {array} \right] \\ 
T(5,5) & =  \left[ \begin {array}{c} 1\end {array} \right] & 
 \end{align*}

The essential class is: [3, 4].
The essential class is of positive type.
An example is the path [3, 4].
The essential class is not a simple loop.
This spectral range will include the interval $[2.302775638, 3.]$.
The minimum comes from the loop $[3, 3, 3]$.
The maximum comes from the loop $[3, 3]$.
These points will include points of local dimension [.892789260, 1.133544891].
The Spectral Range is contained in the range $[2.260322470, 3.000000000]$.
The minimum comes from the total row sub-norm of length 5. 
The maximum comes from the total row sup-norm of length 5. 
These points will have local dimension contained in [.892789260, 1.150482354].

There are 2 additional maximal loops.

Maximal Loop Class: [5].
The maximal loop class is a simple loop.
It's spectral radius is an isolated points of 1.
These points have local dimension 1.892789260.

Maximal Loop Class: [2].
The maximal loop class is a simple loop.
It's spectral radius is an isolated points of 1.
These points have local dimension 1.892789260.

\section{Minimal polynomial $3 x-1$ with $d_i \in [0, 1/6, 1/3, 1/2, 2/3]$} 
\label{sec:3}
 
Consider $\varrho$, the root of $3 x-1$ and the maps $S_i(x) = \varrho x  + d_i$ with $d_{0} = 0$, $d_{1} = 1/6$, $d_{2} = 1/3$, $d_{3} = 1/2$, and $d_{4} = 2/3$.
The probabilities are given by $p_{0} = 1/16$, $p_{1} = 1/4$, $p_{2} = 3/8$, $p_{3} = 1/4$, and $p_{4} = 1/16$.
The reduced transition diagram has 4 reduced characteristic vectors.
The reduced characteristic vectors are:
\begin{itemize}
\item Reduced characteristic vector 1: $(1, (0))$ 
\item Reduced characteristic vector 2: $(1/2, (0))$ 
\item Reduced characteristic vector 3: $(1/2, (0, 1/2))$ 
\item Reduced characteristic vector 4: $(1/2, (1/2))$ 
\end{itemize}
See Figure \ref{fig:Pic3} for the transition diagram.
\begin{figure}[H]
\includegraphics[scale=0.5]{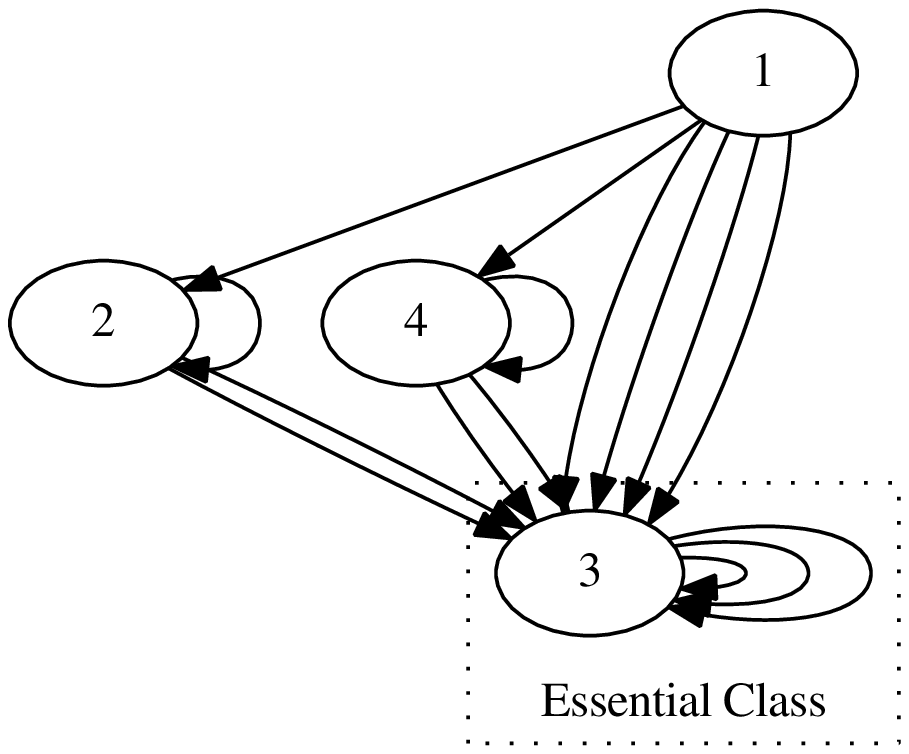}
\caption{$3 x-1$ with $d_i \in [0, 1/6, 1/3, 1/2, 2/3]$, Full set + Essential class}
\label{fig:Pic3}
\end{figure}
This has transition matrices:
\begin{align*}
T(1,2) & =  \left[ \begin {array}{c} 1\end {array} \right] & 
 T(1,3) & =  \left[ \begin {array}{cc} 4&1\end {array} \right] \\ 
T(1,3) & =  \left[ \begin {array}{cc} 6&4\end {array} \right] & 
 T(1,3) & =  \left[ \begin {array}{cc} 4&6\end {array} \right] \\ 
T(1,3) & =  \left[ \begin {array}{cc} 1&4\end {array} \right] & 
 T(1,4) & =  \left[ \begin {array}{c} 1\end {array} \right] \\ 
T(2,2) & =  \left[ \begin {array}{c} 1\end {array} \right] & 
 T(2,3) & =  \left[ \begin {array}{cc} 4&1\end {array} \right] \\ 
T(2,3) & =  \left[ \begin {array}{cc} 6&4\end {array} \right] & 
 T(3,3) & =  \left[ \begin {array}{cc} 1&0\\ 4&6\end {array} \right] \\ 
T(3,3) & =  \left[ \begin {array}{cc} 4&1\\ 1&4\end {array} \right] & 
 T(3,3) & =  \left[ \begin {array}{cc} 6&4\\ 0&1\end {array} \right] \\ 
T(4,3) & =  \left[ \begin {array}{cc} 4&6\end {array} \right] & 
 T(4,3) & =  \left[ \begin {array}{cc} 1&4\end {array} \right] \\ 
T(4,4) & =  \left[ \begin {array}{c} 1\end {array} \right] & 
 \end{align*}

The essential class is: [3].
The essential class is of positive type.
An example is the path [3, 3].
The essential class is not a simple loop.
This spectral range will include the interval $[5., 6.]$.
The minimum comes from the loop $[3, 3]$.
The maximum comes from the loop $[3, 3]$.
These points will include points of local dimension [.892789260, 1.058745493].
The Spectral Range is contained in the range $[5.000000000, 6.000000000]$.
The minimum comes from the total row sub-norm of length 5. 
The maximum comes from the total row sup-norm of length 5. 
These points will have local dimension contained in [.892789260, 1.058745493].

There are 2 additional maximal loops.

Maximal Loop Class: [4].
The maximal loop class is a simple loop.
It's spectral radius is an isolated points of 1.
These points have local dimension 2.523719013.

Maximal Loop Class: [2].
The maximal loop class is a simple loop.
It's spectral radius is an isolated points of 1.
These points have local dimension 2.523719013.

\section{Minimal polynomial $3 x-1$ with $d_i \in [0, 2/15, 4/15, 2/5, 8/15, 2/3]$} 
\label{sec:4}
 
Consider $\varrho$, the root of $3 x-1$ and the maps $S_i(x) = \varrho x  + d_i$ with $d_{0} = 0$, $d_{1} = 2/15$, $d_{2} = 4/15$, $d_{3} = 2/5$, $d_{4} = 8/15$, and $d_{5} = 2/3$.
The probabilities are given by $p_{0} = 1/32$, $p_{1} = 5/32$, $p_{2} = 5/16$, $p_{3} = 5/16$, $p_{4} = 5/32$, and $p_{5} = 1/32$.
The reduced transition diagram has 7 reduced characteristic vectors.
The reduced characteristic vectors are:
\begin{itemize}
\item Reduced characteristic vector 1: $(1, (0))$ 
\item Reduced characteristic vector 2: $(2/5, (0))$ 
\item Reduced characteristic vector 3: $(2/5, (0, 2/5))$ 
\item Reduced characteristic vector 4: $(1/5, (0, 2/5, 4/5))$ 
\item Reduced characteristic vector 5: $(1/5, (1/5, 3/5))$ 
\item Reduced characteristic vector 6: $(2/5, (1/5, 3/5))$ 
\item Reduced characteristic vector 7: $(2/5, (3/5))$ 
\end{itemize}
See Figure \ref{fig:Pic4} for the transition diagram.
\begin{figure}[H]
\includegraphics[scale=0.5]{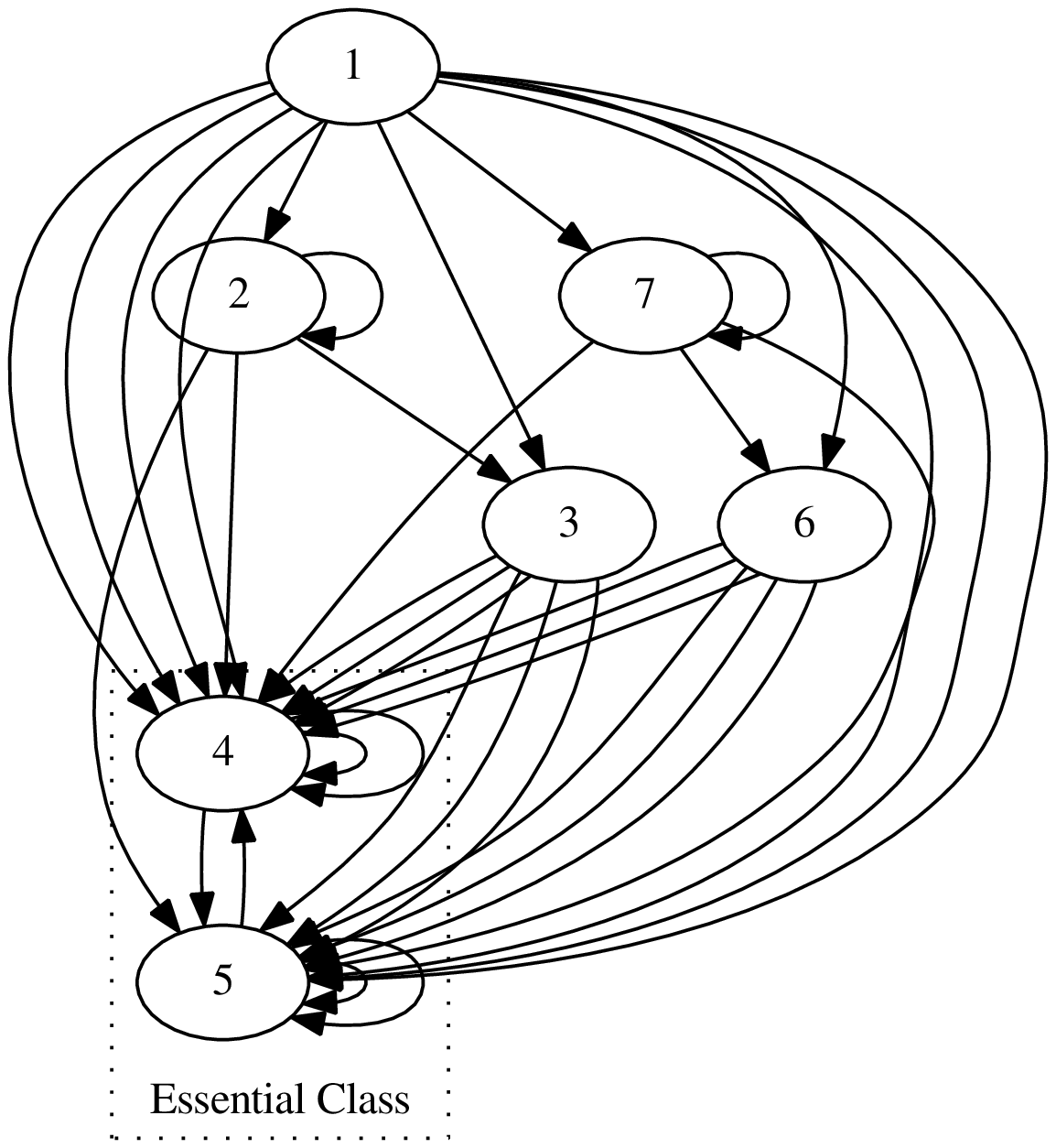}
\caption{$3 x-1$ with $d_i \in [0, 2/15, 4/15, 2/5, 8/15, 2/3]$, Full set + Essential class}
\label{fig:Pic4}
\end{figure}
This has transition matrices:
\begin{align*}
T(1,2) & =  \left[ \begin {array}{c} 1\end {array} \right] & 
 T(1,3) & =  \left[ \begin {array}{cc} 5&1\end {array} \right] \\ 
T(1,4) & =  \left[ \begin {array}{ccc} 10&5&1\end {array} \right] & 
 T(1,5) & =  \left[ \begin {array}{cc} 10&5\end {array} \right] \\ 
T(1,4) & =  \left[ \begin {array}{ccc} 10&10&5\end {array} \right] & 
 T(1,5) & =  \left[ \begin {array}{cc} 10&10\end {array} \right] \\ 
T(1,4) & =  \left[ \begin {array}{ccc} 5&10&10\end {array} \right] & 
 T(1,5) & =  \left[ \begin {array}{cc} 5&10\end {array} \right] \\ 
T(1,4) & =  \left[ \begin {array}{ccc} 1&5&10\end {array} \right] & 
 T(1,6) & =  \left[ \begin {array}{cc} 1&5\end {array} \right] \\ 
T(1,7) & =  \left[ \begin {array}{c} 1\end {array} \right] & 
 T(2,2) & =  \left[ \begin {array}{c} 1\end {array} \right] \\ 
T(2,3) & =  \left[ \begin {array}{cc} 5&1\end {array} \right] & 
 T(2,4) & =  \left[ \begin {array}{ccc} 10&5&1\end {array} \right] \\ 
T(2,5) & =  \left[ \begin {array}{cc} 10&5\end {array} \right] & 
 T(3,4) & =  \left[ \begin {array}{ccc} 1&0&0\\ 10&10&5\end {array} \right] \\ 
T(3,5) & =  \left[ \begin {array}{cc} 1&0\\ 10&10\end {array} \right] & 
 T(3,4) & =  \left[ \begin {array}{ccc} 5&1&0\\ 5&10&10\end {array} \right] \\ 
T(3,5) & =  \left[ \begin {array}{cc} 5&1\\ 5&10\end {array} \right] & 
 T(3,4) & =  \left[ \begin {array}{ccc} 10&5&1\\ 1&5&10\end {array} \right] \\ 
T(3,5) & =  \left[ \begin {array}{cc} 10&5\\ 1&5\end {array} \right] & 
 T(4,4) & =  \left[ \begin {array}{ccc} 1&0&0\\ 10&10&5\\ 0&1&5\end {array} \right] \\ 
T(4,5) & =  \left[ \begin {array}{cc} 1&0\\ 10&10\\ 0&1\end {array} \right] & 
 T(4,4) & =  \left[ \begin {array}{ccc} 5&1&0\\ 5&10&10\\ 0&0&1\end {array} \right] \\ 
T(5,5) & =  \left[ \begin {array}{cc} 5&1\\ 5&10\end {array} \right] & 
 T(5,4) & =  \left[ \begin {array}{ccc} 10&5&1\\ 1&5&10\end {array} \right] \\ 
T(5,5) & =  \left[ \begin {array}{cc} 10&5\\ 1&5\end {array} \right] & 
 T(6,5) & =  \left[ \begin {array}{cc} 5&1\\ 5&10\end {array} \right] \\ 
T(6,4) & =  \left[ \begin {array}{ccc} 10&5&1\\ 1&5&10\end {array} \right] & 
 T(6,5) & =  \left[ \begin {array}{cc} 10&5\\ 1&5\end {array} \right] \\ 
T(6,4) & =  \left[ \begin {array}{ccc} 10&10&5\\ 0&1&5\end {array} \right] & 
 T(6,5) & =  \left[ \begin {array}{cc} 10&10\\ 0&1\end {array} \right] \\ 
T(6,4) & =  \left[ \begin {array}{ccc} 5&10&10\\ 0&0&1\end {array} \right] & 
 T(7,5) & =  \left[ \begin {array}{cc} 5&10\end {array} \right] \\ 
T(7,4) & =  \left[ \begin {array}{ccc} 1&5&10\end {array} \right] & 
 T(7,6) & =  \left[ \begin {array}{cc} 1&5\end {array} \right] \\ 
T(7,7) & =  \left[ \begin {array}{c} 1\end {array} \right] & 
 \end{align*}

The essential class is: [4, 5].
The essential class is of positive type.
An example is the path [5, 4].
The essential class is not a simple loop.
This spectral range will include the interval $[10.34846923, 10.99217650]$.
The minimum comes from the loop $[5, 5, 5]$.
The maximum comes from the loop $[4, 4, 4]$.
These points will include points of local dimension [.972638047, 1.027566600].
The Spectral Range is contained in the range $[10.29826851, 10.99526948]$.
The minimum comes from the total row sub-norm of length 5. 
The maximum comes from the total row sup-norm of length 5. 
These points will have local dimension contained in [.972381959, 1.031992942].

There are 2 additional maximal loops.

Maximal Loop Class: [7].
The maximal loop class is a simple loop.
It's spectral radius is an isolated points of 1.
These points have local dimension 3.154648767.

Maximal Loop Class: [2].
The maximal loop class is a simple loop.
It's spectral radius is an isolated points of 1.
These points have local dimension 3.154648767.

\section{Minimal polynomial $3 x-1$ with $d_i \in [0, 1/9, 2/9, 1/3, 4/9, 5/9, 2/3]$} 
\label{sec:5}
 
Consider $\varrho$, the root of $3 x-1$ and the maps $S_i(x) = \varrho x  + d_i$ with $d_{0} = 0$, $d_{1} = 1/9$, $d_{2} = 2/9$, $d_{3} = 1/3$, $d_{4} = 4/9$, $d_{5} = 5/9$, and $d_{6} = 2/3$.
The probabilities are given by $p_{0} = 1/64$, $p_{1} = 3/32$, $p_{2} = 15/64$, $p_{3} = 5/16$, $p_{4} = 15/64$, $p_{5} = 3/32$, and $p_{6} = 1/64$.
The reduced transition diagram has 6 reduced characteristic vectors.
The reduced characteristic vectors are:
\begin{itemize}
\item Reduced characteristic vector 1: $(1, (0))$ 
\item Reduced characteristic vector 2: $(1/3, (0))$ 
\item Reduced characteristic vector 3: $(1/3, (0, 1/3))$ 
\item Reduced characteristic vector 4: $(1/3, (0, 1/3, 2/3))$ 
\item Reduced characteristic vector 5: $(1/3, (1/3, 2/3))$ 
\item Reduced characteristic vector 6: $(1/3, (2/3))$ 
\end{itemize}
See Figure \ref{fig:Pic5} for the transition diagram.
\begin{figure}[H]
\includegraphics[scale=0.5]{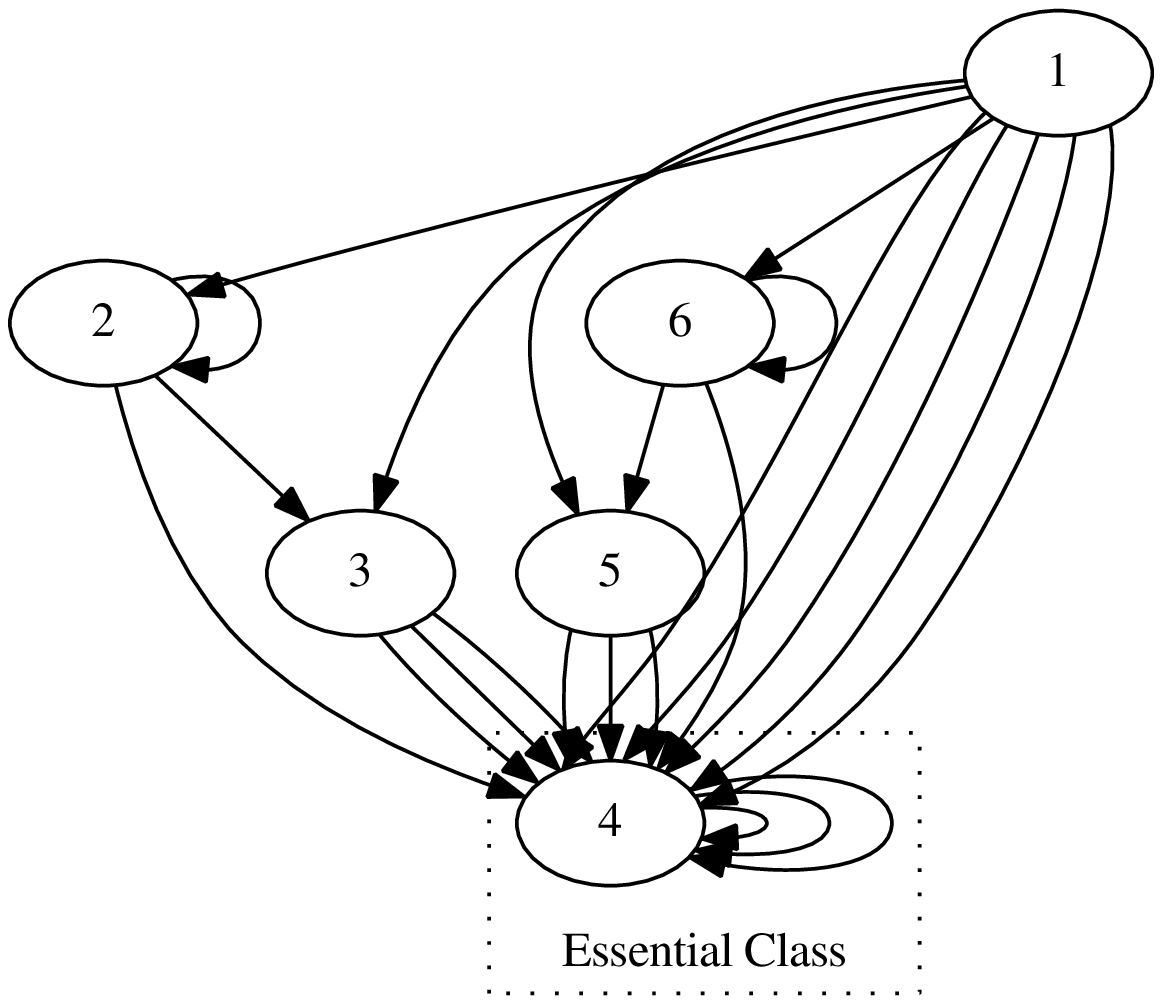}
\caption{$3 x-1$ with $d_i \in [0, 1/9, 2/9, 1/3, 4/9, 5/9, 2/3]$, Full set + Essential class}
\label{fig:Pic5}
\end{figure}
This has transition matrices:
\begin{align*}
T(1,2) & =  \left[ \begin {array}{c} 1\end {array} \right] & 
 T(1,3) & =  \left[ \begin {array}{cc} 6&1\end {array} \right] \\ 
T(1,4) & =  \left[ \begin {array}{ccc} 15&6&1\end {array} \right] & 
 T(1,4) & =  \left[ \begin {array}{ccc} 20&15&6\end {array} \right] \\ 
T(1,4) & =  \left[ \begin {array}{ccc} 15&20&15\end {array} \right] & 
 T(1,4) & =  \left[ \begin {array}{ccc} 6&15&20\end {array} \right] \\ 
T(1,4) & =  \left[ \begin {array}{ccc} 1&6&15\end {array} \right] & 
 T(1,5) & =  \left[ \begin {array}{cc} 1&6\end {array} \right] \\ 
T(1,6) & =  \left[ \begin {array}{c} 1\end {array} \right] & 
 T(2,2) & =  \left[ \begin {array}{c} 1\end {array} \right] \\ 
T(2,3) & =  \left[ \begin {array}{cc} 6&1\end {array} \right] & 
 T(2,4) & =  \left[ \begin {array}{ccc} 15&6&1\end {array} \right] \\ 
T(3,4) & =  \left[ \begin {array}{ccc} 1&0&0\\ 20&15&6\end {array} \right] & 
 T(3,4) & =  \left[ \begin {array}{ccc} 6&1&0\\ 15&20&15\end {array} \right] \\ 
T(3,4) & =  \left[ \begin {array}{ccc} 15&6&1\\ 6&15&20\end {array} \right] & 
 T(4,4) & =  \left[ \begin {array}{ccc} 1&0&0\\ 20&15&6\\ 1&6&15\end {array} \right] \\ 
T(4,4) & =  \left[ \begin {array}{ccc} 6&1&0\\ 15&20&15\\ 0&1&6\end {array} \right] & 
 T(4,4) & =  \left[ \begin {array}{ccc} 15&6&1\\ 6&15&20\\ 0&0&1\end {array} \right] \\ 
T(5,4) & =  \left[ \begin {array}{ccc} 20&15&6\\ 1&6&15\end {array} \right] & 
 T(5,4) & =  \left[ \begin {array}{ccc} 15&20&15\\ 0&1&6\end {array} \right] \\ 
T(5,4) & =  \left[ \begin {array}{ccc} 6&15&20\\ 0&0&1\end {array} \right] & 
 T(6,4) & =  \left[ \begin {array}{ccc} 1&6&15\end {array} \right] \\ 
T(6,5) & =  \left[ \begin {array}{cc} 1&6\end {array} \right] & 
 T(6,6) & =  \left[ \begin {array}{c} 1\end {array} \right] \\ 
\end{align*}

The essential class is: [4].
The essential class is of positive type.
An example is the path [4, 4, 4].
The essential class is not a simple loop.
This spectral range will include the interval $[21., 21.88819442]$.
The minimum comes from the loop $[4, 4]$.
The maximum comes from the loop $[4, 4]$.
These points will include points of local dimension [.976628124, 1.014334771].
The Spectral Range is contained in the range $[21.00000000, 21.91565240]$.
The minimum comes from the total row sub-norm of length 5. 
The maximum comes from the total row sup-norm of length 5. 
These points will have local dimension contained in [.975486976, 1.014334771].

There are 2 additional maximal loops.

Maximal Loop Class: [6].
The maximal loop class is a simple loop.
It's spectral radius is an isolated points of 1.
These points have local dimension 3.785578520.

Maximal Loop Class: [2].
The maximal loop class is a simple loop.
It's spectral radius is an isolated points of 1.
These points have local dimension 3.785578520.

\section{Minimal polynomial $3 x-1$ with $d_i \in [0, 2/21, 4/21, 2/7, 8/21, 10/21, 4/7, 2/3]$} 
\label{sec:6}
 
Consider $\varrho$, the root of $3 x-1$ and the maps $S_i(x) = \varrho x  + d_i$ with $d_{0} = 0$, $d_{1} = 2/21$, $d_{2} = 4/21$, $d_{3} = 2/7$, $d_{4} = 8/21$, $d_{5} = 10/21$, $d_{6} = 4/7$, and $d_{7} = 2/3$.
The probabilities are given by $p_{0} = 1/128$, $p_{1} = 7/128$, $p_{2} = 21/128$, $p_{3} = 35/128$, $p_{4} = 35/128$, $p_{5} = 21/128$, $p_{6} = 7/128$, and $p_{7} = 1/128$.
The reduced transition diagram has 9 reduced characteristic vectors.
The reduced characteristic vectors are:
\begin{itemize}
\item Reduced characteristic vector 1: $(1, (0))$ 
\item Reduced characteristic vector 2: $(2/7, (0))$ 
\item Reduced characteristic vector 3: $(2/7, (0, 2/7))$ 
\item Reduced characteristic vector 4: $(2/7, (0, 2/7, 4/7))$ 
\item Reduced characteristic vector 5: $(1/7, (0, 2/7, 4/7, 6/7))$ 
\item Reduced characteristic vector 6: $(1/7, (1/7, 3/7, 5/7))$ 
\item Reduced characteristic vector 7: $(2/7, (1/7, 3/7, 5/7))$ 
\item Reduced characteristic vector 8: $(2/7, (3/7, 5/7))$ 
\item Reduced characteristic vector 9: $(2/7, (5/7))$ 
\end{itemize}
See Figure \ref{fig:Pic6} for the transition diagram.
\begin{figure}[H]
\includegraphics[scale=0.5]{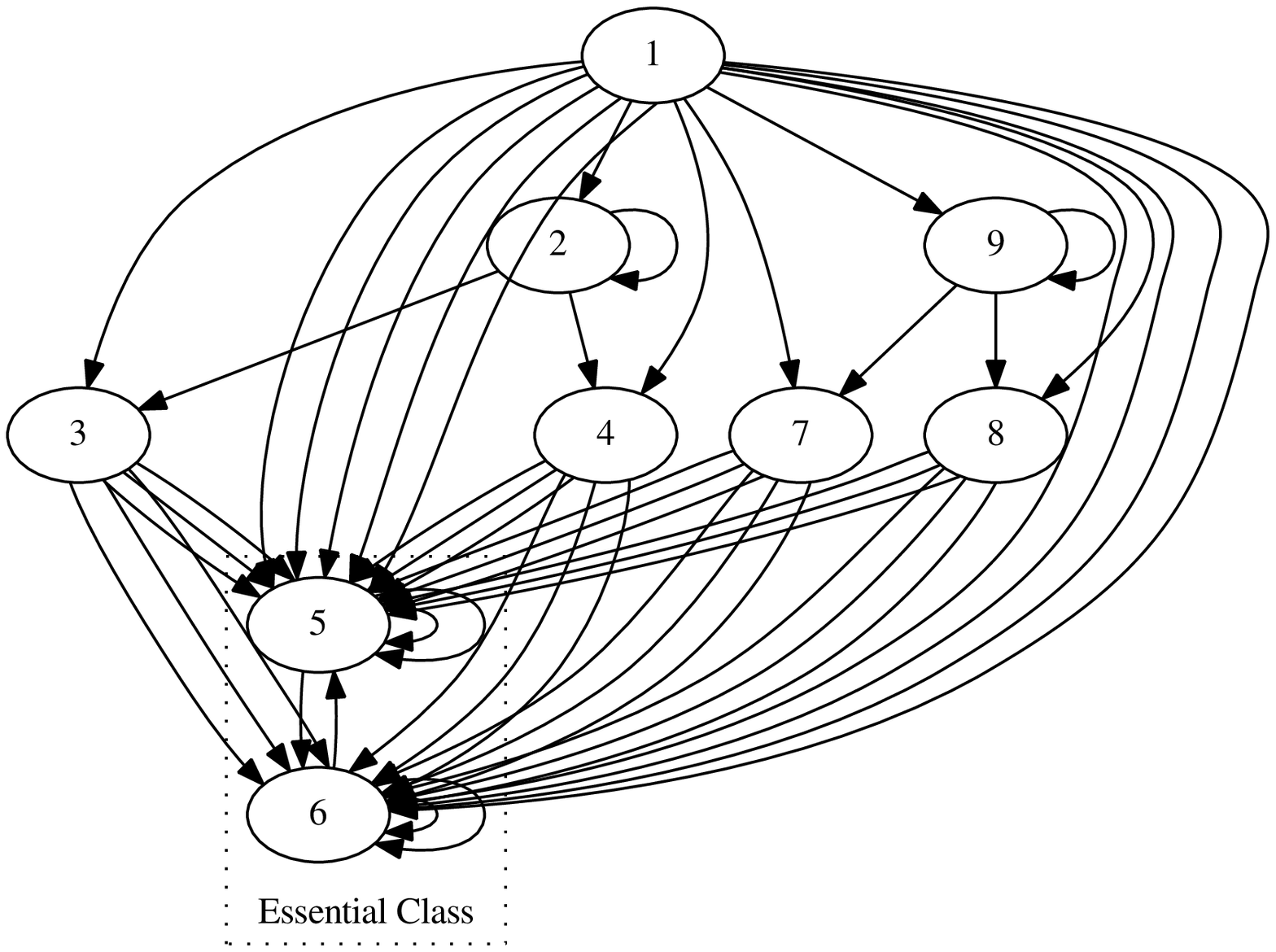}
\caption{$3 x-1$ with $d_i \in [0, 2/21, 4/21, 2/7, 8/21, 10/21, 4/7, 2/3]$, Full set + Essential class}
\label{fig:Pic6}
\end{figure}
This has transition matrices:
\begin{align*}
T(1,2) & =  \left[ \begin {array}{c} 1\end {array} \right] & 
 T(1,3) & =  \left[ \begin {array}{cc} 7&1\end {array} \right] \\ 
T(1,4) & =  \left[ \begin {array}{ccc} 21&7&1\end {array} \right] & 
 T(1,5) & =  \left[ \begin {array}{cccc} 35&21&7&1\end {array} \right] \\ 
T(1,6) & =  \left[ \begin {array}{ccc} 35&21&7\end {array} \right] & 
 T(1,5) & =  \left[ \begin {array}{cccc} 35&35&21&7\end {array} \right] \\ 
T(1,6) & =  \left[ \begin {array}{ccc} 35&35&21\end {array} \right] & 
 T(1,5) & =  \left[ \begin {array}{cccc} 21&35&35&21\end {array} \right] \\ 
T(1,6) & =  \left[ \begin {array}{ccc} 21&35&35\end {array} \right] & 
 T(1,5) & =  \left[ \begin {array}{cccc} 7&21&35&35\end {array} \right] \\ 
T(1,6) & =  \left[ \begin {array}{ccc} 7&21&35\end {array} \right] & 
 T(1,5) & =  \left[ \begin {array}{cccc} 1&7&21&35\end {array} \right] \\ 
T(1,7) & =  \left[ \begin {array}{ccc} 1&7&21\end {array} \right] & 
 T(1,8) & =  \left[ \begin {array}{cc} 1&7\end {array} \right] \\ 
T(1,9) & =  \left[ \begin {array}{c} 1\end {array} \right] & 
 T(2,2) & =  \left[ \begin {array}{c} 1\end {array} \right] \\ 
T(2,3) & =  \left[ \begin {array}{cc} 7&1\end {array} \right] & 
 T(2,4) & =  \left[ \begin {array}{ccc} 21&7&1\end {array} \right] \\ 
T(3,5) & =  \left[ \begin {array}{cccc} 1&0&0&0\\ 35&21&7&1\end {array} \right] & 
 T(3,6) & =  \left[ \begin {array}{ccc} 1&0&0\\ 35&21&7\end {array} \right] \\ 
T(3,5) & =  \left[ \begin {array}{cccc} 7&1&0&0\\ 35&35&21&7\end {array} \right] & 
 T(3,6) & =  \left[ \begin {array}{ccc} 7&1&0\\ 35&35&21\end {array} \right] \\ 
T(3,5) & =  \left[ \begin {array}{cccc} 21&7&1&0\\ 21&35&35&21\end {array} \right] & 
 T(3,6) & =  \left[ \begin {array}{ccc} 21&7&1\\ 21&35&35\end {array} \right] \\ 
T(4,5) & =  \left[ \begin {array}{cccc} 1&0&0&0\\ 35&21&7&1\\ 7&21&35&35\end {array} \right] & 
 T(4,6) & =  \left[ \begin {array}{ccc} 1&0&0\\ 35&21&7\\ 7&21&35\end {array} \right] \\ 
T(4,5) & =  \left[ \begin {array}{cccc} 7&1&0&0\\ 35&35&21&7\\ 1&7&21&35\end {array} \right] & 
 T(4,6) & =  \left[ \begin {array}{ccc} 7&1&0\\ 35&35&21\\ 1&7&21\end {array} \right] \\ 
T(4,5) & =  \left[ \begin {array}{cccc} 21&7&1&0\\ 21&35&35&21\\ 0&1&7&21\end {array} \right] & 
 T(4,6) & =  \left[ \begin {array}{ccc} 21&7&1\\ 21&35&35\\ 0&1&7\end {array} \right] \\ 
T(5,5) & =  \left[ \begin {array}{cccc} 1&0&0&0\\ 35&21&7&1\\ 7&21&35&35\\ 0&0&1&7\end {array} \right] & 
 T(5,6) & =  \left[ \begin {array}{ccc} 1&0&0\\ 35&21&7\\ 7&21&35\\ 0&0&1\end {array} \right] \\ 
T(5,5) & =  \left[ \begin {array}{cccc} 7&1&0&0\\ 35&35&21&7\\ 1&7&21&35\\ 0&0&0&1\end {array} \right] & 
 T(6,6) & =  \left[ \begin {array}{ccc} 7&1&0\\ 35&35&21\\ 1&7&21\end {array} \right] \\ 
T(6,5) & =  \left[ \begin {array}{cccc} 21&7&1&0\\ 21&35&35&21\\ 0&1&7&21\end {array} \right] & 
 T(6,6) & =  \left[ \begin {array}{ccc} 21&7&1\\ 21&35&35\\ 0&1&7\end {array} \right] \\ 
T(7,6) & =  \left[ \begin {array}{ccc} 7&1&0\\ 35&35&21\\ 1&7&21\end {array} \right] & 
 T(7,5) & =  \left[ \begin {array}{cccc} 21&7&1&0\\ 21&35&35&21\\ 0&1&7&21\end {array} \right] \\ 
T(7,6) & =  \left[ \begin {array}{ccc} 21&7&1\\ 21&35&35\\ 0&1&7\end {array} \right] & 
 T(7,5) & =  \left[ \begin {array}{cccc} 35&21&7&1\\ 7&21&35&35\\ 0&0&1&7\end {array} \right] \\ 
T(7,6) & =  \left[ \begin {array}{ccc} 35&21&7\\ 7&21&35\\ 0&0&1\end {array} \right] & 
 T(7,5) & =  \left[ \begin {array}{cccc} 35&35&21&7\\ 1&7&21&35\\ 0&0&0&1\end {array} \right] \\ 
T(8,6) & =  \left[ \begin {array}{ccc} 35&35&21\\ 1&7&21\end {array} \right] & 
 T(8,5) & =  \left[ \begin {array}{cccc} 21&35&35&21\\ 0&1&7&21\end {array} \right] \\ 
T(8,6) & =  \left[ \begin {array}{ccc} 21&35&35\\ 0&1&7\end {array} \right] & 
 T(8,5) & =  \left[ \begin {array}{cccc} 7&21&35&35\\ 0&0&1&7\end {array} \right] \\ 
T(8,6) & =  \left[ \begin {array}{ccc} 7&21&35\\ 0&0&1\end {array} \right] & 
 T(8,5) & =  \left[ \begin {array}{cccc} 1&7&21&35\\ 0&0&0&1\end {array} \right] \\ 
T(9,7) & =  \left[ \begin {array}{ccc} 1&7&21\end {array} \right] & 
 T(9,8) & =  \left[ \begin {array}{cc} 1&7\end {array} \right] \\ 
T(9,9) & =  \left[ \begin {array}{c} 1\end {array} \right] & 
 \end{align*}

The essential class is: [5, 6].
The essential class is of positive type.
An example is the path [6, 5, 5].
The essential class is not a simple loop.
This spectral range will include the interval $[42.38365876, 42.95601552]$.
The minimum comes from the loop $[5, 5, 5]$.
The maximum comes from the loop $[6, 6, 6]$.
These points will include points of local dimension [.993847946, 1.006057728].
The Spectral Range is contained in the range $[42.32331104, 42.96883688]$.
The minimum comes from the total row sub-norm of length 5. 
The maximum comes from the total row sup-norm of length 5. 
These points will have local dimension contained in [.993576302, 1.007354690].

There are 2 additional maximal loops.

Maximal Loop Class: [9].
The maximal loop class is a simple loop.
It's spectral radius is an isolated points of 1.
These points have local dimension 4.416508274.

Maximal Loop Class: [2].
The maximal loop class is a simple loop.
It's spectral radius is an isolated points of 1.
These points have local dimension 4.416508274.

\section{Minimal polynomial $3 x-1$ with $d_i \in [0, 1/12, 1/6, 1/4, 1/3, 5/12, 1/2, 7/12, 2/3]$} 
\label{sec:7}
 
Consider $\varrho$, the root of $3 x-1$ and the maps $S_i(x) = \varrho x  + d_i$ with $d_{0} = 0$, $d_{1} = 1/12$, $d_{2} = 1/6$, $d_{3} = 1/4$, $d_{4} = 1/3$, $d_{5} = 5/12$, $d_{6} = 1/2$, $d_{7} = 7/12$, and $d_{8} = 2/3$.
The probabilities are given by $p_{0} = 1/256$, $p_{1} = 1/32$, $p_{2} = 7/64$, $p_{3} = 7/32$, $p_{4} = 35/128$, $p_{5} = 7/32$, $p_{6} = 7/64$, $p_{7} = 1/32$, and $p_{8} = 1/256$.
The reduced transition diagram has 8 reduced characteristic vectors.
The reduced characteristic vectors are:
\begin{itemize}
\item Reduced characteristic vector 1: $(1, (0))$ 
\item Reduced characteristic vector 2: $(1/4, (0))$ 
\item Reduced characteristic vector 3: $(1/4, (0, 1/4))$ 
\item Reduced characteristic vector 4: $(1/4, (0, 1/4, 1/2))$ 
\item Reduced characteristic vector 5: $(1/4, (0, 1/4, 1/2, 3/4))$ 
\item Reduced characteristic vector 6: $(1/4, (1/4, 1/2, 3/4))$ 
\item Reduced characteristic vector 7: $(1/4, (1/2, 3/4))$ 
\item Reduced characteristic vector 8: $(1/4, (3/4))$ 
\end{itemize}
See Figure \ref{fig:Pic7} for the transition diagram.
\begin{figure}[H]
\includegraphics[scale=0.5]{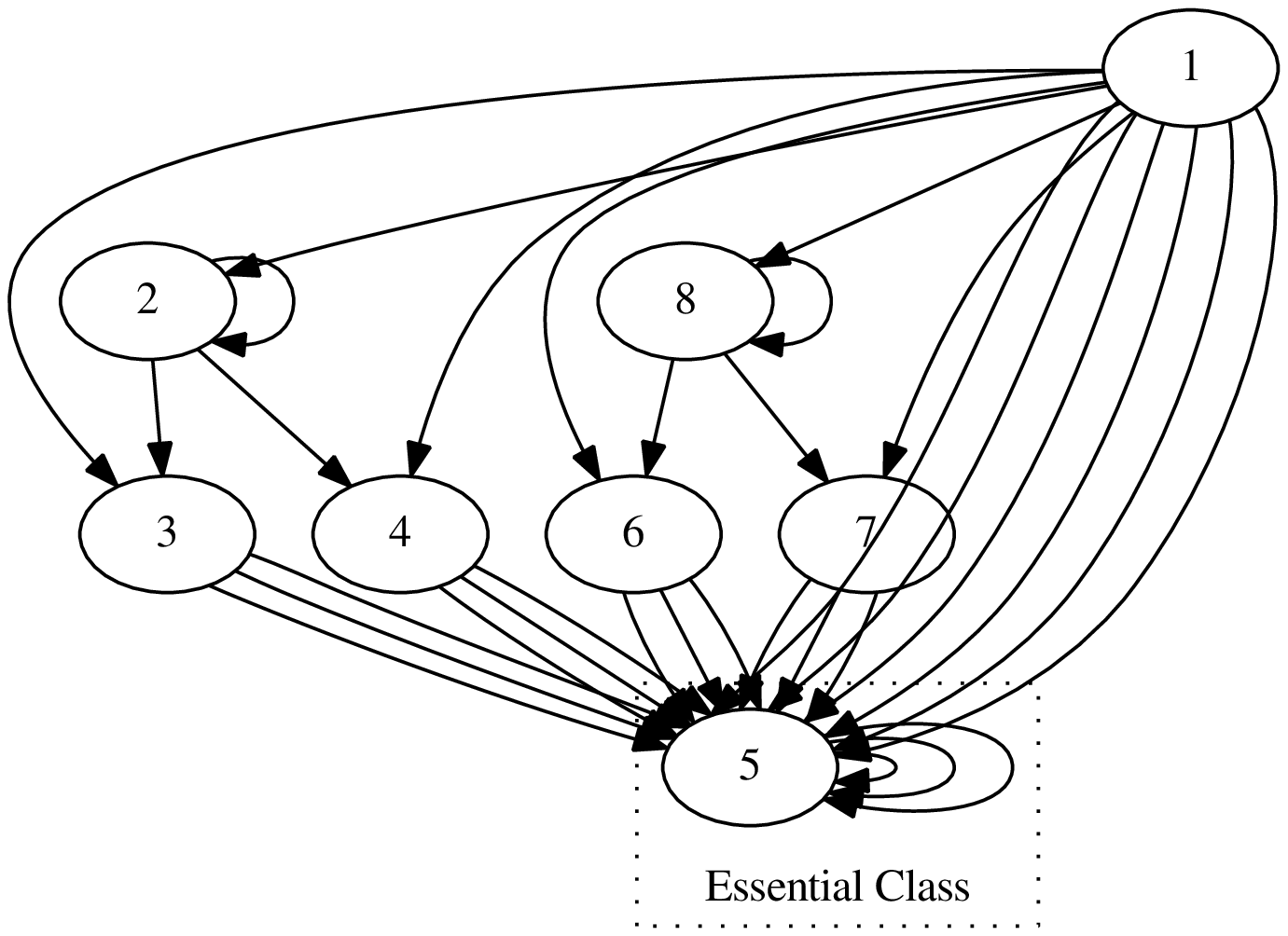}
\caption{$3 x-1$ with $d_i \in [0, 1/12, 1/6, 1/4, 1/3, 5/12, 1/2, 7/12, 2/3]$, Full set + Essential class}
\label{fig:Pic7}
\end{figure}
This has transition matrices:
\begin{align*}
T(1,2) & =  \left[ \begin {array}{c} 1\end {array} \right] & 
 T(1,3) & =  \left[ \begin {array}{cc} 8&1\end {array} \right] \\ 
T(1,4) & =  \left[ \begin {array}{ccc} 28&8&1\end {array} \right] & 
 T(1,5) & =  \left[ \begin {array}{cccc} 56&28&8&1\end {array} \right] \\ 
T(1,5) & =  \left[ \begin {array}{cccc} 70&56&28&8\end {array} \right] & 
 T(1,5) & =  \left[ \begin {array}{cccc} 56&70&56&28\end {array} \right] \\ 
T(1,5) & =  \left[ \begin {array}{cccc} 28&56&70&56\end {array} \right] & 
 T(1,5) & =  \left[ \begin {array}{cccc} 8&28&56&70\end {array} \right] \\ 
T(1,5) & =  \left[ \begin {array}{cccc} 1&8&28&56\end {array} \right] & 
 T(1,6) & =  \left[ \begin {array}{ccc} 1&8&28\end {array} \right] \\ 
T(1,7) & =  \left[ \begin {array}{cc} 1&8\end {array} \right] & 
 T(1,8) & =  \left[ \begin {array}{c} 1\end {array} \right] \\ 
T(2,2) & =  \left[ \begin {array}{c} 1\end {array} \right] & 
 T(2,3) & =  \left[ \begin {array}{cc} 8&1\end {array} \right] \\ 
T(2,4) & =  \left[ \begin {array}{ccc} 28&8&1\end {array} \right] & 
 T(3,5) & =  \left[ \begin {array}{cccc} 1&0&0&0\\ 56&28&8&1\end {array} \right] \\ 
T(3,5) & =  \left[ \begin {array}{cccc} 8&1&0&0\\ 70&56&28&8\end {array} \right] & 
 T(3,5) & =  \left[ \begin {array}{cccc} 28&8&1&0\\ 56&70&56&28\end {array} \right] \\ 
T(4,5) & =  \left[ \begin {array}{cccc} 1&0&0&0\\ 56&28&8&1\\ 28&56&70&56\end {array} \right] & 
 T(4,5) & =  \left[ \begin {array}{cccc} 8&1&0&0\\ 70&56&28&8\\ 8&28&56&70\end {array} \right] \\ 
T(4,5) & =  \left[ \begin {array}{cccc} 28&8&1&0\\ 56&70&56&28\\ 1&8&28&56\end {array} \right] & 
 T(5,5) & =  \left[ \begin {array}{cccc} 1&0&0&0\\ 56&28&8&1\\ 28&56&70&56\\ 0&1&8&28\end {array} \right] \\ 
T(5,5) & =  \left[ \begin {array}{cccc} 8&1&0&0\\ 70&56&28&8\\ 8&28&56&70\\ 0&0&1&8\end {array} \right] & 
 T(5,5) & =  \left[ \begin {array}{cccc} 28&8&1&0\\ 56&70&56&28\\ 1&8&28&56\\ 0&0&0&1\end {array} \right] \\ 
T(6,5) & =  \left[ \begin {array}{cccc} 56&28&8&1\\ 28&56&70&56\\ 0&1&8&28\end {array} \right] & 
 T(6,5) & =  \left[ \begin {array}{cccc} 70&56&28&8\\ 8&28&56&70\\ 0&0&1&8\end {array} \right] \\ 
T(6,5) & =  \left[ \begin {array}{cccc} 56&70&56&28\\ 1&8&28&56\\ 0&0&0&1\end {array} \right] & 
 T(7,5) & =  \left[ \begin {array}{cccc} 28&56&70&56\\ 0&1&8&28\end {array} \right] \\ 
T(7,5) & =  \left[ \begin {array}{cccc} 8&28&56&70\\ 0&0&1&8\end {array} \right] & 
 T(7,5) & =  \left[ \begin {array}{cccc} 1&8&28&56\\ 0&0&0&1\end {array} \right] \\ 
T(8,6) & =  \left[ \begin {array}{ccc} 1&8&28\end {array} \right] & 
 T(8,7) & =  \left[ \begin {array}{cc} 1&8\end {array} \right] \\ 
T(8,8) & =  \left[ \begin {array}{c} 1\end {array} \right] & 
 \end{align*}

The essential class is: [5].
The essential class is of positive type.
An example is the path [5, 5, 5].
The essential class is not a simple loop.
This spectral range will include the interval $[85.01281841, 85.78015987]$.
The minimum comes from the loop $[5, 5]$.
The maximum comes from the loop $[5, 5]$.
These points will include points of local dimension [.995246196, 1.003425326].
The Spectral Range is contained in the range $[85.01002497, 85.83207109]$.
The minimum comes from the total row sub-norm of length 5. 
The maximum comes from the total row sup-norm of length 5. 
These points will have local dimension contained in [.994695517, 1.003455237].

There are 2 additional maximal loops.

Maximal Loop Class: [8].
The maximal loop class is a simple loop.
It's spectral radius is an isolated points of 1.
These points have local dimension 5.047438027.

Maximal Loop Class: [2].
The maximal loop class is a simple loop.
It's spectral radius is an isolated points of 1.
These points have local dimension 5.047438027.

\section{Minimal polynomial $3 x-1$ with $d_i \in [0, 2/27, 4/27, 2/9, 8/27, 10/27, 4/9, 14/27, 16/27, 2/3]$} 
\label{sec:8}
 
Consider $\varrho$, the root of $3 x-1$ and the maps $S_i(x) = \varrho x  + d_i$ with $d_{0} = 0$, $d_{1} = 2/27$, $d_{2} = 4/27$, $d_{3} = 2/9$, $d_{4} = 8/27$, $d_{5} = 10/27$, $d_{6} = 4/9$, $d_{7} = 14/27$, $d_{8} = 16/27$, and $d_{9} = 2/3$.
The probabilities are given by $p_{0} = 1/512$, $p_{1} = 9/512$, $p_{2} = 9/128$, $p_{3} = 21/128$, $p_{4} = 63/256$, $p_{5} = 63/256$, $p_{6} = 21/128$, $p_{7} = 9/128$, $p_{8} = 9/512$, and $p_{9} = 1/512$.
The reduced transition diagram has 11 reduced characteristic vectors.
The reduced characteristic vectors are:
\begin{itemize}
\item Reduced characteristic vector 1: $(1, (0))$ 
\item Reduced characteristic vector 2: $(2/9, (0))$ 
\item Reduced characteristic vector 3: $(2/9, (0, 2/9))$ 
\item Reduced characteristic vector 4: $(2/9, (0, 2/9, 4/9))$ 
\item Reduced characteristic vector 5: $(2/9, (0, 2/9, 4/9, 2/3))$ 
\item Reduced characteristic vector 6: $(1/9, (0, 2/9, 4/9, 2/3, 8/9))$ 
\item Reduced characteristic vector 7: $(1/9, (1/9, 1/3, 5/9, 7/9))$ 
\item Reduced characteristic vector 8: $(2/9, (1/9, 1/3, 5/9, 7/9))$ 
\item Reduced characteristic vector 9: $(2/9, (1/3, 5/9, 7/9))$ 
\item Reduced characteristic vector 10: $(2/9, (5/9, 7/9))$ 
\item Reduced characteristic vector 11: $(2/9, (7/9))$ 
\end{itemize}
See Figure \ref{fig:Pic8} for the transition diagram.
\begin{figure}[H]
\includegraphics[scale=0.5]{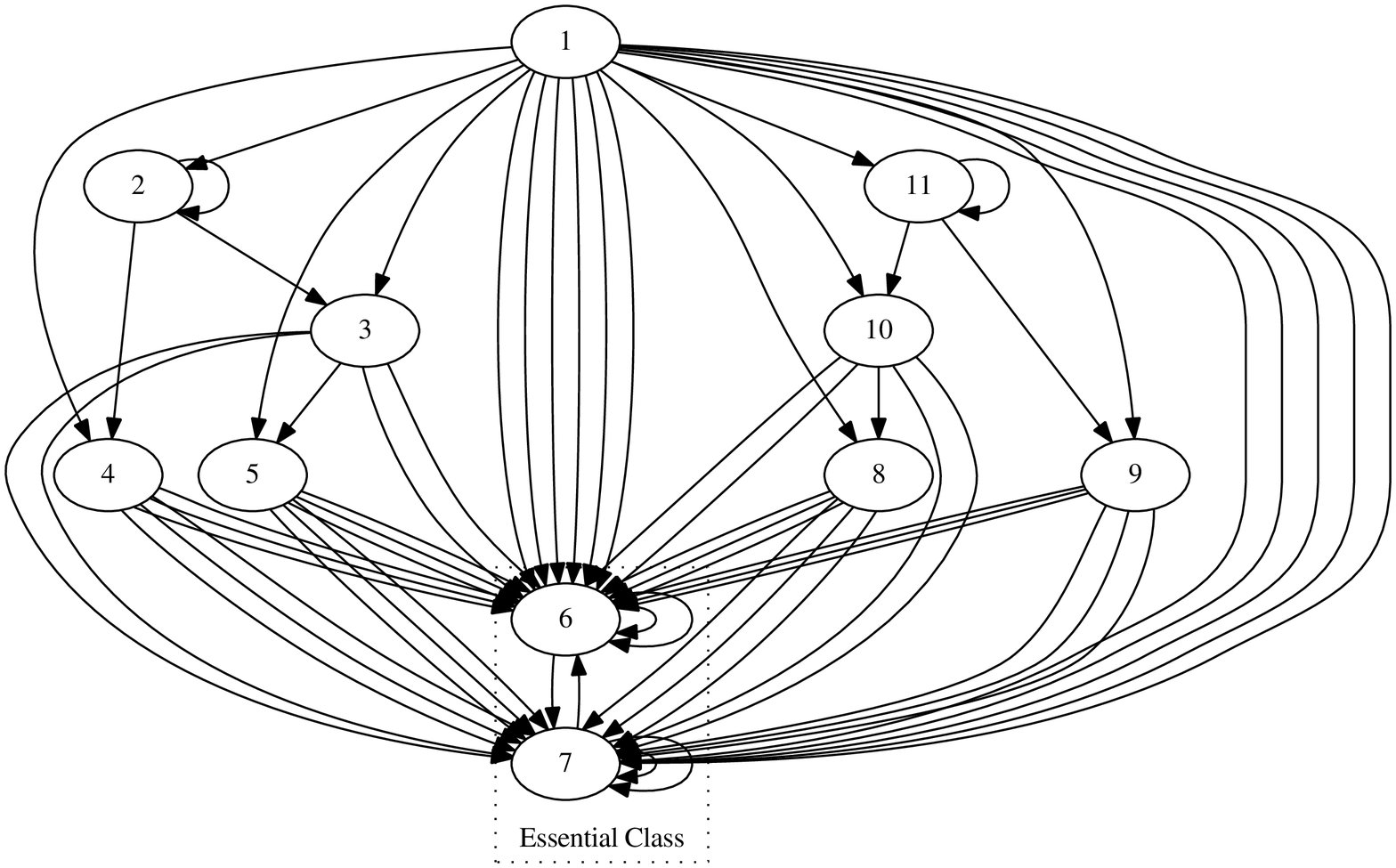}
\caption{$3 x-1$ with $d_i \in [0, 2/27, 4/27, 2/9, 8/27, 10/27, 4/9, 14/27, 16/27, 2/3]$, Full set + Essential class}
\label{fig:Pic8}
\end{figure}
This has transition matrices:
\begin{align*}
T(1,2) & =  \left[ \begin {array}{c} 1\end {array} \right] & 
 T(1,3) & =  \left[ \begin {array}{cc} 9&1\end {array} \right] \\ 
T(1,4) & =  \left[ \begin {array}{ccc} 36&9&1\end {array} \right] & 
 T(1,5) & =  \left[ \begin {array}{cccc} 84&36&9&1\end {array} \right] \\ 
T(1,6) & =  \left[ \begin {array}{ccccc} 126&84&36&9&1\end {array} \right] & 
 T(1,7) & =  \left[ \begin {array}{cccc} 126&84&36&9\end {array} \right] \\ 
T(1,6) & =  \left[ \begin {array}{ccccc} 126&126&84&36&9\end {array} \right] & 
 T(1,7) & =  \left[ \begin {array}{cccc} 126&126&84&36\end {array} \right] \\ 
T(1,6) & =  \left[ \begin {array}{ccccc} 84&126&126&84&36\end {array} \right] & 
 T(1,7) & =  \left[ \begin {array}{cccc} 84&126&126&84\end {array} \right] \\ 
T(1,6) & =  \left[ \begin {array}{ccccc} 36&84&126&126&84\end {array} \right] & 
 T(1,7) & =  \left[ \begin {array}{cccc} 36&84&126&126\end {array} \right] \\ 
T(1,6) & =  \left[ \begin {array}{ccccc} 9&36&84&126&126\end {array} \right] & 
 T(1,7) & =  \left[ \begin {array}{cccc} 9&36&84&126\end {array} \right] \\ 
T(1,6) & =  \left[ \begin {array}{ccccc} 1&9&36&84&126\end {array} \right] & 
 T(1,8) & =  \left[ \begin {array}{cccc} 1&9&36&84\end {array} \right] \\ 
T(1,9) & =  \left[ \begin {array}{ccc} 1&9&36\end {array} \right] & 
 T(1,10) & =  \left[ \begin {array}{cc} 1&9\end {array} \right] \\ 
T(1,11) & =  \left[ \begin {array}{c} 1\end {array} \right] & 
 T(2,2) & =  \left[ \begin {array}{c} 1\end {array} \right] \\ 
T(2,3) & =  \left[ \begin {array}{cc} 9&1\end {array} \right] & 
 T(2,4) & =  \left[ \begin {array}{ccc} 36&9&1\end {array} \right] \\ 
T(3,5) & =  \left[ \begin {array}{cccc} 1&0&0&0\\ 84&36&9&1\end {array} \right] & 
 T(3,6) & =  \left[ \begin {array}{ccccc} 9&1&0&0&0\\ 126&84&36&9&1\end {array} \right] \\ 
T(3,7) & =  \left[ \begin {array}{cccc} 9&1&0&0\\ 126&84&36&9\end {array} \right] & 
 T(3,6) & =  \left[ \begin {array}{ccccc} 36&9&1&0&0\\ 126&126&84&36&9\end {array} \right] \\ 
T(3,7) & =  \left[ \begin {array}{cccc} 36&9&1&0\\ 126&126&84&36\end {array} \right] & 
 T(4,6) & =  \left[ \begin {array}{ccccc} 1&0&0&0&0\\ 84&36&9&1&0\\ 84&126&126&84&36\end {array} \right] \\ 
T(4,7) & =  \left[ \begin {array}{cccc} 1&0&0&0\\ 84&36&9&1\\ 84&126&126&84\end {array} \right] & 
 T(4,6) & =  \left[ \begin {array}{ccccc} 9&1&0&0&0\\ 126&84&36&9&1\\ 36&84&126&126&84\end {array} \right] \\ 
T(4,7) & =  \left[ \begin {array}{cccc} 9&1&0&0\\ 126&84&36&9\\ 36&84&126&126\end {array} \right] & 
 T(4,6) & =  \left[ \begin {array}{ccccc} 36&9&1&0&0\\ 126&126&84&36&9\\ 9&36&84&126&126\end {array} \right] \\ 
T(4,7) & =  \left[ \begin {array}{cccc} 36&9&1&0\\ 126&126&84&36\\ 9&36&84&126\end {array} \right] & 
 T(5,6) & =  \left[ \begin {array}{ccccc} 1&0&0&0&0\\ 84&36&9&1&0\\ 84&126&126&84&36\\ 1&9&36&84&126\end {array} \right] \\ 
T(5,7) & =  \left[ \begin {array}{cccc} 1&0&0&0\\ 84&36&9&1\\ 84&126&126&84\\ 1&9&36&84\end {array} \right] & 
 T(5,6) & =  \left[ \begin {array}{ccccc} 9&1&0&0&0\\ 126&84&36&9&1\\ 36&84&126&126&84\\ 0&1&9&36&84\end {array} \right] \\ 
T(5,7) & =  \left[ \begin {array}{cccc} 9&1&0&0\\ 126&84&36&9\\ 36&84&126&126\\ 0&1&9&36\end {array} \right] & 
 T(5,6) & =  \left[ \begin {array}{ccccc} 36&9&1&0&0\\ 126&126&84&36&9\\ 9&36&84&126&126\\ 0&0&1&9&36\end {array} \right] \\ 
T(5,7) & =  \left[ \begin {array}{cccc} 36&9&1&0\\ 126&126&84&36\\ 9&36&84&126\\ 0&0&1&9\end {array} \right] & 
 T(6,6) & =  \left[ \begin {array}{ccccc} 1&0&0&0&0\\ 84&36&9&1&0\\ 84&126&126&84&36\\ 1&9&36&84&126\\ 0&0&0&1&9\end {array} \right] \\ 
T(6,7) & =  \left[ \begin {array}{cccc} 1&0&0&0\\ 84&36&9&1\\ 84&126&126&84\\ 1&9&36&84\\ 0&0&0&1\end {array} \right] & 
 T(6,6) & =  \left[ \begin {array}{ccccc} 9&1&0&0&0\\ 126&84&36&9&1\\ 36&84&126&126&84\\ 0&1&9&36&84\\ 0&0&0&0&1\end {array} \right] \\ 
T(7,7) & =  \left[ \begin {array}{cccc} 9&1&0&0\\ 126&84&36&9\\ 36&84&126&126\\ 0&1&9&36\end {array} \right] & 
 T(7,6) & =  \left[ \begin {array}{ccccc} 36&9&1&0&0\\ 126&126&84&36&9\\ 9&36&84&126&126\\ 0&0&1&9&36\end {array} \right] \\ 
T(7,7) & =  \left[ \begin {array}{cccc} 36&9&1&0\\ 126&126&84&36\\ 9&36&84&126\\ 0&0&1&9\end {array} \right] & 
 T(8,7) & =  \left[ \begin {array}{cccc} 9&1&0&0\\ 126&84&36&9\\ 36&84&126&126\\ 0&1&9&36\end {array} \right] \\ 
T(8,6) & =  \left[ \begin {array}{ccccc} 36&9&1&0&0\\ 126&126&84&36&9\\ 9&36&84&126&126\\ 0&0&1&9&36\end {array} \right] & 
 T(8,7) & =  \left[ \begin {array}{cccc} 36&9&1&0\\ 126&126&84&36\\ 9&36&84&126\\ 0&0&1&9\end {array} \right] \\ 
T(8,6) & =  \left[ \begin {array}{ccccc} 84&36&9&1&0\\ 84&126&126&84&36\\ 1&9&36&84&126\\ 0&0&0&1&9\end {array} \right] & 
 T(8,7) & =  \left[ \begin {array}{cccc} 84&36&9&1\\ 84&126&126&84\\ 1&9&36&84\\ 0&0&0&1\end {array} \right] \\ 
T(8,6) & =  \left[ \begin {array}{ccccc} 126&84&36&9&1\\ 36&84&126&126&84\\ 0&1&9&36&84\\ 0&0&0&0&1\end {array} \right] & 
 T(9,7) & =  \left[ \begin {array}{cccc} 126&84&36&9\\ 36&84&126&126\\ 0&1&9&36\end {array} \right] \\ 
T(9,6) & =  \left[ \begin {array}{ccccc} 126&126&84&36&9\\ 9&36&84&126&126\\ 0&0&1&9&36\end {array} \right] & 
 T(9,7) & =  \left[ \begin {array}{cccc} 126&126&84&36\\ 9&36&84&126\\ 0&0&1&9\end {array} \right] \\ 
T(9,6) & =  \left[ \begin {array}{ccccc} 84&126&126&84&36\\ 1&9&36&84&126\\ 0&0&0&1&9\end {array} \right] & 
 T(9,7) & =  \left[ \begin {array}{cccc} 84&126&126&84\\ 1&9&36&84\\ 0&0&0&1\end {array} \right] \\ 
T(9,6) & =  \left[ \begin {array}{ccccc} 36&84&126&126&84\\ 0&1&9&36&84\\ 0&0&0&0&1\end {array} \right] & 
 T(10,7) & =  \left[ \begin {array}{cccc} 36&84&126&126\\ 0&1&9&36\end {array} \right] \\ 
T(10,6) & =  \left[ \begin {array}{ccccc} 9&36&84&126&126\\ 0&0&1&9&36\end {array} \right] & 
 T(10,7) & =  \left[ \begin {array}{cccc} 9&36&84&126\\ 0&0&1&9\end {array} \right] \\ 
T(10,6) & =  \left[ \begin {array}{ccccc} 1&9&36&84&126\\ 0&0&0&1&9\end {array} \right] & 
 T(10,8) & =  \left[ \begin {array}{cccc} 1&9&36&84\\ 0&0&0&1\end {array} \right] \\ 
T(11,9) & =  \left[ \begin {array}{ccc} 1&9&36\end {array} \right] & 
 T(11,10) & =  \left[ \begin {array}{cc} 1&9\end {array} \right] \\ 
T(11,11) & =  \left[ \begin {array}{c} 1\end {array} \right] & 
 \end{align*}

The essential class is: [6, 7].
The essential class is of positive type.
An example is the path [7, 6, 6].
The essential class is not a simple loop.
This spectral range will include the interval $[170.4173766, 170.9186289]$.
The minimum comes from the loop $[7, 7, 7]$.
The maximum comes from the loop $[6, 6, 6]$.
These points will include points of local dimension [.998657167, 1.001330544].
The Spectral Range is contained in the range $[170.3465777, 170.9404180]$.
The minimum comes from the total row sub-norm of length 5. 
The maximum comes from the total row sup-norm of length 5. 
These points will have local dimension contained in [.998541136, 1.001708776].

There are 2 additional maximal loops.

Maximal Loop Class: [11].
The maximal loop class is a simple loop.
It's spectral radius is an isolated points of 1.
These points have local dimension 5.678367781.

Maximal Loop Class: [2].
The maximal loop class is a simple loop.
It's spectral radius is an isolated points of 1.
These points have local dimension 5.678367781.

\section{Minimal polynomial $3 x-1$ with $d_i \in [0, 1/15, 2/15, 1/5, 4/15, 1/3, 2/5, 7/15, 8/15, 3/5, 2/3]$} 
\label{sec:9}
 
Consider $\varrho$, the root of $3 x-1$ and the maps $S_i(x) = \varrho x  + d_i$ with $d_{0} = 0$, $d_{1} = 1/15$, $d_{2} = 2/15$, $d_{3} = 1/5$, $d_{4} = 4/15$, $d_{5} = 1/3$, $d_{6} = 2/5$, $d_{7} = 7/15$, $d_{8} = 8/15$, $d_{9} = 3/5$, and $d_{10} = 2/3$.
The probabilities are given by $p_{0} = 1/1024$, $p_{1} = 5/512$, $p_{2} = 45/1024$, $p_{3} = 15/128$, $p_{4} = 105/512$, $p_{5} = 63/256$, $p_{6} = 105/512$, $p_{7} = 15/128$, $p_{8} = 45/1024$, $p_{9} = 5/512$, and $p_{10} = 1/1024$.
The reduced transition diagram has 10 reduced characteristic vectors.
The reduced characteristic vectors are:
\begin{itemize}
\item Reduced characteristic vector 1: $(1, (0))$ 
\item Reduced characteristic vector 2: $(1/5, (0))$ 
\item Reduced characteristic vector 3: $(1/5, (0, 1/5))$ 
\item Reduced characteristic vector 4: $(1/5, (0, 1/5, 2/5))$ 
\item Reduced characteristic vector 5: $(1/5, (0, 1/5, 2/5, 3/5))$ 
\item Reduced characteristic vector 6: $(1/5, (0, 1/5, 2/5, 3/5, 4/5))$ 
\item Reduced characteristic vector 7: $(1/5, (1/5, 2/5, 3/5, 4/5))$ 
\item Reduced characteristic vector 8: $(1/5, (2/5, 3/5, 4/5))$ 
\item Reduced characteristic vector 9: $(1/5, (3/5, 4/5))$ 
\item Reduced characteristic vector 10: $(1/5, (4/5))$ 
\end{itemize}
See Figure \ref{fig:Pic9} for the transition diagram.
\begin{figure}[H]
\includegraphics[scale=0.5]{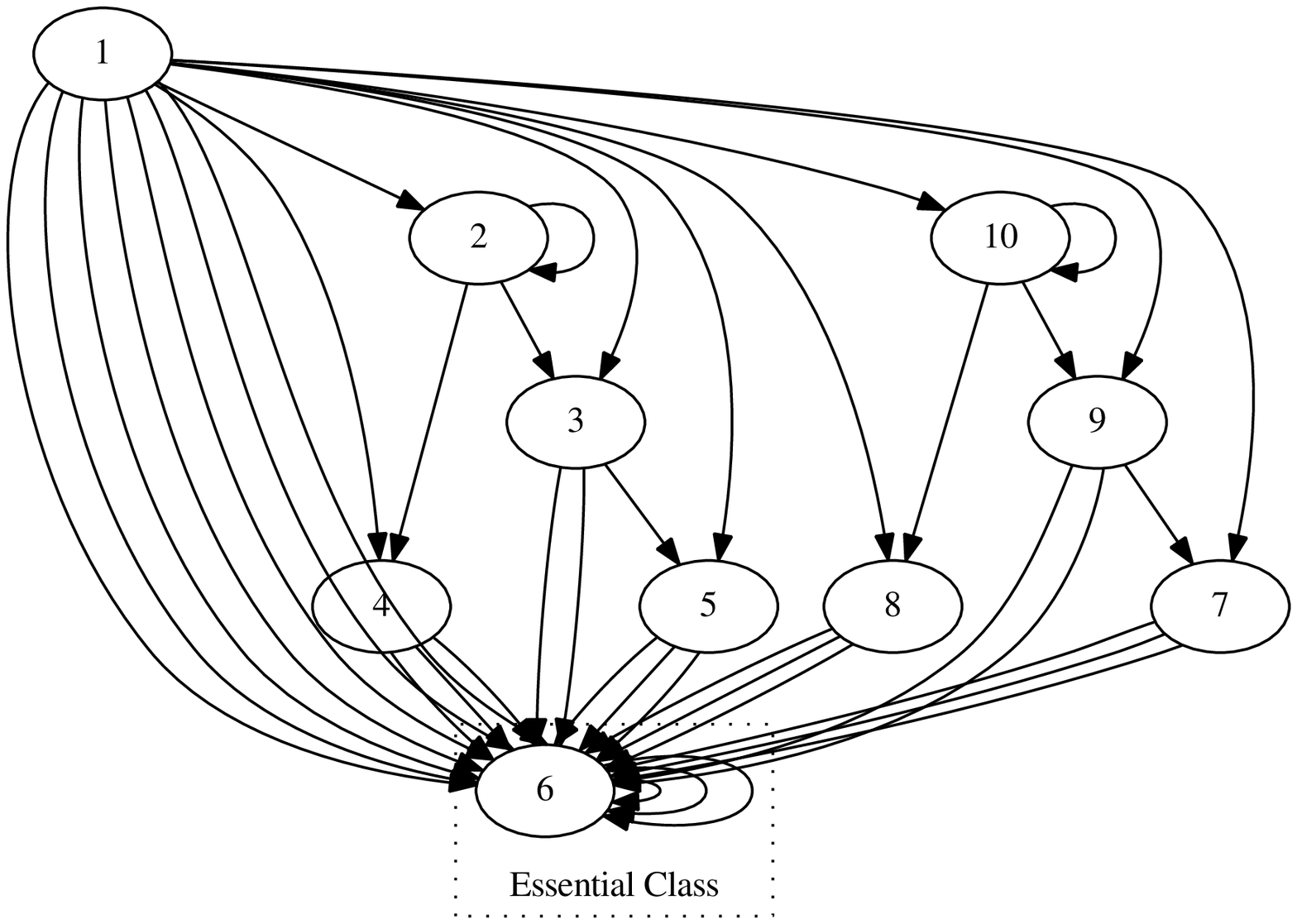}
\caption{$3 x-1$ with $d_i \in [0, 1/15, 2/15, 1/5, 4/15, 1/3, 2/5, 7/15, 8/15, 3/5, 2/3]$, Full set + Essential class}
\label{fig:Pic9}
\end{figure}
This has transition matrices:
\begin{align*}
T(1,2) & =  \left[ \begin {array}{c} 1\end {array} \right] & 
 T(1,3) & =  \left[ \begin {array}{cc} 10&1\end {array} \right] \\ 
T(1,4) & =  \left[ \begin {array}{ccc} 45&10&1\end {array} \right] & 
 T(1,5) & =  \left[ \begin {array}{cccc} 120&45&10&1\end {array} \right] \\ 
T(1,6) & =  \left[ \begin {array}{ccccc} 210&120&45&10&1\end {array} \right] & 
 T(1,6) & =  \left[ \begin {array}{ccccc} 252&210&120&45&10\end {array} \right] \\ 
T(1,6) & =  \left[ \begin {array}{ccccc} 210&252&210&120&45\end {array} \right] & 
 T(1,6) & =  \left[ \begin {array}{ccccc} 120&210&252&210&120\end {array} \right] \\ 
T(1,6) & =  \left[ \begin {array}{ccccc} 45&120&210&252&210\end {array} \right] & 
 T(1,6) & =  \left[ \begin {array}{ccccc} 10&45&120&210&252\end {array} \right] \\ 
T(1,6) & =  \left[ \begin {array}{ccccc} 1&10&45&120&210\end {array} \right] & 
 T(1,7) & =  \left[ \begin {array}{cccc} 1&10&45&120\end {array} \right] \\ 
T(1,8) & =  \left[ \begin {array}{ccc} 1&10&45\end {array} \right] & 
 T(1,9) & =  \left[ \begin {array}{cc} 1&10\end {array} \right] \\ 
T(1,10) & =  \left[ \begin {array}{c} 1\end {array} \right] & 
 T(2,2) & =  \left[ \begin {array}{c} 1\end {array} \right] \\ 
T(2,3) & =  \left[ \begin {array}{cc} 10&1\end {array} \right] & 
 T(2,4) & =  \left[ \begin {array}{ccc} 45&10&1\end {array} \right] \\ 
T(3,5) & =  \left[ \begin {array}{cccc} 1&0&0&0\\ 120&45&10&1\end {array} \right] & 
 T(3,6) & =  \left[ \begin {array}{ccccc} 10&1&0&0&0\\ 210&120&45&10&1\end {array} \right] \\ 
T(3,6) & =  \left[ \begin {array}{ccccc} 45&10&1&0&0\\ 252&210&120&45&10\end {array} \right] & 
 T(4,6) & =  \left[ \begin {array}{ccccc} 1&0&0&0&0\\ 120&45&10&1&0\\ 210&252&210&120&45\end {array} \right] \\ 
T(4,6) & =  \left[ \begin {array}{ccccc} 10&1&0&0&0\\ 210&120&45&10&1\\ 120&210&252&210&120\end {array} \right] & 
 T(4,6) & =  \left[ \begin {array}{ccccc} 45&10&1&0&0\\ 252&210&120&45&10\\ 45&120&210&252&210\end {array} \right] \\ 
T(5,6) & =  \left[ \begin {array}{ccccc} 1&0&0&0&0\\ 120&45&10&1&0\\ 210&252&210&120&45\\ 10&45&120&210&252\end {array} \right] & 
 T(5,6) & =  \left[ \begin {array}{ccccc} 10&1&0&0&0\\ 210&120&45&10&1\\ 120&210&252&210&120\\ 1&10&45&120&210\end {array} \right] \\ 
T(5,6) & =  \left[ \begin {array}{ccccc} 45&10&1&0&0\\ 252&210&120&45&10\\ 45&120&210&252&210\\ 0&1&10&45&120\end {array} \right] & 
 T(6,6) & =  \left[ \begin {array}{ccccc} 1&0&0&0&0\\ 120&45&10&1&0\\ 210&252&210&120&45\\ 10&45&120&210&252\\ 0&0&1&10&45\end {array} \right] \\ 
T(6,6) & =  \left[ \begin {array}{ccccc} 10&1&0&0&0\\ 210&120&45&10&1\\ 120&210&252&210&120\\ 1&10&45&120&210\\ 0&0&0&1&10\end {array} \right] & 
 T(6,6) & =  \left[ \begin {array}{ccccc} 45&10&1&0&0\\ 252&210&120&45&10\\ 45&120&210&252&210\\ 0&1&10&45&120\\ 0&0&0&0&1\end {array} \right] \\ 
T(7,6) & =  \left[ \begin {array}{ccccc} 120&45&10&1&0\\ 210&252&210&120&45\\ 10&45&120&210&252\\ 0&0&1&10&45\end {array} \right] & 
 T(7,6) & =  \left[ \begin {array}{ccccc} 210&120&45&10&1\\ 120&210&252&210&120\\ 1&10&45&120&210\\ 0&0&0&1&10\end {array} \right] \\ 
T(7,6) & =  \left[ \begin {array}{ccccc} 252&210&120&45&10\\ 45&120&210&252&210\\ 0&1&10&45&120\\ 0&0&0&0&1\end {array} \right] & 
 T(8,6) & =  \left[ \begin {array}{ccccc} 210&252&210&120&45\\ 10&45&120&210&252\\ 0&0&1&10&45\end {array} \right] \\ 
T(8,6) & =  \left[ \begin {array}{ccccc} 120&210&252&210&120\\ 1&10&45&120&210\\ 0&0&0&1&10\end {array} \right] & 
 T(8,6) & =  \left[ \begin {array}{ccccc} 45&120&210&252&210\\ 0&1&10&45&120\\ 0&0&0&0&1\end {array} \right] \\ 
T(9,6) & =  \left[ \begin {array}{ccccc} 10&45&120&210&252\\ 0&0&1&10&45\end {array} \right] & 
 T(9,6) & =  \left[ \begin {array}{ccccc} 1&10&45&120&210\\ 0&0&0&1&10\end {array} \right] \\ 
T(9,7) & =  \left[ \begin {array}{cccc} 1&10&45&120\\ 0&0&0&1\end {array} \right] & 
 T(10,8) & =  \left[ \begin {array}{ccc} 1&10&45\end {array} \right] \\ 
T(10,9) & =  \left[ \begin {array}{cc} 1&10\end {array} \right] & 
 T(10,10) & =  \left[ \begin {array}{c} 1\end {array} \right] \\ 
\end{align*}

The essential class is: [6].
The essential class is of positive type.
An example is the path [6, 6, 6].
The essential class is not a simple loop.
This spectral range will include the interval $[341.0358264, 341.7004294]$.
The minimum comes from the loop $[6, 6]$.
The maximum comes from the loop $[6, 6]$.
These points will include points of local dimension [.999021586, 1.000793713].
The Spectral Range is contained in the range $[341.0278683, 341.7686680]$.
The minimum comes from the total row sub-norm of length 5. 
The maximum comes from the total row sup-norm of length 5. 
These points will have local dimension contained in [.998839826, 1.000814954].

There are 2 additional maximal loops.

Maximal Loop Class: [10].
The maximal loop class is a simple loop.
It's spectral radius is an isolated points of 1.
These points have local dimension 6.309297535.

Maximal Loop Class: [2].
The maximal loop class is a simple loop.
It's spectral radius is an isolated points of 1.
These points have local dimension 6.309297535.

\section{Minimal polynomial $x^3+x-1$ with $d_i \in [0, 1-\varrho]$} 
\label{sec:10}
 
Consider $\varrho$, the root of $x^3+x-1$ and the maps $S_i(x) = \varrho x  + d_i$ with $d_{0} = 0$, and $d_{1} = 1-\varrho$.
The probabilities are uniform.
The reduced transition diagram has 152 reduced characteristic vectors.

The number of reduced characteristic vectors in the essential class is 46.
The essential class is of positive type.
An example is the path [129, 92, 91, 114, 112, 127, 140, 150, 75, 75, 75, 75, 75].
The essential class is not a simple loop.
This spectral range will include the interval $[1.324717958, 1.380277569]$.
The minimum comes from the loop $[75, 75]$.
The maximum comes from the loop $[77, 98, 100, 116, 120, 91, 114, 112, 128, 124, 77]$.
These points will include points of local dimension [.9702219294, 1.077705433].
The Spectral Range is contained in the range $[1.113263577, 1.446125550]$.
The minimum comes from the column sub-norm on the subset ${{2, 3, 4}}$ of length 15. 
The maximum comes from the total column sup-norm of length 10. 
These points will have local dimension contained in [.8483019061, 1.532658865].

There are 4 additional maximal loops.

Maximal Loop Class: [40, 41, 45, 46, 47, 51, 52, 53, 54, 61, 62, 63, 70, 71, 72, 82, 84, 86, 96, 106, 122, 135, 146].
The reduced characteristic vectors are:
\begin{itemize}
\item Reduced characteristic vector 40: $(\varrho^2-2 \varrho+1, (0, \varrho^2+\varrho-1, 2 \varrho-1, \varrho^2, \varrho, \varrho^2+2 \varrho-1))$ 
\item Reduced characteristic vector 41: $(-2 \varrho^2+1, (0, \varrho^2-2 \varrho+1, 2 \varrho^2-\varrho, \varrho^2, 2 \varrho^2-2 \varrho+1, \varrho^2-\varrho+1, 2 \varrho^2))$ 
\item Reduced characteristic vector 45: $(\varrho^2+\varrho-1, (0, \varrho^2-2 \varrho+1, 1-\varrho, -\varrho^2+1, -2 \varrho+2, \varrho^2-\varrho+1))$ 
\item Reduced characteristic vector 46: $(-2 \varrho^2+1, (0, \varrho^2+\varrho-1, 2 \varrho^2-\varrho, \varrho^2, \varrho, \varrho^2-\varrho+1, 2 \varrho^2))$ 
\item Reduced characteristic vector 47: $(\varrho^2+\varrho-1, (-2 \varrho^2+1, -\varrho^2+\varrho, 1-\varrho, -\varrho^2+1, -2 \varrho^2+\varrho+1, -\varrho^2-\varrho+2))$ 
\item Reduced characteristic vector 51: $(-2 \varrho^2+1, (0, \varrho^2+\varrho-1, 2 \varrho-1, \varrho^2, \varrho, \varrho^2+2 \varrho-1, 2 \varrho^2))$ 
\item Reduced characteristic vector 52: $(\varrho^2-2 \varrho+1, (-2 \varrho^2+1, -\varrho^2+\varrho, -2 \varrho^2+2 \varrho, -\varrho^2+1, -2 \varrho^2+\varrho+1, -\varrho^2+2 \varrho))$ 
\item Reduced characteristic vector 53: $(\varrho^2+\varrho-1, (0, -2 \varrho^2+1, -\varrho^2+\varrho, -\varrho^2+1, \varrho, -2 \varrho^2+\varrho+1))$ 
\item Reduced characteristic vector 54: $(\varrho^2-2 \varrho+1, (0, \varrho^2+\varrho-1, -\varrho^2+\varrho, 2 \varrho-1, \varrho, \varrho^2+2 \varrho-1, -\varrho^2+2 \varrho))$ 
\item Reduced characteristic vector 61: $(\varrho^2+\varrho-1, (-2 \varrho^2+1, -\varrho^2+\varrho, -2 \varrho^2-\varrho+2, -\varrho^2+1, -2 \varrho^2+\varrho+1, -\varrho^2-\varrho+2))$ 
\item Reduced characteristic vector 62: $(\varrho^2-2 \varrho+1, (0, -\varrho^2+\varrho, 2 \varrho-1, -\varrho^2+1, \varrho, -\varrho^2+2 \varrho))$ 
\item Reduced characteristic vector 63: $(\varrho^2+\varrho-1, (0, \varrho^2-2 \varrho+1, 1-\varrho, \varrho^2, -2 \varrho+2, \varrho^2-\varrho+1))$ 
\item Reduced characteristic vector 70: $(\varrho^2-2 \varrho+1, (0, -2 \varrho^2+1, -\varrho^2+\varrho, -\varrho^2+1, \varrho, -2 \varrho^2+\varrho+1, -\varrho^2+2 \varrho))$ 
\item Reduced characteristic vector 71: $(\varrho^2+\varrho-1, (\varrho^2-2 \varrho+1, -\varrho^2-2 \varrho+2, 1-\varrho, -2 \varrho+2, \varrho^2-\varrho+1, -\varrho^2-\varrho+2))$ 
\item Reduced characteristic vector 72: $(\varrho^2+\varrho-1, (0, -2 \varrho^2+1, -\varrho^2+\varrho, 1-\varrho, -\varrho^2+1, -2 \varrho^2+\varrho+1, -\varrho^2-\varrho+2))$ 
\item Reduced characteristic vector 82: $(\varrho^2+\varrho-1, (\varrho^2-2 \varrho+1, 1-\varrho, 2 \varrho^2-2 \varrho+1, -2 \varrho+2, \varrho^2-\varrho+1))$ 
\item Reduced characteristic vector 84: $(\varrho^2+\varrho-1, (0, -2 \varrho^2+1, 1-\varrho, -\varrho^2+1, \varrho^2-\varrho+1, -\varrho^2-\varrho+2))$ 
\item Reduced characteristic vector 86: $(\varrho^2+\varrho-1, (-2 \varrho^2+1, -\varrho^2+\varrho, -3 \varrho^2+\varrho+1, -\varrho^2+1, -2 \varrho^2+\varrho+1))$ 
\item Reduced characteristic vector 96: $(\varrho^2+\varrho-1, (0, \varrho^2-2 \varrho+1, 1-\varrho, -\varrho^2+1, -2 \varrho+2, \varrho^2-\varrho+1, -\varrho^2-\varrho+2))$ 
\item Reduced characteristic vector 106: $(-\varrho^2+\varrho, (0, \varrho^2-2 \varrho+1, 1-\varrho, \varrho^2, \varrho, \varrho^2-\varrho+1))$ 
\item Reduced characteristic vector 122: $(1-\varrho, (0, \varrho^2+\varrho-1, -\varrho^2+\varrho, \varrho^2, -\varrho^2+1, \varrho))$ 
\item Reduced characteristic vector 135: $(3 \varrho^2-2 \varrho, (-2 \varrho^2+1, -\varrho^2+\varrho, -2 \varrho^2+2 \varrho, -\varrho^2+1, -2 \varrho^2+\varrho+1, -\varrho^2+2 \varrho))$ 
\item Reduced characteristic vector 146: $(3 \varrho-2, (\varrho^2-2 \varrho+1, -\varrho^2-2 \varrho+2, 1-\varrho, -2 \varrho+2, \varrho^2-\varrho+1, -\varrho^2-\varrho+2))$ 
\end{itemize}
See Figure \ref{fig:Pic10} for the transition diagram.
\begin{figure}[H]
\includegraphics[scale=0.5]{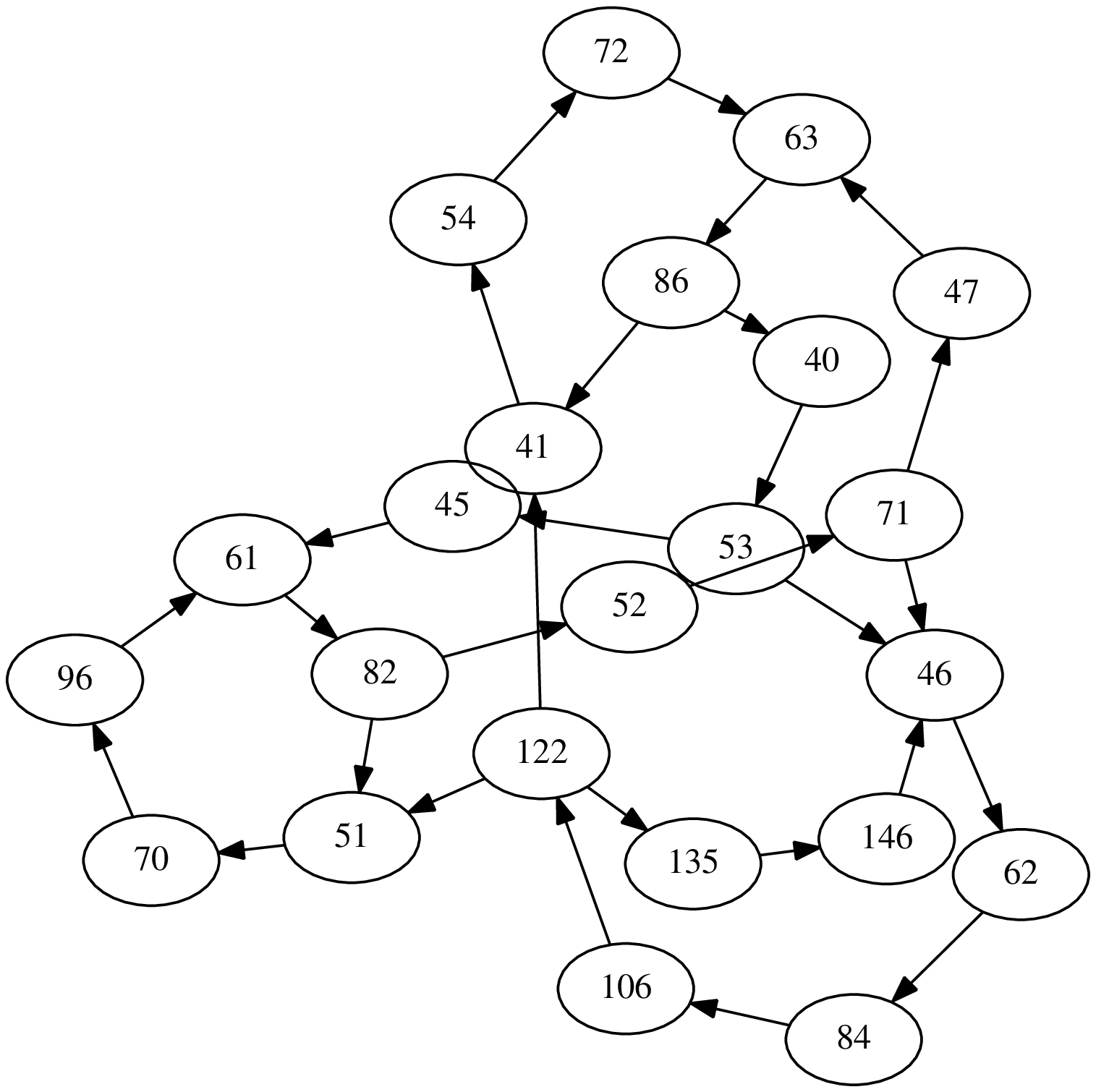}
\caption{$x^3+x-1$ with $d_i \in [0, 1-\varrho]$, Subloop [40, 41, 45, 46, 47, 51, 52, 53, 54, 61, 62, 63, 70, 71, 72, 82, 84, 86, 96, 106, 122, 135, 146]}
\label{fig:Pic10}
\end{figure}
This has transition matrices:
\begin{align*}
T(40,53) & =  \left[ \begin {array}{cccccc} 1&0&0&0&0&0\\ 0&0&1&0&0&0\\ 0&1&0&1&0&0\\ 0&0&1&0&1&0\\ 0&0&0&1&0&0\\ 0&0&0&0&0&1\end {array} \right] \\ 
T(41,54) & =  \left[ \begin {array}{ccccccc} 1&0&0&0&0&0&0\\ 0&1&0&0&0&0&0\\ 0&0&0&1&0&0&0\\ 0&0&1&0&1&0&0\\ 0&0&0&1&0&1&0\\ 0&0&0&0&1&0&0\\ 0&0&0&0&0&0&1\end {array} \right] \\ 
T(45,61) & =  \left[ \begin {array}{cccccc} 1&0&0&0&0&0\\ 0&1&0&0&0&0\\ 1&0&0&1&0&0\\ 0&0&1&0&0&1\\ 0&0&0&1&0&0\\ 0&0&0&0&1&0\end {array} \right] \\ 
T(46,62) & =  \left[ \begin {array}{cccccc} 1&0&0&0&0&0\\ 0&1&0&0&0&0\\ 0&0&1&0&0&0\\ 0&1&0&0&1&0\\ 0&0&0&1&0&0\\ 0&0&0&0&1&0\\ 0&0&0&0&0&1\end {array} \right] \\ 
T(47,63) & =  \left[ \begin {array}{cccccc} 0&1&0&0&0&0\\ 0&0&1&0&0&0\\ 1&0&0&1&0&0\\ 0&0&1&0&0&1\\ 0&0&0&0&1&0\\ 0&0&0&0&0&1\end {array} \right] \\ 
T(51,70) & =  \left[ \begin {array}{ccccccc} 1&0&0&0&0&0&0\\ 0&0&1&0&0&0&0\\ 0&1&0&1&0&0&0\\ 0&0&1&0&1&0&0\\ 0&0&0&1&0&0&0\\ 0&0&0&0&0&1&0\\ 0&0&0&0&0&0&1\end {array} \right] \\ 
T(52,71) & =  \left[ \begin {array}{cccccc} 1&0&0&0&0&0\\ 0&0&1&0&0&0\\ 0&1&0&1&0&0\\ 0&0&1&0&1&0\\ 0&0&0&1&0&0\\ 0&0&0&0&0&1\end {array} \right] \\ 
T(53,45) & =  \left[ \begin {array}{cccccc} 1&0&0&0&0&0\\ 0&1&0&0&0&0\\ 0&0&1&0&0&0\\ 0&0&1&0&0&1\\ 0&0&0&1&0&0\\ 0&0&0&0&1&0\end {array} \right] \\ 
T(53,46) & =  \left[ \begin {array}{ccccccc} 0&1&0&0&0&0&0\\ 0&0&1&0&0&0&0\\ 1&0&0&1&0&0&0\\ 0&0&0&1&0&0&1\\ 0&0&0&0&1&0&0\\ 0&0&0&0&0&1&0\end {array} \right] \\ 
T(54,72) & =  \left[ \begin {array}{ccccccc} 1&0&0&0&0&0&0\\ 0&0&1&0&0&0&0\\ 0&0&0&1&0&0&0\\ 0&1&0&0&1&0&0\\ 0&0&0&0&1&0&0\\ 0&0&0&0&0&1&0\\ 0&0&0&0&0&0&1\end {array} \right] \\ 
T(61,82) & =  \left[ \begin {array}{ccccc} 1&0&0&0&0\\ 0&1&0&0&0\\ 1&0&1&0&0\\ 0&1&0&0&1\\ 0&0&0&1&0\\ 0&0&0&0&1\end {array} \right] \\ 
T(62,84) & =  \left[ \begin {array}{cccccc} 1&0&0&0&0&0\\ 0&0&1&0&0&0\\ 0&1&0&1&0&0\\ 0&0&1&0&1&0\\ 0&0&0&1&0&0\\ 0&0&0&0&0&1\end {array} \right] \\ 
T(63,86) & =  \left[ \begin {array}{ccccc} 1&0&0&0&0\\ 0&1&0&0&0\\ 1&0&0&1&0\\ 0&0&1&0&1\\ 0&0&0&1&0\\ 0&0&0&0&1\end {array} \right] \\ 
T(70,96) & =  \left[ \begin {array}{ccccccc} 1&0&0&0&0&0&0\\ 0&1&0&0&0&0&0\\ 0&0&1&0&0&0&0\\ 0&0&1&0&0&1&0\\ 0&0&0&1&0&0&0\\ 0&0&0&0&1&0&0\\ 0&0&0&0&0&0&1\end {array} \right] \\ 
T(71,46) & =  \left[ \begin {array}{ccccccc} 0&1&0&0&0&0&0\\ 0&0&1&0&0&0&0\\ 1&0&0&1&0&0&0\\ 0&0&0&1&0&0&1\\ 0&0&0&0&1&0&0\\ 0&0&0&0&0&1&0\end {array} \right] \\ 
T(71,47) & =  \left[ \begin {array}{cccccc} 0&1&0&0&0&0\\ 0&0&1&0&0&0\\ 1&0&0&1&0&0\\ 0&0&0&1&0&0\\ 0&0&0&0&1&0\\ 0&0&0&0&0&1\end {array} \right] \\ 
T(72,63) & =  \left[ \begin {array}{cccccc} 1&0&0&0&0&0\\ 0&1&0&0&0&0\\ 0&0&1&0&0&0\\ 1&0&0&1&0&0\\ 0&0&1&0&0&1\\ 0&0&0&0&1&0\\ 0&0&0&0&0&1\end {array} \right] \\ 
T(82,51) & =  \left[ \begin {array}{ccccccc} 0&1&0&0&0&0&0\\ 1&0&0&1&0&0&0\\ 0&0&1&0&0&1&0\\ 0&0&0&1&0&0&1\\ 0&0&0&0&1&0&0\end {array} \right] \\ 
T(82,52) & =  \left[ \begin {array}{cccccc} 0&1&0&0&0&0\\ 1&0&0&1&0&0\\ 0&0&1&0&0&1\\ 0&0&0&1&0&0\\ 0&0&0&0&1&0\end {array} \right] & 
 T(84,106) & =  \left[ \begin {array}{cccccc} 1&0&0&0&0&0\\ 0&1&0&0&0&0\\ 1&0&0&1&0&0\\ 0&0&1&0&0&1\\ 0&0&0&0&1&0\\ 0&0&0&0&0&1\end {array} \right] \\ 
T(86,40) & =  \left[ \begin {array}{cccccc} 0&1&0&0&0&0\\ 0&0&1&0&0&0\\ 1&0&0&1&0&0\\ 0&0&1&0&0&1\\ 0&0&0&0&1&0\end {array} \right] & 
 T(86,41) & =  \left[ \begin {array}{ccccccc} 0&0&1&0&0&0&0\\ 1&0&0&1&0&0&0\\ 0&1&0&0&1&0&0\\ 0&0&0&1&0&0&1\\ 0&0&0&0&0&1&0\end {array} \right] \\ 
T(96,61) & =  \left[ \begin {array}{cccccc} 1&0&0&0&0&0\\ 0&1&0&0&0&0\\ 1&0&0&1&0&0\\ 0&0&1&0&0&1\\ 0&0&0&1&0&0\\ 0&0&0&0&1&0\\ 0&0&0&0&0&1\end {array} \right] \\ 
T(106,122) & =  \left[ \begin {array}{cccccc} 1&0&0&0&0&0\\ 0&1&0&0&0&0\\ 1&0&0&1&0&0\\ 0&0&1&0&0&1\\ 0&0&0&0&1&0\\ 0&0&0&0&0&1\end {array} \right] \\ 
T(122,51) & =  \left[ \begin {array}{ccccccc} 0&1&0&0&0&0&0\\ 0&0&1&0&0&0&0\\ 1&0&0&1&0&0&0\\ 0&0&1&0&0&1&0\\ 0&0&0&1&0&0&1\\ 0&0&0&0&1&0&0\end {array} \right] \\ 
T(122,135) & =  \left[ \begin {array}{cccccc} 0&1&0&0&0&0\\ 0&0&1&0&0&0\\ 1&0&0&1&0&0\\ 0&0&1&0&0&1\\ 0&0&0&1&0&0\\ 0&0&0&0&1&0\end {array} \right] \\ 
T(122,41) & =  \left[ \begin {array}{ccccccc} 0&0&1&0&0&0&0\\ 1&0&0&1&0&0&0\\ 0&1&0&0&1&0&0\\ 0&0&0&1&0&0&1\\ 0&0&0&0&1&0&0\\ 0&0&0&0&0&1&0\end {array} \right] \\ 
T(135,146) & =  \left[ \begin {array}{cccccc} 1&0&0&0&0&0\\ 0&0&1&0&0&0\\ 0&1&0&1&0&0\\ 0&0&1&0&1&0\\ 0&0&0&1&0&0\\ 0&0&0&0&0&1\end {array} \right] \\ 
T(146,46) & =  \left[ \begin {array}{ccccccc} 0&1&0&0&0&0&0\\ 0&0&1&0&0&0&0\\ 1&0&0&1&0&0&0\\ 0&0&0&1&0&0&1\\ 0&0&0&0&1&0&0\\ 0&0&0&0&0&1&0\end {array} \right] \\ 
\end{align*}
The maximal loop class is of positive type.
An example is the path [82, 52, 71, 46, 62, 84, 106, 122, 135].
The maximal loop class is not a simple loop.
This spectral range will include the interval $[1.368888709, 1.380277569]$.
The minimum comes from the loop $[41, 54, 72, 63, 86, 41]$.
The maximum comes from the loop $[40, 53, 45, 61, 82, 52, 71, 47, 63, 86, 40]$.
These points will include points of local dimension [.9702219294, .9918974638].
The Spectral Range is contained in the range $[1.212508617, 1.424590656]$.
The minimum comes from the total column sub-norm of length 15. 
The maximum comes from the total column sup-norm of length 15. 
These points will have local dimension contained in [.8875527238, 1.309253535].

Maximal Loop Class: [22, 30].
The reduced characteristic vectors are:
\begin{itemize}
\item Reduced characteristic vector 22: $(\varrho^2+\varrho-1, (-2 \varrho^2+1, -\varrho^2+\varrho, -\varrho^2+1, -2 \varrho^2+\varrho+1))$ 
\item Reduced characteristic vector 30: $(\varrho^2+\varrho-1, (\varrho^2-2 \varrho+1, 1-\varrho, -2 \varrho+2, \varrho^2-\varrho+1))$ 
\end{itemize}
See Figure \ref{fig:Pic11} for the transition diagram.
\begin{figure}[H]
\includegraphics[scale=0.5]{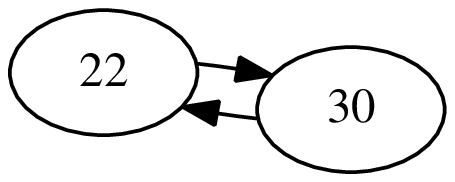}
\caption{$x^3+x-1$ with $d_i \in [0, 1-\varrho]$, Subloop [22, 30]}
\label{fig:Pic11}
\end{figure}
This has transition matrices:
\begin{align*}
T(22,30) & =  \left[ \begin {array}{cccc} 1&0&0&0\\ 0&1&0&0\\ 0&1&0&1\\ 0&0&1&0\end {array} \right] & 
 T(30,22) & =  \left[ \begin {array}{cccc} 0&1&0&0\\ 1&0&1&0\\ 0&0&1&0\\ 0&0&0&1\end {array} \right] \\ 
\end{align*}
The maximal loop class is a simple loop.
It's spectral radius is an isolated points of 1.380277569.
These points have local dimension .9702219294.

Maximal Loop Class: [4].
The reduced characteristic vectors are:
\begin{itemize}
\item Reduced characteristic vector 4: $(\varrho^2, (-\varrho^2+1))$ 
\end{itemize}
See Figure \ref{fig:Pic12} for the transition diagram.
\begin{figure}[H]
\includegraphics[scale=0.5]{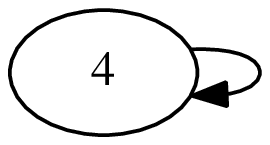}
\caption{$x^3+x-1$ with $d_i \in [0, 1-\varrho]$, Subloop [4]}
\label{fig:Pic12}
\end{figure}
This has transition matrices:
\begin{align*}
T(4,4) & =  \left[ \begin {array}{c} 1\end {array} \right] & 
 \end{align*}
The maximal loop class is a simple loop.
It's spectral radius is an isolated points of 1.
These points have local dimension 1.813357990.

Maximal Loop Class: [2].
The reduced characteristic vectors are:
\begin{itemize}
\item Reduced characteristic vector 2: $(\varrho^2, (0))$ 
\end{itemize}
See Figure \ref{fig:Pic13} for the transition diagram.
\begin{figure}[H]
\includegraphics[scale=0.5]{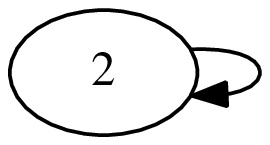}
\caption{$x^3+x-1$ with $d_i \in [0, 1-\varrho]$, Subloop [2]}
\label{fig:Pic13}
\end{figure}
This has transition matrices:
\begin{align*}
T(2,2) & =  \left[ \begin {array}{c} 1\end {array} \right] & 
 \end{align*}
The maximal loop class is a simple loop.
It's spectral radius is an isolated points of 1.
These points have local dimension 1.813357990.

\section{Minimal polynomial $x^3+x^2-1$ with $d_i \in [0, 1-\varrho]$} 
\label{sec:11}
 
Consider $\varrho$, the root of $x^3+x^2-1$ and the maps $S_i(x) = \varrho x  + d_i$ with $d_{0} = 0$, and $d_{1} = 1-\varrho$.
The probabilities are uniform.
The reduced transition diagram has 1809 reduced characteristic vectors.

The number of reduced characteristic vectors in the essential class is 1207.
The essential class is of positive type.
An example is the path [683, 631, 688, 624, 683, 631, 688, 624, 683, 631, 688, 624, 683, 631, 688, 624, 683, 631, 688, 624, 683, 631, 689, 625, 684, 632, 690, 509].
The essential class is not a simple loop.
This spectral range will include the interval $[1.506135680, 1.510626480]$.
The minimum comes from the loop $[1782, 1782]$.
The maximum comes from the loop $[624, 683, 631, 688, 624]$.
These points will include points of local dimension [.997948624, 1.008536243].
The Spectral Range is contained in the range $[1.148698355, 1.620656597]$.
The minimum comes from the column sub-norm on the subset ${{4, 5, 6, 7, 8, 12, 13, 16}}$ of length 15. 
The maximum comes from the total column sup-norm of length 10. 
These points will have local dimension contained in [.747923631, 1.971972204].

There are 5 additional maximal loops.

Maximal Loop Class: [388, 454, 460, 524, 531, 590].
The reduced characteristic vectors are:
\begin{itemize}
\item Reduced characteristic vector 388: $(-3 \varrho^2-3 \varrho+4, (0, -3 \varrho^2+\varrho+1, \varrho^2+2 \varrho-2, -\varrho^2+\varrho, 3 \varrho^2+2 \varrho-3, 3 \varrho-2, \varrho^2+\varrho-1, -2 \varrho^2+2 \varrho, 2 \varrho^2+3 \varrho-3, -\varrho^2+1, 2 \varrho-1, \varrho^2, -2 \varrho^2+\varrho+1, 2 \varrho^2+2 \varrho-2, -\varrho^2+3 \varrho-1, \varrho, \varrho^2+3 \varrho-2, -2 \varrho^2+2, 2 \varrho^2+\varrho-1, -\varrho^2+2 \varrho, 3 \varrho^2+3 \varrho-3))$ 
\item Reduced characteristic vector 454: $(-3 \varrho^2-3 \varrho+4, (0, 4 \varrho^2+\varrho-3, \varrho^2+2 \varrho-2, 5 \varrho^2+3 \varrho-5, 2 \varrho^2-1, 3 \varrho^2+2 \varrho-3, 4 \varrho^2-2, \varrho^2+\varrho-1, 5 \varrho^2+2 \varrho-4, 2 \varrho^2+3 \varrho-3, 3 \varrho^2+\varrho-2, 4 \varrho^2+3 \varrho-4, \varrho^2, 5 \varrho^2+\varrho-3, 2 \varrho^2+2 \varrho-2, 3 \varrho^2-1, \varrho, 4 \varrho^2+2 \varrho-3, 2 \varrho^2+\varrho-1, 6 \varrho^2+2 \varrho-4, 3 \varrho^2+3 \varrho-3))$ 
\item Reduced characteristic vector 460: $(4 \varrho^2+\varrho-3, (0, -3 \varrho^2-3 \varrho+4, \varrho^2-2 \varrho+1, -2 \varrho^2-\varrho+2, -\varrho^2-3 \varrho+3, -4 \varrho^2-2 \varrho+4, 1-\varrho, -3 \varrho^2+2, -2 \varrho^2-2 \varrho+3, -\varrho^2+1, -2 \varrho+2, -3 \varrho^2-\varrho+3, \varrho^2, -2 \varrho^2-3 \varrho+4, -\varrho^2-\varrho+2, \varrho, -3 \varrho^2-2 \varrho+4, \varrho^2-\varrho+1, -2 \varrho^2+2, -\varrho^2-2 \varrho+3, -4 \varrho^2-\varrho+4))$ 
\item Reduced characteristic vector 524: $(4 \varrho^2+\varrho-3, (0, -3 \varrho^2+\varrho+1, -2 \varrho^2-\varrho+2, -5 \varrho^2+3, -\varrho^2+\varrho, -4 \varrho^2-2 \varrho+4, -3 \varrho^2+2, -2 \varrho^2+2 \varrho, -5 \varrho^2-\varrho+4, -\varrho^2+1, -4 \varrho^2+\varrho+2, -3 \varrho^2-\varrho+3, -2 \varrho^2+\varrho+1, -\varrho^2-\varrho+2, -4 \varrho^2+3, \varrho, -3 \varrho^2+2 \varrho+1, -2 \varrho^2+2, -5 \varrho^2+\varrho+3, -\varrho^2+2 \varrho, -4 \varrho^2-\varrho+4))$ 
\item Reduced characteristic vector 531: $(-3 \varrho^2+\varrho+1, (0, 4 \varrho^2+\varrho-3, \varrho^2-2 \varrho+1, \varrho^2+2 \varrho-2, 2 \varrho^2-1, 3 \varrho^2+2 \varrho-3, 1-\varrho, 4 \varrho^2-2, \varrho^2+\varrho-1, 2 \varrho^2-\varrho, -\varrho^2+1, 3 \varrho^2+\varrho-2, \varrho^2, 2 \varrho^2+2 \varrho-2, -\varrho^2-\varrho+2, 3 \varrho^2-1, \varrho, 4 \varrho^2+2 \varrho-3, \varrho^2-\varrho+1, 2 \varrho^2+\varrho-1, 3 \varrho^2-\varrho))$ 
\item Reduced characteristic vector 590: $(-3 \varrho^2+\varrho+1, (0, \varrho^2-2 \varrho+1, 2 \varrho^2-1, -\varrho^2-3 \varrho+3, 3 \varrho^2-2 \varrho, 1-\varrho, 4 \varrho^2-2, \varrho^2-3 \varrho+2, 2 \varrho^2-\varrho, -2 \varrho+2, 4 \varrho^2-\varrho-1, \varrho^2, 2 \varrho^2-2 \varrho+1, -\varrho^2-\varrho+2, 3 \varrho^2-1, -3 \varrho+3, \varrho^2-\varrho+1, 2 \varrho^2-3 \varrho+2, 2 \varrho^2+\varrho-1, -\varrho^2-2 \varrho+3, 3 \varrho^2-\varrho))$ 
\end{itemize}
See Figure \ref{fig:Pic14} for the transition diagram.
\begin{figure}[H]
\includegraphics[scale=0.5]{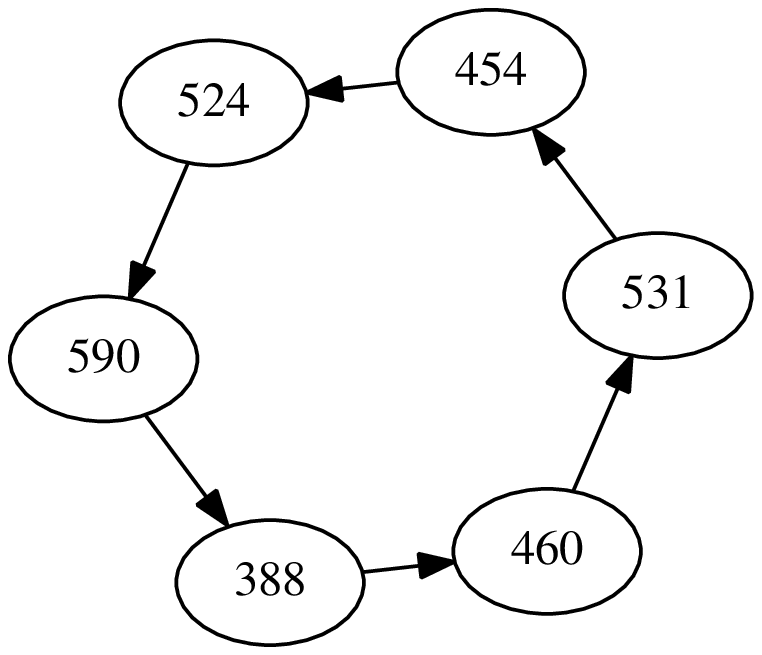}
\caption{$x^3+x^2-1$ with $d_i \in [0, 1-\varrho]$, Subloop [388, 454, 460, 524, 531, 590]}
\label{fig:Pic14}
\end{figure}
This has transition matrices:
\begin{align*}
T(388,460) & =  \left[ \begin {array}{ccccccccccccccccccccc} 1&0&0&0&0&0&0&0&0&0&0&0&0&0&0&0&0&0&0&0&0\\ 0&0&1&0&0&0&0&0&0&0&0&0&0&0&0&0&0&0&0&0&0\\ 0&0&0&1&0&0&0&0&0&0&0&0&0&0&0&0&0&0&0&0&0\\ 0&0&0&0&0&0&1&0&0&0&0&0&0&0&0&0&0&0&0&0&0\\ 0&0&0&0&0&0&0&1&0&0&0&0&0&0&0&0&0&0&0&0&0\\ 0&1&0&0&0&0&0&0&1&0&0&0&0&0&0&0&0&0&0&0&0\\ 0&0&0&1&0&0&0&0&0&1&0&0&0&0&0&0&0&0&0&0&0\\ 0&0&0&0&1&0&0&0&0&0&1&0&0&0&0&0&0&0&0&0&0\\ 0&0&0&0&0&1&0&0&0&0&0&1&0&0&0&0&0&0&0&0&0\\ 0&0&0&0&0&0&1&0&0&0&0&0&1&0&0&0&0&0&0&0&0\\ 0&0&0&0&0&0&0&0&1&0&0&0&0&0&1&0&0&0&0&0&0\\ 0&0&0&0&0&0&0&0&0&1&0&0&0&0&0&1&0&0&0&0&0\\ 0&0&0&0&0&0&0&0&0&0&1&0&0&0&0&0&0&1&0&0&0\\ 0&0&0&0&0&0&0&0&0&0&0&1&0&0&0&0&0&0&1&0&0\\ 0&0&0&0&0&0&0&0&0&0&0&0&0&1&0&0&0&0&0&1&0\\ 0&0&0&0&0&0&0&0&0&0&0&0&0&0&1&0&0&0&0&0&0\\ 0&0&0&0&0&0&0&0&0&0&0&0&0&0&0&0&1&0&0&0&0\\ 0&0&0&0&0&0&0&0&0&0&0&0&0&0&0&0&0&1&0&0&0\\ 0&0&0&0&0&0&0&0&0&0&0&0&0&0&0&0&0&0&1&0&0\\ 0&0&0&0&0&0&0&0&0&0&0&0&0&0&0&0&0&0&0&1&0\\ 0&0&0&0&0&0&0&0&0&0&0&0&0&0&0&0&0&0&0&0&1\end {array} \right] \\ 
T(454,524) & =  \left[ \begin {array}{ccccccccccccccccccccc} 1&0&0&0&0&0&0&0&0&0&0&0&0&0&0&0&0&0&0&0&0\\ 0&1&0&0&0&0&0&0&0&0&0&0&0&0&0&0&0&0&0&0&0\\ 0&0&1&0&0&0&0&0&0&0&0&0&0&0&0&0&0&0&0&0&0\\ 0&0&0&1&0&0&0&0&0&0&0&0&0&0&0&0&0&0&0&0&0\\ 0&0&0&0&1&0&0&0&0&0&0&0&0&0&0&0&0&0&0&0&0\\ 0&0&0&0&0&0&1&0&0&0&0&0&0&0&0&0&0&0&0&0&0\\ 0&1&0&0&0&0&0&1&0&0&0&0&0&0&0&0&0&0&0&0&0\\ 0&0&1&0&0&0&0&0&0&1&0&0&0&0&0&0&0&0&0&0&0\\ 0&0&0&1&0&0&0&0&0&0&1&0&0&0&0&0&0&0&0&0&0\\ 0&0&0&0&0&1&0&0&0&0&0&1&0&0&0&0&0&0&0&0&0\\ 0&0&0&0&0&0&1&0&0&0&0&0&1&0&0&0&0&0&0&0&0\\ 0&0&0&0&0&0&0&0&1&0&0&0&0&0&1&0&0&0&0&0&0\\ 0&0&0&0&0&0&0&0&0&1&0&0&0&0&0&1&0&0&0&0&0\\ 0&0&0&0&0&0&0&0&0&0&1&0&0&0&0&0&1&0&0&0&0\\ 0&0&0&0&0&0&0&0&0&0&0&1&0&0&0&0&0&1&0&0&0\\ 0&0&0&0&0&0&0&0&0&0&0&0&1&0&0&0&0&0&0&1&0\\ 0&0&0&0&0&0&0&0&0&0&0&0&0&1&0&0&0&0&0&0&0\\ 0&0&0&0&0&0&0&0&0&0&0&0&0&0&1&0&0&0&0&0&0\\ 0&0&0&0&0&0&0&0&0&0&0&0&0&0&0&0&0&1&0&0&0\\ 0&0&0&0&0&0&0&0&0&0&0&0&0&0&0&0&0&0&1&0&0\\ 0&0&0&0&0&0&0&0&0&0&0&0&0&0&0&0&0&0&0&0&1\end {array} \right] \\ 
T(460,531) & =  \left[ \begin {array}{ccccccccccccccccccccc} 1&0&0&0&0&0&0&0&0&0&0&0&0&0&0&0&0&0&0&0&0\\ 0&1&0&0&0&0&0&0&0&0&0&0&0&0&0&0&0&0&0&0&0\\ 0&0&0&1&0&0&0&0&0&0&0&0&0&0&0&0&0&0&0&0&0\\ 0&0&0&0&1&0&0&0&0&0&0&0&0&0&0&0&0&0&0&0&0\\ 0&0&0&0&0&1&0&0&0&0&0&0&0&0&0&0&0&0&0&0&0\\ 0&0&0&0&0&0&0&1&0&0&0&0&0&0&0&0&0&0&0&0&0\\ 1&0&0&0&0&0&0&0&1&0&0&0&0&0&0&0&0&0&0&0&0\\ 0&0&1&0&0&0&0&0&0&1&0&0&0&0&0&0&0&0&0&0&0\\ 0&0&0&0&1&0&0&0&0&0&0&1&0&0&0&0&0&0&0&0&0\\ 0&0&0&0&0&0&1&0&0&0&0&0&1&0&0&0&0&0&0&0&0\\ 0&0&0&0&0&0&0&0&1&0&0&0&0&1&0&0&0&0&0&0&0\\ 0&0&0&0&0&0&0&0&0&1&0&0&0&0&0&1&0&0&0&0&0\\ 0&0&0&0&0&0&0&0&0&0&1&0&0&0&0&0&1&0&0&0&0\\ 0&0&0&0&0&0&0&0&0&0&0&1&0&0&0&0&0&1&0&0&0\\ 0&0&0&0&0&0&0&0&0&0&0&0&1&0&0&0&0&0&0&1&0\\ 0&0&0&0&0&0&0&0&0&0&0&0&0&0&1&0&0&0&0&0&0\\ 0&0&0&0&0&0&0&0&0&0&0&0&0&0&0&1&0&0&0&0&0\\ 0&0&0&0&0&0&0&0&0&0&0&0&0&0&0&0&1&0&0&0&0\\ 0&0&0&0&0&0&0&0&0&0&0&0&0&0&0&0&0&0&1&0&0\\ 0&0&0&0&0&0&0&0&0&0&0&0&0&0&0&0&0&0&0&1&0\\ 0&0&0&0&0&0&0&0&0&0&0&0&0&0&0&0&0&0&0&0&1\end {array} \right] \\ 
T(524,590) & =  \left[ \begin {array}{ccccccccccccccccccccc} 1&0&0&0&0&0&0&0&0&0&0&0&0&0&0&0&0&0&0&0&0\\ 0&1&0&0&0&0&0&0&0&0&0&0&0&0&0&0&0&0&0&0&0\\ 0&0&1&0&0&0&0&0&0&0&0&0&0&0&0&0&0&0&0&0&0\\ 0&0&0&0&1&0&0&0&0&0&0&0&0&0&0&0&0&0&0&0&0\\ 0&0&0&0&0&1&0&0&0&0&0&0&0&0&0&0&0&0&0&0&0\\ 0&0&0&0&0&0&1&0&0&0&0&0&0&0&0&0&0&0&0&0&0\\ 0&1&0&0&0&0&0&0&1&0&0&0&0&0&0&0&0&0&0&0&0\\ 0&0&0&1&0&0&0&0&0&1&0&0&0&0&0&0&0&0&0&0&0\\ 0&0&0&0&1&0&0&0&0&0&1&0&0&0&0&0&0&0&0&0&0\\ 0&0&0&0&0&1&0&0&0&0&0&1&0&0&0&0&0&0&0&0&0\\ 0&0&0&0&0&0&0&1&0&0&0&0&1&0&0&0&0&0&0&0&0\\ 0&0&0&0&0&0&0&0&1&0&0&0&0&0&1&0&0&0&0&0&0\\ 0&0&0&0&0&0&0&0&0&1&0&0&0&0&0&0&1&0&0&0&0\\ 0&0&0&0&0&0&0&0&0&0&0&1&0&0&0&0&0&0&1&0&0\\ 0&0&0&0&0&0&0&0&0&0&0&0&1&0&0&0&0&0&0&0&1\\ 0&0&0&0&0&0&0&0&0&0&0&0&0&1&0&0&0&0&0&0&0\\ 0&0&0&0&0&0&0&0&0&0&0&0&0&0&0&1&0&0&0&0&0\\ 0&0&0&0&0&0&0&0&0&0&0&0&0&0&0&0&1&0&0&0&0\\ 0&0&0&0&0&0&0&0&0&0&0&0&0&0&0&0&0&1&0&0&0\\ 0&0&0&0&0&0&0&0&0&0&0&0&0&0&0&0&0&0&0&1&0\\ 0&0&0&0&0&0&0&0&0&0&0&0&0&0&0&0&0&0&0&0&1\end {array} \right] \\ 
T(531,454) & =  \left[ \begin {array}{ccccccccccccccccccccc} 0&1&0&0&0&0&0&0&0&0&0&0&0&0&0&0&0&0&0&0&0\\ 0&0&1&0&0&0&0&0&0&0&0&0&0&0&0&0&0&0&0&0&0\\ 0&0&0&1&0&0&0&0&0&0&0&0&0&0&0&0&0&0&0&0&0\\ 0&0&0&0&1&0&0&0&0&0&0&0&0&0&0&0&0&0&0&0&0\\ 0&0&0&0&0&1&0&0&0&0&0&0&0&0&0&0&0&0&0&0&0\\ 1&0&0&0&0&0&0&1&0&0&0&0&0&0&0&0&0&0&0&0&0\\ 0&1&0&0&0&0&0&0&1&0&0&0&0&0&0&0&0&0&0&0&0\\ 0&0&1&0&0&0&0&0&0&1&0&0&0&0&0&0&0&0&0&0&0\\ 0&0&0&0&1&0&0&0&0&0&1&0&0&0&0&0&0&0&0&0&0\\ 0&0&0&0&0&1&0&0&0&0&0&1&0&0&0&0&0&0&0&0&0\\ 0&0&0&0&0&0&1&0&0&0&0&0&0&1&0&0&0&0&0&0&0\\ 0&0&0&0&0&0&0&1&0&0&0&0&0&0&1&0&0&0&0&0&0\\ 0&0&0&0&0&0&0&0&0&0&1&0&0&0&0&0&0&1&0&0&0\\ 0&0&0&0&0&0&0&0&0&0&0&0&1&0&0&0&0&0&1&0&0\\ 0&0&0&0&0&0&0&0&0&0&0&0&0&1&0&0&0&0&0&1&0\\ 0&0&0&0&0&0&0&0&0&0&0&0&0&0&1&0&0&0&0&0&1\\ 0&0&0&0&0&0&0&0&0&0&0&0&0&0&0&1&0&0&0&0&0\\ 0&0&0&0&0&0&0&0&0&0&0&0&0&0&0&0&1&0&0&0&0\\ 0&0&0&0&0&0&0&0&0&0&0&0&0&0&0&0&0&1&0&0&0\\ 0&0&0&0&0&0&0&0&0&0&0&0&0&0&0&0&0&0&1&0&0\\ 0&0&0&0&0&0&0&0&0&0&0&0&0&0&0&0&0&0&0&0&1\end {array} \right] \\ 
T(590,388) & =  \left[ \begin {array}{ccccccccccccccccccccc} 1&0&0&0&0&0&0&0&0&0&0&0&0&0&0&0&0&0&0&0&0\\ 0&0&1&0&0&0&0&0&0&0&0&0&0&0&0&0&0&0&0&0&0\\ 0&0&0&1&0&0&0&0&0&0&0&0&0&0&0&0&0&0&0&0&0\\ 0&0&0&0&1&0&0&0&0&0&0&0&0&0&0&0&0&0&0&0&0\\ 0&0&0&0&0&1&0&0&0&0&0&0&0&0&0&0&0&0&0&0&0\\ 1&0&0&0&0&0&1&0&0&0&0&0&0&0&0&0&0&0&0&0&0\\ 0&1&0&0&0&0&0&1&0&0&0&0&0&0&0&0&0&0&0&0&0\\ 0&0&1&0&0&0&0&0&1&0&0&0&0&0&0&0&0&0&0&0&0\\ 0&0&0&1&0&0&0&0&0&0&1&0&0&0&0&0&0&0&0&0&0\\ 0&0&0&0&0&0&1&0&0&0&0&0&0&1&0&0&0&0&0&0&0\\ 0&0&0&0&0&0&0&1&0&0&0&0&0&0&1&0&0&0&0&0&0\\ 0&0&0&0&0&0&0&0&0&1&0&0&0&0&0&1&0&0&0&0&0\\ 0&0&0&0&0&0&0&0&0&0&1&0&0&0&0&0&1&0&0&0&0\\ 0&0&0&0&0&0&0&0&0&0&0&1&0&0&0&0&0&0&1&0&0\\ 0&0&0&0&0&0&0&0&0&0&0&0&1&0&0&0&0&0&0&1&0\\ 0&0&0&0&0&0&0&0&0&0&0&0&0&1&0&0&0&0&0&0&1\\ 0&0&0&0&0&0&0&0&0&0&0&0&0&0&0&1&0&0&0&0&0\\ 0&0&0&0&0&0&0&0&0&0&0&0&0&0&0&0&1&0&0&0&0\\ 0&0&0&0&0&0&0&0&0&0&0&0&0&0&0&0&0&1&0&0&0\\ 0&0&0&0&0&0&0&0&0&0&0&0&0&0&0&0&0&0&1&0&0\\ 0&0&0&0&0&0&0&0&0&0&0&0&0&0&0&0&0&0&0&1&0\end {array} \right] \\ 
\end{align*}
The maximal loop class is a simple loop.
It's spectral radius is an isolated points of 1.510392888.
These points have local dimension .998498570.

Maximal Loop Class: [104].
The reduced characteristic vectors are:
\begin{itemize}
\item Reduced characteristic vector 104: $(\varrho^2+2 \varrho-2, (-3 \varrho^2-3 \varrho+4, \varrho^2-2 \varrho+1, -2 \varrho^2-\varrho+2, -\varrho^2-3 \varrho+3, 1-\varrho, -2 \varrho^2-2 \varrho+3, -2 \varrho+2, -2 \varrho^2-3 \varrho+4, -\varrho^2-\varrho+2, -3 \varrho+3, -3 \varrho^2-2 \varrho+4, \varrho^2-\varrho+1, -\varrho^2-2 \varrho+3))$ 
\end{itemize}
See Figure \ref{fig:Pic15} for the transition diagram.
\begin{figure}[H]
\includegraphics[scale=0.5]{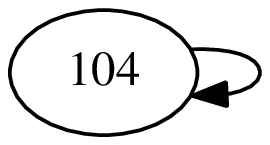}
\caption{$x^3+x^2-1$ with $d_i \in [0, 1-\varrho]$, Subloop [104]}
\label{fig:Pic15}
\end{figure}
This has transition matrices:
\begin{align*}
T(104,104) & =  \left[ \begin {array}{ccccccccccccc} 0&1&0&0&0&0&0&0&0&0&0&0&0\\ 0&0&1&0&0&0&0&0&0&0&0&0&0\\ 0&0&0&1&0&0&0&0&0&0&0&0&0\\ 0&0&0&0&1&0&0&0&0&0&0&0&0\\ 1&0&0&0&0&1&0&0&0&0&0&0&0\\ 0&0&0&1&0&0&1&0&0&0&0&0&0\\ 0&0&0&0&0&1&0&0&1&0&0&0&0\\ 0&0&0&0&0&0&1&0&0&0&0&1&0\\ 0&0&0&0&0&0&0&1&0&0&0&0&1\\ 0&0&0&0&0&0&0&0&1&0&0&0&0\\ 0&0&0&0&0&0&0&0&0&1&0&0&0\\ 0&0&0&0&0&0&0&0&0&0&1&0&0\\ 0&0&0&0&0&0&0&0&0&0&0&0&1\end {array} \right] \\ 
\end{align*}
The maximal loop class is a simple loop.
It's spectral radius is an isolated points of 1.506135679.
These points have local dimension 1.008536246.

Maximal Loop Class: [82].
The reduced characteristic vectors are:
\begin{itemize}
\item Reduced characteristic vector 82: $(\varrho^2+2 \varrho-2, (0, -2 \varrho^2-\varrho+2, 2 \varrho^2-1, -\varrho^2+\varrho, 1-\varrho, \varrho^2+\varrho-1, -\varrho^2+1, \varrho^2, -\varrho^2-\varrho+2, \varrho, \varrho^2-\varrho+1, -2 \varrho^2+2, 2 \varrho^2+\varrho-1))$ 
\end{itemize}
See Figure \ref{fig:Pic16} for the transition diagram.
\begin{figure}[H]
\includegraphics[scale=0.5]{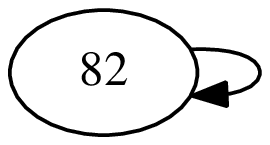}
\caption{$x^3+x^2-1$ with $d_i \in [0, 1-\varrho]$, Subloop [82]}
\label{fig:Pic16}
\end{figure}
This has transition matrices:
\begin{align*}
T(82,82) & =  \left[ \begin {array}{ccccccccccccc} 1&0&0&0&0&0&0&0&0&0&0&0&0\\ 0&0&1&0&0&0&0&0&0&0&0&0&0\\ 0&0&0&1&0&0&0&0&0&0&0&0&0\\ 0&0&0&0&1&0&0&0&0&0&0&0&0\\ 1&0&0&0&0&1&0&0&0&0&0&0&0\\ 0&1&0&0&0&0&1&0&0&0&0&0&0\\ 0&0&0&0&1&0&0&1&0&0&0&0&0\\ 0&0&0&0&0&0&1&0&0&1&0&0&0\\ 0&0&0&0&0&0&0&1&0&0&0&0&1\\ 0&0&0&0&0&0&0&0&1&0&0&0&0\\ 0&0&0&0&0&0&0&0&0&1&0&0&0\\ 0&0&0&0&0&0&0&0&0&0&1&0&0\\ 0&0&0&0&0&0&0&0&0&0&0&1&0\end {array} \right] \\ 
\end{align*}
The maximal loop class is a simple loop.
It's spectral radius is an isolated points of 1.506135679.
These points have local dimension 1.008536246.

Maximal Loop Class: [4].
The reduced characteristic vectors are:
\begin{itemize}
\item Reduced characteristic vector 4: $(\varrho^2+\varrho-1, (-\varrho^2-\varrho+2))$ 
\end{itemize}
See Figure \ref{fig:Pic17} for the transition diagram.
\begin{figure}[H]
\includegraphics[scale=0.5]{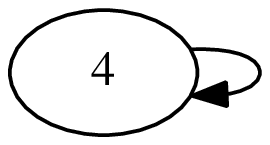}
\caption{$x^3+x^2-1$ with $d_i \in [0, 1-\varrho]$, Subloop [4]}
\label{fig:Pic17}
\end{figure}
This has transition matrices:
\begin{align*}
T(4,4) & =  \left[ \begin {array}{c} 1\end {array} \right] & 
 \end{align*}
The maximal loop class is a simple loop.
It's spectral radius is an isolated points of 1.
These points have local dimension 2.464965255.

Maximal Loop Class: [2].
The reduced characteristic vectors are:
\begin{itemize}
\item Reduced characteristic vector 2: $(\varrho^2+\varrho-1, (0))$ 
\end{itemize}
See Figure \ref{fig:Pic18} for the transition diagram.
\begin{figure}[H]
\includegraphics[scale=0.5]{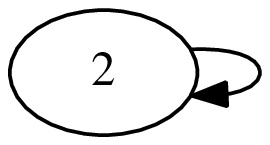}
\caption{$x^3+x^2-1$ with $d_i \in [0, 1-\varrho]$, Subloop [2]}
\label{fig:Pic18}
\end{figure}
This has transition matrices:
\begin{align*}
T(2,2) & =  \left[ \begin {array}{c} 1\end {array} \right] & 
 \end{align*}
The maximal loop class is a simple loop.
It's spectral radius is an isolated points of 1.
These points have local dimension 2.464965255.

\section{Minimal polynomial $x^3-x^2+2 x-1$ with $d_i \in [0, 1-\varrho]$} 
\label{sec:12}
 
Consider $\varrho$, the root of $x^3-x^2+2 x-1$ and the maps $S_i(x) = \varrho x  + d_i$ with $d_{0} = 0$, and $d_{1} = 1-\varrho$.
The probabilities are uniform.
The reduced transition diagram has 30 reduced characteristic vectors.
The reduced characteristic vectors are:
\begin{itemize}
\item Reduced characteristic vector 1: $(1, (0))$ 
\item Reduced characteristic vector 2: $(\varrho^2-\varrho+1, (0))$ 
\item Reduced characteristic vector 3: $(-\varrho^2+\varrho, (0, \varrho^2-\varrho+1))$ 
\item Reduced characteristic vector 4: $(\varrho^2-\varrho+1, (-\varrho^2+\varrho))$ 
\item Reduced characteristic vector 5: $(\varrho^2, (-\varrho^2+\varrho))$ 
\item Reduced characteristic vector 6: $(1-\varrho, (0, \varrho))$ 
\item Reduced characteristic vector 7: $(\varrho^2, (1-\varrho))$ 
\item Reduced characteristic vector 8: $(\varrho^2-\varrho+1, (0, -\varrho^2+\varrho))$ 
\item Reduced characteristic vector 9: $(\varrho^2, (0, 1-\varrho))$ 
\item Reduced characteristic vector 10: $(-\varrho^2+\varrho, (0, \varrho^2, \varrho^2-\varrho+1))$ 
\item Reduced characteristic vector 11: $(\varrho^2-2 \varrho+1, (-\varrho^2+\varrho, \varrho))$ 
\item Reduced characteristic vector 12: $(-\varrho^2+\varrho, (0, 1-\varrho, \varrho^2-\varrho+1))$ 
\item Reduced characteristic vector 13: $(\varrho^2, (-\varrho^2+\varrho, -\varrho^2+1))$ 
\item Reduced characteristic vector 14: $(\varrho^2-2 \varrho+1, (0, \varrho))$ 
\item Reduced characteristic vector 15: $(-\varrho^2+\varrho, (0, \varrho^2-2 \varrho+1, \varrho^2-\varrho+1))$ 
\item Reduced characteristic vector 16: $(\varrho^2, (-\varrho^2+\varrho, 1-\varrho))$ 
\item Reduced characteristic vector 17: $(-\varrho^2+\varrho, (0, \varrho, \varrho^2-\varrho+1))$ 
\item Reduced characteristic vector 18: $(\varrho^2-2 \varrho+1, (-\varrho^2+\varrho, -\varrho^2+2 \varrho))$ 
\item Reduced characteristic vector 19: $(\varrho^2, (0, -\varrho^2+\varrho))$ 
\item Reduced characteristic vector 20: $(1-\varrho, (0, \varrho^2, \varrho))$ 
\item Reduced characteristic vector 21: $(2 \varrho^2-\varrho, (-\varrho^2+\varrho, -\varrho^2+1))$ 
\item Reduced characteristic vector 22: $(1-\varrho, (0, -\varrho^2+\varrho, \varrho))$ 
\item Reduced characteristic vector 23: $(\varrho^2, (1-\varrho, -\varrho^2+1))$ 
\item Reduced characteristic vector 24: $(\varrho^2-2 \varrho+1, (0, -\varrho^2+\varrho, \varrho))$ 
\item Reduced characteristic vector 25: $(-\varrho^2+\varrho, (0, \varrho^2-2 \varrho+1, 1-\varrho, \varrho^2-\varrho+1))$ 
\item Reduced characteristic vector 26: $(\varrho^2, (-\varrho^2+\varrho, 1-\varrho, -\varrho^2+1))$ 
\item Reduced characteristic vector 27: $(2 \varrho-1, (1-\varrho))$ 
\item Reduced characteristic vector 28: $(\varrho^2, (0, -\varrho^2+\varrho, 1-\varrho))$ 
\item Reduced characteristic vector 29: $(-\varrho^2+\varrho, (0, \varrho^2, \varrho, \varrho^2-\varrho+1))$ 
\item Reduced characteristic vector 30: $(\varrho^2-2 \varrho+1, (-\varrho^2+\varrho, \varrho, -\varrho^2+2 \varrho))$ 
\end{itemize}
See Figure \ref{fig:Pic19} for the transition diagram.
\begin{figure}[H]
\includegraphics[scale=0.4]{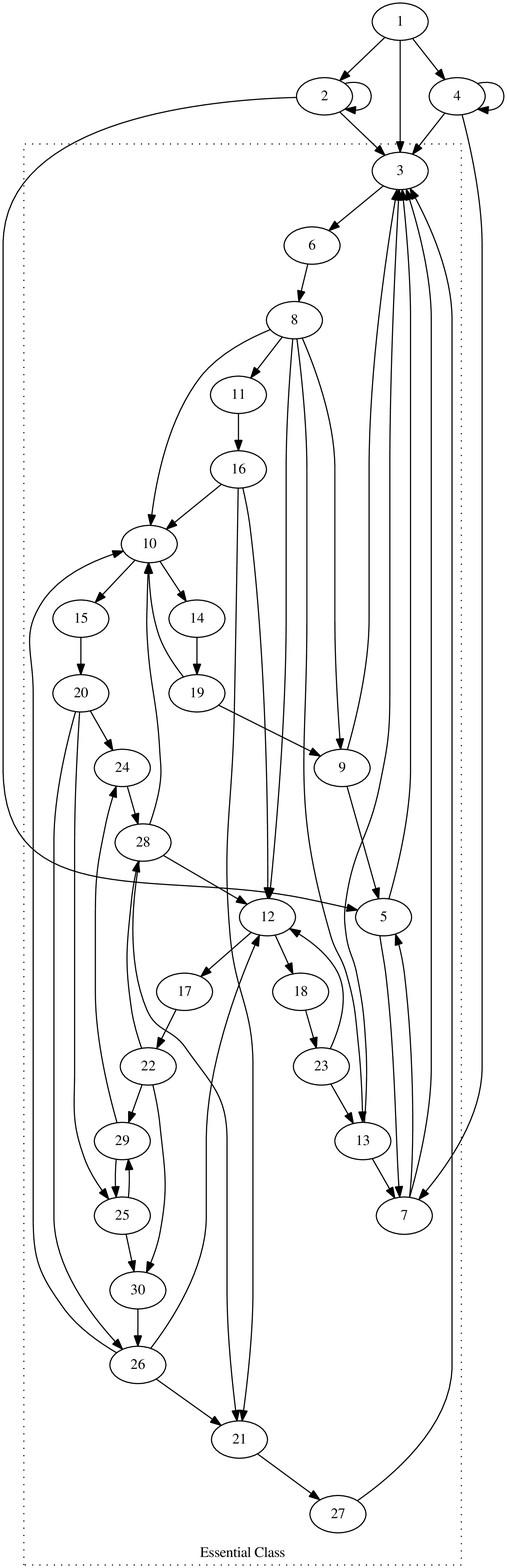}
\caption{$x^3-x^2+2 x-1$ with $d_i \in [0, 1-\varrho]$, Full set + Essential class}
\label{fig:Pic19}
\end{figure}
This has transition matrices:
\begin{align*}
T(1,2) & =  \left[ \begin {array}{c} 1\end {array} \right] & 
 T(1,3) & =  \left[ \begin {array}{cc} 1&1\end {array} \right] \\ 
T(1,4) & =  \left[ \begin {array}{c} 1\end {array} \right] & 
 T(2,2) & =  \left[ \begin {array}{c} 1\end {array} \right] \\ 
T(2,3) & =  \left[ \begin {array}{cc} 1&1\end {array} \right] & 
 T(2,5) & =  \left[ \begin {array}{c} 1\end {array} \right] \\ 
T(3,6) & =  \left[ \begin {array}{cc} 1&0\\ 0&1\end {array} \right] & 
 T(4,7) & =  \left[ \begin {array}{c} 1\end {array} \right] \\ 
T(4,3) & =  \left[ \begin {array}{cc} 1&1\end {array} \right] & 
 T(4,4) & =  \left[ \begin {array}{c} 1\end {array} \right] \\ 
T(5,7) & =  \left[ \begin {array}{c} 1\end {array} \right] & 
 T(5,3) & =  \left[ \begin {array}{cc} 1&1\end {array} \right] \\ 
T(6,8) & =  \left[ \begin {array}{cc} 1&0\\ 0&1\end {array} \right] & 
 T(7,3) & =  \left[ \begin {array}{cc} 1&1\end {array} \right] \\ 
T(7,5) & =  \left[ \begin {array}{c} 1\end {array} \right] & 
 T(8,9) & =  \left[ \begin {array}{cc} 1&0\\ 0&1\end {array} \right] \\ 
T(8,10) & =  \left[ \begin {array}{ccc} 0&1&0\\ 1&0&1\end {array} \right] & 
 T(8,11) & =  \left[ \begin {array}{cc} 0&1\\ 1&0\end {array} \right] \\ 
T(8,12) & =  \left[ \begin {array}{ccc} 1&0&1\\ 0&1&0\end {array} \right] & 
 T(8,13) & =  \left[ \begin {array}{cc} 1&0\\ 0&1\end {array} \right] \\ 
T(9,3) & =  \left[ \begin {array}{cc} 1&0\\ 1&1\end {array} \right] & 
 T(9,5) & =  \left[ \begin {array}{c} 1\\ 1\end {array} \right] \\ 
T(10,14) & =  \left[ \begin {array}{cc} 1&0\\ 0&1\\ 0&1\end {array} \right] & 
 T(10,15) & =  \left[ \begin {array}{ccc} 0&1&0\\ 1&0&1\\ 0&0&1\end {array} \right] \\ 
T(11,16) & =  \left[ \begin {array}{cc} 0&1\\ 1&0\end {array} \right] & 
 T(12,17) & =  \left[ \begin {array}{ccc} 1&0&0\\ 1&0&1\\ 0&1&0\end {array} \right] \\ 
T(12,18) & =  \left[ \begin {array}{cc} 1&0\\ 1&0\\ 0&1\end {array} \right] & 
 T(13,7) & =  \left[ \begin {array}{c} 1\\ 1\end {array} \right] \\ 
T(13,3) & =  \left[ \begin {array}{cc} 1&1\\ 0&1\end {array} \right] & 
 T(14,19) & =  \left[ \begin {array}{cc} 1&0\\ 0&1\end {array} \right] \\ 
T(15,20) & =  \left[ \begin {array}{ccc} 1&0&0\\ 0&1&0\\ 0&0&1\end {array} \right] & 
 T(16,12) & =  \left[ \begin {array}{ccc} 0&1&0\\ 1&0&1\end {array} \right] \\ 
T(16,21) & =  \left[ \begin {array}{cc} 0&1\\ 1&0\end {array} \right] & 
 T(16,10) & =  \left[ \begin {array}{ccc} 1&0&1\\ 0&1&0\end {array} \right] \\ 
T(17,22) & =  \left[ \begin {array}{ccc} 1&0&0\\ 0&1&0\\ 0&0&1\end {array} \right] & 
 T(18,23) & =  \left[ \begin {array}{cc} 1&0\\ 0&1\end {array} \right] \\ 
T(19,9) & =  \left[ \begin {array}{cc} 1&0\\ 0&1\end {array} \right] & 
 T(19,10) & =  \left[ \begin {array}{ccc} 0&1&0\\ 1&0&1\end {array} \right] \\ 
T(20,24) & =  \left[ \begin {array}{ccc} 1&0&0\\ 0&0&1\\ 0&1&0\end {array} \right] & 
 T(20,25) & =  \left[ \begin {array}{cccc} 0&1&0&0\\ 1&0&0&1\\ 0&0&1&0\end {array} \right] \\ 
T(20,26) & =  \left[ \begin {array}{ccc} 0&1&0\\ 1&0&0\\ 0&0&1\end {array} \right] & 
 T(21,27) & =  \left[ \begin {array}{c} 1\\ 1\end {array} \right] \\ 
T(22,28) & =  \left[ \begin {array}{ccc} 1&0&0\\ 0&0&1\\ 0&1&0\end {array} \right] & 
 T(22,29) & =  \left[ \begin {array}{cccc} 0&1&0&0\\ 1&0&0&1\\ 0&0&1&0\end {array} \right] \\ 
T(22,30) & =  \left[ \begin {array}{ccc} 0&1&0\\ 1&0&0\\ 0&0&1\end {array} \right] & 
 T(23,12) & =  \left[ \begin {array}{ccc} 1&0&1\\ 0&1&0\end {array} \right] \\ 
T(23,13) & =  \left[ \begin {array}{cc} 1&0\\ 0&1\end {array} \right] & 
 T(24,28) & =  \left[ \begin {array}{ccc} 1&0&0\\ 0&0&1\\ 0&1&0\end {array} \right] \\ 
T(25,29) & =  \left[ \begin {array}{cccc} 1&0&0&0\\ 0&1&0&0\\ 1&0&0&1\\ 0&0&1&0\end {array} \right] & 
 T(25,30) & =  \left[ \begin {array}{ccc} 1&0&0\\ 0&1&0\\ 1&0&0\\ 0&0&1\end {array} \right] \\ 
T(26,12) & =  \left[ \begin {array}{ccc} 0&1&0\\ 1&0&1\\ 0&1&0\end {array} \right] & 
 T(26,21) & =  \left[ \begin {array}{cc} 0&1\\ 1&0\\ 0&1\end {array} \right] \\ 
T(26,10) & =  \left[ \begin {array}{ccc} 1&0&1\\ 0&1&0\\ 0&0&1\end {array} \right] & 
 T(27,3) & =  \left[ \begin {array}{cc} 1&1\end {array} \right] \\ 
T(28,12) & =  \left[ \begin {array}{ccc} 1&0&0\\ 0&1&0\\ 1&0&1\end {array} \right] & 
 T(28,21) & =  \left[ \begin {array}{cc} 1&0\\ 0&1\\ 1&0\end {array} \right] \\ 
T(28,10) & =  \left[ \begin {array}{ccc} 0&1&0\\ 1&0&1\\ 0&1&0\end {array} \right] & 
 T(29,24) & =  \left[ \begin {array}{ccc} 1&0&0\\ 0&0&1\\ 0&1&0\\ 0&0&1\end {array} \right] \\ 
T(29,25) & =  \left[ \begin {array}{cccc} 0&1&0&0\\ 1&0&0&1\\ 0&0&1&0\\ 0&0&0&1\end {array} \right] & 
 T(30,26) & =  \left[ \begin {array}{ccc} 0&1&0\\ 1&0&0\\ 0&0&1\end {array} \right] \\ 
\end{align*}

The essential class is: [3, 5, 6, 7, 8, 9, 10, 11, 12, 13, 14, 15, 16, 17, 18, 19, 20, 21, 22, 23, 24, 25, 26, 27, 28, 29, 30].
The essential class is of positive type.
An example is the path [5, 3].
The essential class is not a simple loop.
This spectral range will include the interval $[1., 1.220744085]$.
The minimum comes from the loop $[5, 7, 5]$.
The maximum comes from the loop $[25, 29, 25]$.
These points will include points of local dimension [.8778224563, 1.232482628].
The Spectral Range is contained in the range $[1., 1.292392221]$.
The minimum comes from the total row sub-norm of length 10. 
The maximum comes from the total row sup-norm of length 10. 
These points will have local dimension contained in [.7764098609, 1.232482628].

There are 2 additional maximal loops.

Maximal Loop Class: [4].
The maximal loop class is a simple loop.
It's spectral radius is an isolated points of 1.
These points have local dimension 1.232482628.

Maximal Loop Class: [2].
The maximal loop class is a simple loop.
It's spectral radius is an isolated points of 1.
These points have local dimension 1.232482628.

\section{Minimal polynomial $x^3+x^2+x-1$ with $d_i \in [0, 1-\varrho]$} 
\label{sec:13}
 
Consider $\varrho$, the root of $x^3+x^2+x-1$ and the maps $S_i(x) = \varrho x  + d_i$ with $d_{0} = 0$, and $d_{1} = 1-\varrho$.
The probabilities are uniform.
The reduced transition diagram has 11 reduced characteristic vectors.
The reduced characteristic vectors are:
\begin{itemize}
\item Reduced characteristic vector 1: $(1, (0))$ 
\item Reduced characteristic vector 2: $(\varrho^2+\varrho, (0))$ 
\item Reduced characteristic vector 3: $(-\varrho^2-\varrho+1, (0, \varrho^2+\varrho))$ 
\item Reduced characteristic vector 4: $(\varrho^2+\varrho, (-\varrho^2-\varrho+1))$ 
\item Reduced characteristic vector 5: $(\varrho, (-\varrho^2-\varrho+1))$ 
\item Reduced characteristic vector 6: $(\varrho^2, (0, -\varrho^2+1))$ 
\item Reduced characteristic vector 7: $(\varrho, (\varrho^2))$ 
\item Reduced characteristic vector 8: $(\varrho^2, (-\varrho^2-\varrho+1))$ 
\item Reduced characteristic vector 9: $(\varrho, (0, 1-\varrho))$ 
\item Reduced characteristic vector 10: $(\varrho^2, (\varrho))$ 
\item Reduced characteristic vector 11: $(2 \varrho^2+2 \varrho-1, (-\varrho^2-\varrho+1))$ 
\end{itemize}
See Figure \ref{fig:Pic20} for the transition diagram.
\begin{figure}[H]
\includegraphics[scale=0.5]{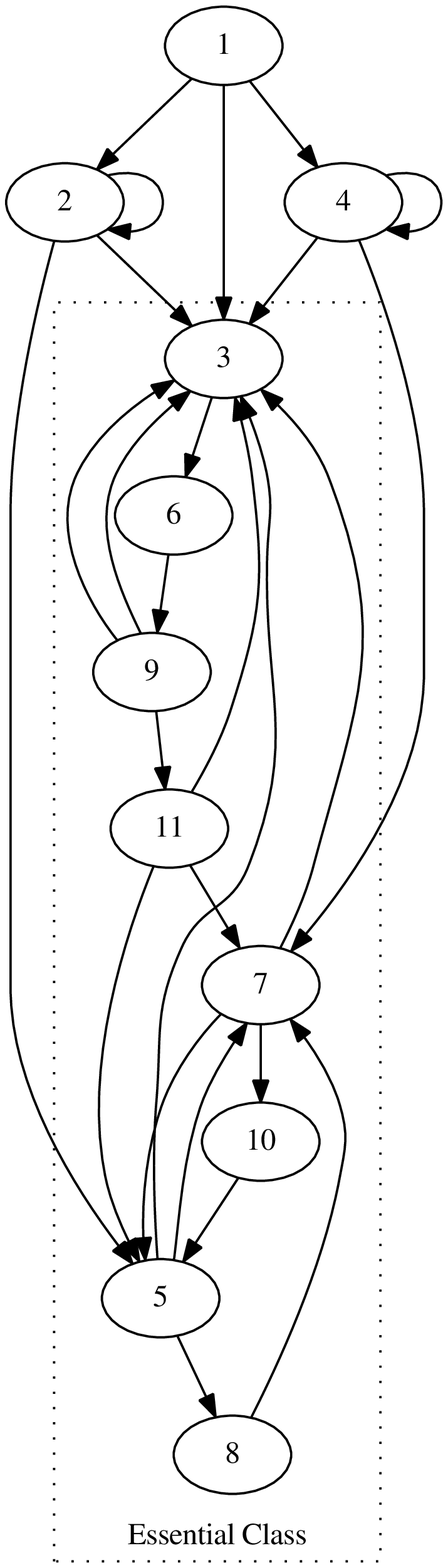}
\caption{$x^3+x^2+x-1$ with $d_i \in [0, 1-\varrho]$, Full set + Essential class}
\label{fig:Pic20}
\end{figure}
This has transition matrices:
\begin{align*}
T(1,2) & =  \left[ \begin {array}{c} 1\end {array} \right] & 
 T(1,3) & =  \left[ \begin {array}{cc} 1&1\end {array} \right] \\ 
T(1,4) & =  \left[ \begin {array}{c} 1\end {array} \right] & 
 T(2,2) & =  \left[ \begin {array}{c} 1\end {array} \right] \\ 
T(2,3) & =  \left[ \begin {array}{cc} 1&1\end {array} \right] & 
 T(2,5) & =  \left[ \begin {array}{c} 1\end {array} \right] \\ 
T(3,6) & =  \left[ \begin {array}{cc} 1&0\\ 0&1\end {array} \right] & 
 T(4,7) & =  \left[ \begin {array}{c} 1\end {array} \right] \\ 
T(4,3) & =  \left[ \begin {array}{cc} 1&1\end {array} \right] & 
 T(4,4) & =  \left[ \begin {array}{c} 1\end {array} \right] \\ 
T(5,7) & =  \left[ \begin {array}{c} 1\end {array} \right] & 
 T(5,3) & =  \left[ \begin {array}{cc} 1&1\end {array} \right] \\ 
T(5,8) & =  \left[ \begin {array}{c} 1\end {array} \right] & 
 T(6,9) & =  \left[ \begin {array}{cc} 1&0\\ 0&1\end {array} \right] \\ 
T(7,10) & =  \left[ \begin {array}{c} 1\end {array} \right] & 
 T(7,3) & =  \left[ \begin {array}{cc} 1&1\end {array} \right] \\ 
T(7,5) & =  \left[ \begin {array}{c} 1\end {array} \right] & 
 T(8,7) & =  \left[ \begin {array}{c} 1\end {array} \right] \\ 
T(9,3) & =  \left[ \begin {array}{cc} 1&0\\ 1&1\end {array} \right] & 
 T(9,11) & =  \left[ \begin {array}{c} 1\\ 1\end {array} \right] \\ 
T(9,3) & =  \left[ \begin {array}{cc} 1&1\\ 0&1\end {array} \right] & 
 T(10,5) & =  \left[ \begin {array}{c} 1\end {array} \right] \\ 
T(11,7) & =  \left[ \begin {array}{c} 1\end {array} \right] & 
 T(11,3) & =  \left[ \begin {array}{cc} 1&1\end {array} \right] \\ 
T(11,5) & =  \left[ \begin {array}{c} 1\end {array} \right] & 
 \end{align*}

The essential class is: [3, 5, 6, 7, 8, 9, 10, 11].
The essential class is of positive type.
An example is the path [5, 3].
The essential class is not a simple loop.
This spectral range will include the interval $[1., 1.189207115]$.
The minimum comes from the loop $[5, 7, 5]$.
The maximum comes from the loop $[3, 6, 9, 11, 3]$.
These points will include points of local dimension [.8531002132, 1.137466951].
The Spectral Range is contained in the range $[1., 1.231144413]$.
The minimum comes from the total row sub-norm of length 10. 
The maximum comes from the total row sup-norm of length 10. 
These points will have local dimension contained in [.7962268660, 1.137466951].

There are 2 additional maximal loops.

Maximal Loop Class: [4].
The maximal loop class is a simple loop.
It's spectral radius is an isolated points of 1.
These points have local dimension 1.137466951.

Maximal Loop Class: [2].
The maximal loop class is a simple loop.
It's spectral radius is an isolated points of 1.
These points have local dimension 1.137466951.

\section{Minimal polynomial $x^4-2 x^2-x+1$ with $d_i \in [0, 1-\varrho]$} 
\label{sec:14}
 
Consider $\varrho$, the root of $x^4-2 x^2-x+1$ and the maps $S_i(x) = \varrho x  + d_i$ with $d_{0} = 0$, and $d_{1} = 1-\varrho$.
The probabilities are uniform.
The reduced transition diagram has 538 reduced characteristic vectors.

The number of reduced characteristic vectors in the essential class is 535.
The essential class is of positive type.
An example is the path [5, 3].
The essential class is not a simple loop.
This spectral range will include the interval $[1., 1.150963925]$.
The minimum comes from the loop $[46, 46]$.
The maximum comes from the loop $[16, 22, 18, 28, 16]$.
These points will include points of local dimension [.8572351366, 1.075364982].
The Spectral Range is contained in the range $[1., 1.174618943]$.
The minimum comes from the total column sub-norm of length 10. 
The maximum comes from the total column sup-norm of length 10. 
These points will have local dimension contained in [.8256729656, 1.075364982].

There are 2 additional maximal loops.

Maximal Loop Class: [4].
The reduced characteristic vectors are:
\begin{itemize}
\item Reduced characteristic vector 4: $(-\varrho^3+2 \varrho, (\varrho^3-2 \varrho+1))$ 
\end{itemize}
See Figure \ref{fig:Pic21} for the transition diagram.
\begin{figure}[H]
\includegraphics[scale=0.5]{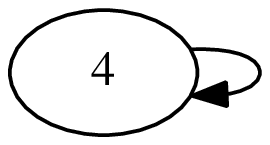}
\caption{$x^4-2 x^2-x+1$ with $d_i \in [0, 1-\varrho]$, Subloop [4]}
\label{fig:Pic21}
\end{figure}
This has transition matrices:
\begin{align*}
T(4,4) & =  \left[ \begin {array}{c} 1\end {array} \right] & 
 \end{align*}
The maximal loop class is a simple loop.
It's spectral radius is an isolated points of 1.
These points have local dimension 1.075364982.

Maximal Loop Class: [2].
The reduced characteristic vectors are:
\begin{itemize}
\item Reduced characteristic vector 2: $(-\varrho^3+2 \varrho, (0))$ 
\end{itemize}
See Figure \ref{fig:Pic22} for the transition diagram.
\begin{figure}[H]
\includegraphics[scale=0.5]{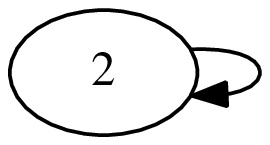}
\caption{$x^4-2 x^2-x+1$ with $d_i \in [0, 1-\varrho]$, Subloop [2]}
\label{fig:Pic22}
\end{figure}
This has transition matrices:
\begin{align*}
T(2,2) & =  \left[ \begin {array}{c} 1\end {array} \right] & 
 \end{align*}
The maximal loop class is a simple loop.
It's spectral radius is an isolated points of 1.
These points have local dimension 1.075364982.

\section{Minimal polynomial $x^4-x^3+2 x-1$ with $d_i \in [0, 1-\varrho]$} 
\label{sec:15}
 
Consider $\varrho$, the root of $x^4-x^3+2 x-1$ and the maps $S_i(x) = \varrho x  + d_i$ with $d_{0} = 0$, and $d_{1} = 1-\varrho$.
The probabilities are uniform.
The reduced transition diagram has 190 reduced characteristic vectors.

The number of reduced characteristic vectors in the essential class is 187.
The essential class is of positive type.
An example is the path [5, 3].
The essential class is not a simple loop.
This spectral range will include the interval $[1., 1.142219979]$.
The minimum comes from the loop $[5, 7, 5]$.
The maximum comes from the loop $[68, 90, 87, 68]$.
These points will include points of local dimension [.8974197439, 1.110448830].
The Spectral Range is contained in the range $[1., 1.245730940]$.
The minimum comes from the total row sub-norm of length 10. 
The maximum comes from the total row sup-norm of length 10. 
These points will have local dimension contained in [.7584448787, 1.110448830].

There are 2 additional maximal loops.

Maximal Loop Class: [4].
The reduced characteristic vectors are:
\begin{itemize}
\item Reduced characteristic vector 4: $(\varrho^3-\varrho^2+1, (-\varrho^3+\varrho^2))$ 
\end{itemize}
See Figure \ref{fig:Pic23} for the transition diagram.
\begin{figure}[H]
\includegraphics[scale=0.5]{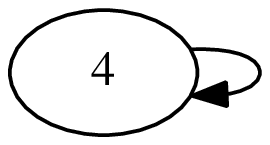}
\caption{$x^4-x^3+2 x-1$ with $d_i \in [0, 1-\varrho]$, Subloop [4]}
\label{fig:Pic23}
\end{figure}
This has transition matrices:
\begin{align*}
T(4,4) & =  \left[ \begin {array}{c} 1\end {array} \right] & 
 \end{align*}
The maximal loop class is a simple loop.
It's spectral radius is an isolated points of 1.
These points have local dimension 1.110448830.

Maximal Loop Class: [2].
The reduced characteristic vectors are:
\begin{itemize}
\item Reduced characteristic vector 2: $(\varrho^3-\varrho^2+1, (0))$ 
\end{itemize}
See Figure \ref{fig:Pic24} for the transition diagram.
\begin{figure}[H]
\includegraphics[scale=0.5]{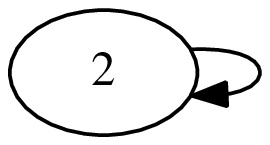}
\caption{$x^4-x^3+2 x-1$ with $d_i \in [0, 1-\varrho]$, Subloop [2]}
\label{fig:Pic24}
\end{figure}
This has transition matrices:
\begin{align*}
T(2,2) & =  \left[ \begin {array}{c} 1\end {array} \right] & 
 \end{align*}
The maximal loop class is a simple loop.
It's spectral radius is an isolated points of 1.
These points have local dimension 1.110448830.

\section{Minimal polynomial $x^4+x^3+x^2+x-1$ with $d_i \in [0, 1-\varrho]$} 
\label{sec:16}
 
Consider $\varrho$, the root of $x^4+x^3+x^2+x-1$ and the maps $S_i(x) = \varrho x  + d_i$ with $d_{0} = 0$, and $d_{1} = 1-\varrho$.
The probabilities are uniform.
The reduced transition diagram has 14 reduced characteristic vectors.
The reduced characteristic vectors are:
\begin{itemize}
\item Reduced characteristic vector 1: $(1, (0))$ 
\item Reduced characteristic vector 2: $(\varrho^3+\varrho^2+\varrho, (0))$ 
\item Reduced characteristic vector 3: $(-\varrho^3-\varrho^2-\varrho+1, (0, \varrho^3+\varrho^2+\varrho))$ 
\item Reduced characteristic vector 4: $(\varrho^3+\varrho^2+\varrho, (-\varrho^3-\varrho^2-\varrho+1))$ 
\item Reduced characteristic vector 5: $(\varrho^2+\varrho, (-\varrho^3-\varrho^2-\varrho+1))$ 
\item Reduced characteristic vector 6: $(\varrho^3, (0, -\varrho^3+1))$ 
\item Reduced characteristic vector 7: $(\varrho^2+\varrho, (\varrho^3))$ 
\item Reduced characteristic vector 8: $(\varrho^3+\varrho, (-\varrho^3-\varrho^2-\varrho+1))$ 
\item Reduced characteristic vector 9: $(\varrho^2, (0, -\varrho^2+1))$ 
\item Reduced characteristic vector 10: $(\varrho^3+\varrho, (\varrho^2))$ 
\item Reduced characteristic vector 11: $(\varrho^3+\varrho^2, (-\varrho^3-\varrho^2-\varrho+1))$ 
\item Reduced characteristic vector 12: $(\varrho, (0, 1-\varrho))$ 
\item Reduced characteristic vector 13: $(\varrho^3+\varrho^2, (\varrho))$ 
\item Reduced characteristic vector 14: $(2 \varrho^3+2 \varrho^2+2 \varrho-1, (-\varrho^3-\varrho^2-\varrho+1))$ 
\end{itemize}
See Figure \ref{fig:Pic25} for the transition diagram.
\begin{figure}[H]
\includegraphics[scale=0.5]{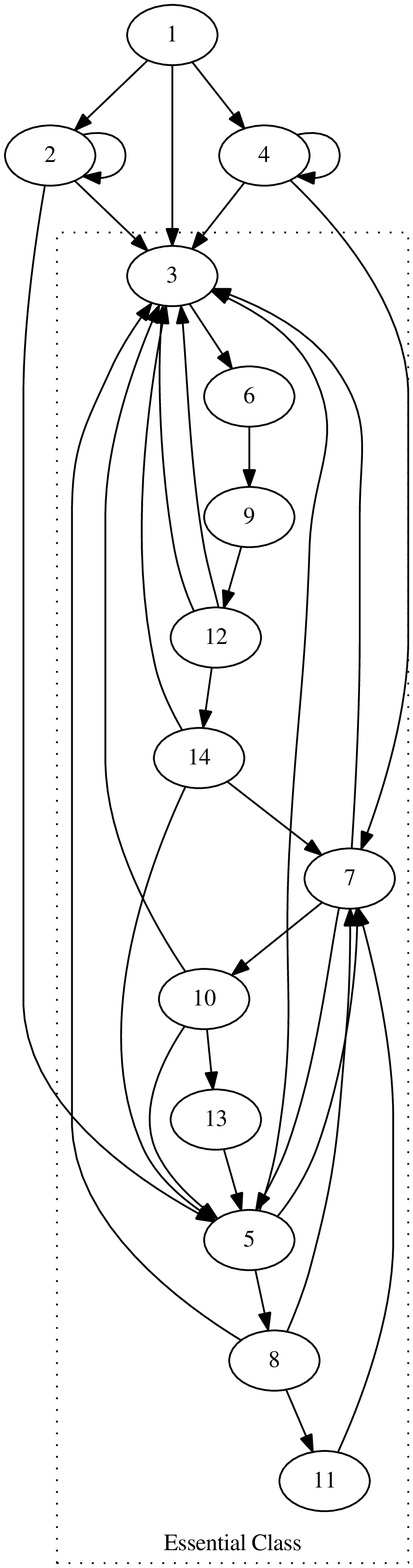}
\caption{$x^4+x^3+x^2+x-1$ with $d_i \in [0, 1-\varrho]$, Full set + Essential class}
\label{fig:Pic25}
\end{figure}
This has transition matrices:
\begin{align*}
T(1,2) & =  \left[ \begin {array}{c} 1\end {array} \right] & 
 T(1,3) & =  \left[ \begin {array}{cc} 1&1\end {array} \right] \\ 
T(1,4) & =  \left[ \begin {array}{c} 1\end {array} \right] & 
 T(2,2) & =  \left[ \begin {array}{c} 1\end {array} \right] \\ 
T(2,3) & =  \left[ \begin {array}{cc} 1&1\end {array} \right] & 
 T(2,5) & =  \left[ \begin {array}{c} 1\end {array} \right] \\ 
T(3,6) & =  \left[ \begin {array}{cc} 1&0\\ 0&1\end {array} \right] & 
 T(4,7) & =  \left[ \begin {array}{c} 1\end {array} \right] \\ 
T(4,3) & =  \left[ \begin {array}{cc} 1&1\end {array} \right] & 
 T(4,4) & =  \left[ \begin {array}{c} 1\end {array} \right] \\ 
T(5,7) & =  \left[ \begin {array}{c} 1\end {array} \right] & 
 T(5,3) & =  \left[ \begin {array}{cc} 1&1\end {array} \right] \\ 
T(5,8) & =  \left[ \begin {array}{c} 1\end {array} \right] & 
 T(6,9) & =  \left[ \begin {array}{cc} 1&0\\ 0&1\end {array} \right] \\ 
T(7,10) & =  \left[ \begin {array}{c} 1\end {array} \right] & 
 T(7,3) & =  \left[ \begin {array}{cc} 1&1\end {array} \right] \\ 
T(7,5) & =  \left[ \begin {array}{c} 1\end {array} \right] & 
 T(8,7) & =  \left[ \begin {array}{c} 1\end {array} \right] \\ 
T(8,3) & =  \left[ \begin {array}{cc} 1&1\end {array} \right] & 
 T(8,11) & =  \left[ \begin {array}{c} 1\end {array} \right] \\ 
T(9,12) & =  \left[ \begin {array}{cc} 1&0\\ 0&1\end {array} \right] & 
 T(10,13) & =  \left[ \begin {array}{c} 1\end {array} \right] \\ 
T(10,3) & =  \left[ \begin {array}{cc} 1&1\end {array} \right] & 
 T(10,5) & =  \left[ \begin {array}{c} 1\end {array} \right] \\ 
T(11,7) & =  \left[ \begin {array}{c} 1\end {array} \right] & 
 T(12,3) & =  \left[ \begin {array}{cc} 1&0\\ 1&1\end {array} \right] \\ 
T(12,14) & =  \left[ \begin {array}{c} 1\\ 1\end {array} \right] & 
 T(12,3) & =  \left[ \begin {array}{cc} 1&1\\ 0&1\end {array} \right] \\ 
T(13,5) & =  \left[ \begin {array}{c} 1\end {array} \right] & 
 T(14,7) & =  \left[ \begin {array}{c} 1\end {array} \right] \\ 
T(14,3) & =  \left[ \begin {array}{cc} 1&1\end {array} \right] & 
 T(14,5) & =  \left[ \begin {array}{c} 1\end {array} \right] \\ 
\end{align*}

The essential class is: [3, 5, 6, 7, 8, 9, 10, 11, 12, 13, 14].
The essential class is of positive type.
An example is the path [5, 3].
The essential class is not a simple loop.
This spectral range will include the interval $[1., 1.148698355]$.
The minimum comes from the loop $[5, 7, 5]$.
The maximum comes from the loop $[3, 6, 9, 12, 14, 3]$.
These points will include points of local dimension [.8449717217, 1.056214652].
The Spectral Range is contained in the range $[1., 1.196231199]$.
The minimum comes from the total row sub-norm of length 10. 
The maximum comes from the total row sup-norm of length 10. 
These points will have local dimension contained in [.7831871251, 1.056214652].

There are 2 additional maximal loops.

Maximal Loop Class: [4].
The maximal loop class is a simple loop.
It's spectral radius is an isolated points of 1.
These points have local dimension 1.056214652.

Maximal Loop Class: [2].
The maximal loop class is a simple loop.
It's spectral radius is an isolated points of 1.
These points have local dimension 1.056214652.

\section{Minimal polynomial $x^2+x-1$ with $d_i \in [0, 1/2-1/2 \varrho, 1-\varrho]$} 
\label{sec:17}
 
Consider $\varrho$, the root of $x^2+x-1$ and the maps $S_i(x) = \varrho x  + d_i$ with $d_{0} = 0$, $d_{1} = 1/2-1/2 \varrho$, and $d_{2} = 1-\varrho$.
The probabilities are given by $p_{0} = 1/4$, $p_{1} = 1/2$, and $p_{2} = 1/4$.
The reduced transition diagram has 40 reduced characteristic vectors.
The reduced characteristic vectors are:
\begin{itemize}
\item Reduced characteristic vector 1: $(1, (0))$ 
\item Reduced characteristic vector 2: $(1/2 \varrho, (0))$ 
\item Reduced characteristic vector 3: $(1/2 \varrho, (0, 1/2 \varrho))$ 
\item Reduced characteristic vector 4: $(1-\varrho, (0, 1/2 \varrho, \varrho))$ 
\item Reduced characteristic vector 5: $(1/2 \varrho, (1-\varrho, 1-1/2 \varrho))$ 
\item Reduced characteristic vector 6: $(1/2 \varrho, (1-1/2 \varrho))$ 
\item Reduced characteristic vector 7: $(1/2-1/2 \varrho, (0, 1/2 \varrho))$ 
\item Reduced characteristic vector 8: $(\varrho-1/2, (0, 1/2-1/2 \varrho, 1/2))$ 
\item Reduced characteristic vector 9: $(1/2-1/2 \varrho, (0, \varrho-1/2, 1/2 \varrho, \varrho))$ 
\item Reduced characteristic vector 10: $(1/2-1/2 \varrho, (0, 1/2-1/2 \varrho, 1/2 \varrho, 1/2, 1/2+1/2 \varrho))$ 
\item Reduced characteristic vector 11: $(\varrho-1/2, (0, 1/2-1/2 \varrho, 1-\varrho, 1/2, 1-1/2 \varrho))$ 
\item Reduced characteristic vector 12: $(1/2-1/2 \varrho, (0, \varrho-1/2, 1/2 \varrho, 1/2, \varrho, 1/2+1/2 \varrho))$ 
\item Reduced characteristic vector 13: $(1/2-1/2 \varrho, (0, 1/2-1/2 \varrho, 1/2 \varrho, 1/2, 1-1/2 \varrho, 1/2+1/2 \varrho))$ 
\item Reduced characteristic vector 14: $(\varrho-1/2, (1/2-1/2 \varrho, 1-\varrho, 1/2, 1-1/2 \varrho, 3/2-\varrho))$ 
\item Reduced characteristic vector 15: $(1/2-1/2 \varrho, (0, 1/2 \varrho, 1/2, \varrho, 1/2+1/2 \varrho))$ 
\item Reduced characteristic vector 16: $(1/2-1/2 \varrho, (1/2-1/2 \varrho, 1/2, 1-1/2 \varrho, 1/2+1/2 \varrho))$ 
\item Reduced characteristic vector 17: $(\varrho-1/2, (1-\varrho, 1-1/2 \varrho, 3/2-\varrho))$ 
\item Reduced characteristic vector 18: $(1/2-1/2 \varrho, (1/2, 1/2+1/2 \varrho))$ 
\item Reduced characteristic vector 19: $(\varrho-1/2, (0, 1/2-1/2 \varrho, 1/2 \varrho, 1/2, 1-1/2 \varrho, 1/2+1/2 \varrho))$ 
\item Reduced characteristic vector 20: $(1-3/2 \varrho, (0, \varrho-1/2, 1/2 \varrho, -1/2+3/2 \varrho, \varrho, 1/2+1/2 \varrho, 3/2 \varrho))$ 
\item Reduced characteristic vector 21: $(\varrho-1/2, (1-3/2 \varrho, 1/2-1/2 \varrho, 1-\varrho, 1/2, 1-1/2 \varrho, 3/2-\varrho))$ 
\item Reduced characteristic vector 22: $(1/2-1/2 \varrho, (0, 1/2-1/2 \varrho, 1/2 \varrho, 1/2, \varrho, 1/2+1/2 \varrho))$ 
\item Reduced characteristic vector 23: $(\varrho-1/2, (0, 1/2-1/2 \varrho, 1-\varrho, 1/2, 1-1/2 \varrho, 1/2+1/2 \varrho))$ 
\item Reduced characteristic vector 24: $(1-3/2 \varrho, (0, \varrho-1/2, 1/2 \varrho, 1/2, \varrho, 1/2+1/2 \varrho, 3/2 \varrho))$ 
\item Reduced characteristic vector 25: $(\varrho-1/2, (1-3/2 \varrho, 1/2-1/2 \varrho, 1-\varrho, 3/2-3/2 \varrho, 1-1/2 \varrho, 3/2-\varrho))$ 
\item Reduced characteristic vector 26: $(\varrho-1/2, (0, 1-3/2 \varrho, 1/2-1/2 \varrho, 1-\varrho, 1/2, 1-1/2 \varrho, 3/2-\varrho))$ 
\item Reduced characteristic vector 27: $(1/2-1/2 \varrho, (0, \varrho-1/2, 1/2-1/2 \varrho, 1/2 \varrho, 1/2, \varrho, 1/2+1/2 \varrho))$ 
\item Reduced characteristic vector 28: $(\varrho-1/2, (0, 1/2-1/2 \varrho, 1/2 \varrho, 1-\varrho, 1/2, 1-1/2 \varrho, 1/2+1/2 \varrho))$ 
\item Reduced characteristic vector 29: $(1-3/2 \varrho, (0, \varrho-1/2, 1/2 \varrho, -1/2+3/2 \varrho, 1/2, \varrho, 1/2+1/2 \varrho, 3/2 \varrho))$ 
\item Reduced characteristic vector 30: $(\varrho-1/2, (1-3/2 \varrho, 1/2-1/2 \varrho, 1-\varrho, 1/2, 3/2-3/2 \varrho, 1-1/2 \varrho, 3/2-\varrho))$ 
\item Reduced characteristic vector 31: $(1/2-1/2 \varrho, (0, 1/2-1/2 \varrho, 1/2 \varrho, 1/2, \varrho, 1-1/2 \varrho, 1/2+1/2 \varrho))$ 
\item Reduced characteristic vector 32: $(\varrho-1/2, (0, 1/2-1/2 \varrho, 1-\varrho, 1/2, 1-1/2 \varrho, 1/2+1/2 \varrho, 3/2-\varrho))$ 
\item Reduced characteristic vector 33: $(\varrho-1/2, (0, 1/2-1/2 \varrho, 1/2 \varrho, 1/2, \varrho, 1-1/2 \varrho, 1/2+1/2 \varrho))$ 
\item Reduced characteristic vector 34: $(1-3/2 \varrho, (0, \varrho-1/2, 1/2 \varrho, -1/2+3/2 \varrho, \varrho, 2 \varrho-1/2, 1/2+1/2 \varrho, 3/2 \varrho))$ 
\item Reduced characteristic vector 35: $(\varrho-1/2, (0, 1-3/2 \varrho, 1/2-1/2 \varrho, 1-\varrho, 1/2, 1-1/2 \varrho, 1/2+1/2 \varrho, 3/2-\varrho))$ 
\item Reduced characteristic vector 36: $(1-3/2 \varrho, (0, \varrho-1/2, 1/2-1/2 \varrho, 1/2 \varrho, 1/2, \varrho, 1/2+1/2 \varrho, 3/2 \varrho))$ 
\item Reduced characteristic vector 37: $(\varrho-1/2, (1-3/2 \varrho, 1/2-1/2 \varrho, 3/2-2 \varrho, 1-\varrho, 3/2-3/2 \varrho, 1-1/2 \varrho, 3/2-\varrho))$ 
\item Reduced characteristic vector 38: $(\varrho-1/2, (0, 1-3/2 \varrho, 1/2-1/2 \varrho, 1-\varrho, 1/2, 3/2-3/2 \varrho, 1-1/2 \varrho, 3/2-\varrho))$ 
\item Reduced characteristic vector 39: $(1/2-1/2 \varrho, (0, \varrho-1/2, 1/2-1/2 \varrho, 1/2 \varrho, 1/2, \varrho, 1-1/2 \varrho, 1/2+1/2 \varrho))$ 
\item Reduced characteristic vector 40: $(\varrho-1/2, (0, 1/2-1/2 \varrho, 1/2 \varrho, 1-\varrho, 1/2, 1-1/2 \varrho, 1/2+1/2 \varrho, 3/2-\varrho))$ 
\end{itemize}
See Figure \ref{fig:Pic26} for the transition diagram.
\begin{figure}[H]
\includegraphics[scale=0.4]{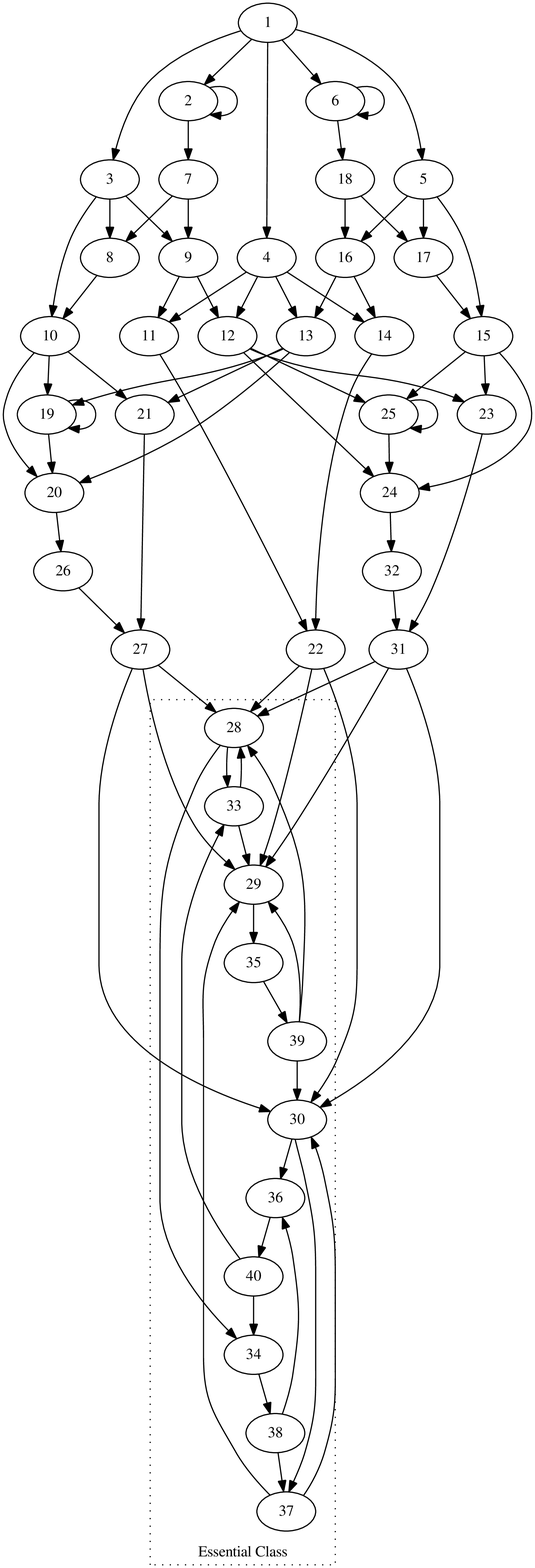}
\caption{$x^2+x-1$ with $d_i \in [0, 1/2-1/2 \varrho, 1-\varrho]$, Full set + Essential class}
\label{fig:Pic26}
\end{figure}
This has transition matrices:
\begin{align*}
T(1,2) & =  \left[ \begin {array}{c} 1\end {array} \right] & 
 T(1,3) & =  \left[ \begin {array}{cc} 2&1\end {array} \right] \\ 
T(1,4) & =  \left[ \begin {array}{ccc} 1&2&1\end {array} \right] & 
 T(1,5) & =  \left[ \begin {array}{cc} 1&2\end {array} \right] \\ 
T(1,6) & =  \left[ \begin {array}{c} 1\end {array} \right] & 
 T(2,2) & =  \left[ \begin {array}{c} 1\end {array} \right] \\ 
T(2,7) & =  \left[ \begin {array}{cc} 2&1\end {array} \right] & 
 T(3,8) & =  \left[ \begin {array}{ccc} 1&0&0\\ 0&2&1\end {array} \right] \\ 
T(3,9) & =  \left[ \begin {array}{cccc} 0&1&0&0\\ 1&0&2&1\end {array} \right] & 
 T(3,10) & =  \left[ \begin {array}{ccccc} 2&0&1&0&0\\ 0&1&0&2&1\end {array} \right] \\ 
T(4,11) & =  \left[ \begin {array}{ccccc} 1&0&0&0&0\\ 0&2&0&1&0\\ 0&0&1&0&2\end {array} \right] & 
 T(4,12) & =  \left[ \begin {array}{cccccc} 0&1&0&0&0&0\\ 1&0&2&0&1&0\\ 0&0&0&1&0&2\end {array} \right] \\ 
T(4,13) & =  \left[ \begin {array}{cccccc} 2&0&1&0&0&0\\ 0&1&0&2&0&1\\ 0&0&0&0&1&0\end {array} \right] & 
 T(4,14) & =  \left[ \begin {array}{ccccc} 2&0&1&0&0\\ 0&1&0&2&0\\ 0&0&0&0&1\end {array} \right] \\ 
T(5,15) & =  \left[ \begin {array}{ccccc} 1&2&0&1&0\\ 0&0&1&0&2\end {array} \right] & 
 T(5,16) & =  \left[ \begin {array}{cccc} 1&2&0&1\\ 0&0&1&0\end {array} \right] \\ 
T(5,17) & =  \left[ \begin {array}{ccc} 1&2&0\\ 0&0&1\end {array} \right] & 
 T(6,18) & =  \left[ \begin {array}{cc} 1&2\end {array} \right] \\ 
T(6,6) & =  \left[ \begin {array}{c} 1\end {array} \right] & 
 T(7,8) & =  \left[ \begin {array}{ccc} 1&0&0\\ 0&2&1\end {array} \right] \\ 
T(7,9) & =  \left[ \begin {array}{cccc} 0&1&0&0\\ 1&0&2&1\end {array} \right] & 
 T(8,10) & =  \left[ \begin {array}{ccccc} 1&0&0&0&0\\ 2&0&1&0&0\\ 0&1&0&2&1\end {array} \right] \\ 
T(9,11) & =  \left[ \begin {array}{ccccc} 1&0&0&0&0\\ 0&1&0&0&0\\ 0&2&0&1&0\\ 0&0&1&0&2\end {array} \right] & 
 T(9,12) & =  \left[ \begin {array}{cccccc} 0&1&0&0&0&0\\ 2&0&1&0&0&0\\ 1&0&2&0&1&0\\ 0&0&0&1&0&2\end {array} \right] \\ 
T(10,19) & =  \left[ \begin {array}{cccccc} 1&0&0&0&0&0\\ 2&0&1&0&0&0\\ 0&2&0&1&0&0\\ 0&1&0&2&0&1\\ 0&0&0&0&1&0\end {array} \right] & 
 T(10,20) & =  \left[ \begin {array}{ccccccc} 0&1&0&0&0&0&0\\ 0&2&0&1&0&0&0\\ 1&0&2&0&1&0&0\\ 0&0&1&0&2&0&1\\ 0&0&0&0&0&1&0\end {array} \right] \\ 
T(10,21) & =  \left[ \begin {array}{cccccc} 0&1&0&0&0&0\\ 0&2&0&1&0&0\\ 1&0&2&0&1&0\\ 0&0&1&0&2&0\\ 0&0&0&0&0&1\end {array} \right] & 
 T(11,22) & =  \left[ \begin {array}{cccccc} 1&0&0&0&0&0\\ 2&0&1&0&0&0\\ 1&0&2&0&1&0\\ 0&1&0&2&0&1\\ 0&0&0&1&0&2\end {array} \right] \\ 
T(12,23) & =  \left[ \begin {array}{cccccc} 1&0&0&0&0&0\\ 0&1&0&0&0&0\\ 0&2&0&1&0&0\\ 0&1&0&2&0&1\\ 0&0&1&0&2&0\\ 0&0&0&0&1&0\end {array} \right] \\ 
T(12,24) & =  \left[ \begin {array}{ccccccc} 0&1&0&0&0&0&0\\ 2&0&1&0&0&0&0\\ 1&0&2&0&1&0&0\\ 0&0&1&0&2&0&1\\ 0&0&0&1&0&2&0\\ 0&0&0&0&0&1&0\end {array} \right] \\ 
T(12,25) & =  \left[ \begin {array}{cccccc} 0&1&0&0&0&0\\ 2&0&1&0&0&0\\ 1&0&2&0&1&0\\ 0&0&1&0&2&0\\ 0&0&0&1&0&2\\ 0&0&0&0&0&1\end {array} \right] \\ 
T(13,19) & =  \left[ \begin {array}{cccccc} 1&0&0&0&0&0\\ 2&0&1&0&0&0\\ 0&2&0&1&0&0\\ 0&1&0&2&0&1\\ 0&0&0&1&0&2\\ 0&0&0&0&1&0\end {array} \right] \\ 
T(13,20) & =  \left[ \begin {array}{ccccccc} 0&1&0&0&0&0&0\\ 0&2&0&1&0&0&0\\ 1&0&2&0&1&0&0\\ 0&0&1&0&2&0&1\\ 0&0&0&0&1&0&2\\ 0&0&0&0&0&1&0\end {array} \right] \\ 
T(13,21) & =  \left[ \begin {array}{cccccc} 0&1&0&0&0&0\\ 0&2&0&1&0&0\\ 1&0&2&0&1&0\\ 0&0&1&0&2&0\\ 0&0&0&0&1&0\\ 0&0&0&0&0&1\end {array} \right] \\ 
T(14,22) & =  \left[ \begin {array}{cccccc} 2&0&1&0&0&0\\ 1&0&2&0&1&0\\ 0&1&0&2&0&1\\ 0&0&0&1&0&2\\ 0&0&0&0&0&1\end {array} \right] & 
 T(15,23) & =  \left[ \begin {array}{cccccc} 1&0&0&0&0&0\\ 0&2&0&1&0&0\\ 0&1&0&2&0&1\\ 0&0&1&0&2&0\\ 0&0&0&0&1&0\end {array} \right] \\ 
T(15,24) & =  \left[ \begin {array}{ccccccc} 0&1&0&0&0&0&0\\ 1&0&2&0&1&0&0\\ 0&0&1&0&2&0&1\\ 0&0&0&1&0&2&0\\ 0&0&0&0&0&1&0\end {array} \right] & 
 T(15,25) & =  \left[ \begin {array}{cccccc} 0&1&0&0&0&0\\ 1&0&2&0&1&0\\ 0&0&1&0&2&0\\ 0&0&0&1&0&2\\ 0&0&0&0&0&1\end {array} \right] \\ 
T(16,13) & =  \left[ \begin {array}{cccccc} 2&0&1&0&0&0\\ 0&1&0&2&0&1\\ 0&0&0&1&0&2\\ 0&0&0&0&1&0\end {array} \right] & 
 T(16,14) & =  \left[ \begin {array}{ccccc} 2&0&1&0&0\\ 0&1&0&2&0\\ 0&0&0&1&0\\ 0&0&0&0&1\end {array} \right] \\ 
T(17,15) & =  \left[ \begin {array}{ccccc} 1&2&0&1&0\\ 0&0&1&0&2\\ 0&0&0&0&1\end {array} \right] & 
 T(18,16) & =  \left[ \begin {array}{cccc} 1&2&0&1\\ 0&0&1&0\end {array} \right] \\ 
T(18,17) & =  \left[ \begin {array}{ccc} 1&2&0\\ 0&0&1\end {array} \right] & 
 T(19,19) & =  \left[ \begin {array}{cccccc} 1&0&0&0&0&0\\ 2&0&1&0&0&0\\ 0&2&0&1&0&0\\ 0&1&0&2&0&1\\ 0&0&0&1&0&2\\ 0&0&0&0&1&0\end {array} \right] \\ 
T(19,20) & =  \left[ \begin {array}{ccccccc} 0&1&0&0&0&0&0\\ 0&2&0&1&0&0&0\\ 1&0&2&0&1&0&0\\ 0&0&1&0&2&0&1\\ 0&0&0&0&1&0&2\\ 0&0&0&0&0&1&0\end {array} \right] \\ 
T(20,26) & =  \left[ \begin {array}{ccccccc} 1&0&0&0&0&0&0\\ 0&0&1&0&0&0&0\\ 0&0&2&0&1&0&0\\ 0&1&0&2&0&1&0\\ 0&0&0&1&0&2&0\\ 0&0&0&0&0&1&0\\ 0&0&0&0&0&0&1\end {array} \right] \\ 
T(21,27) & =  \left[ \begin {array}{ccccccc} 0&1&0&0&0&0&0\\ 2&0&0&1&0&0&0\\ 1&0&0&2&0&1&0\\ 0&0&1&0&2&0&1\\ 0&0&0&0&1&0&2\\ 0&0&0&0&0&0&1\end {array} \right] \\ 
T(22,28) & =  \left[ \begin {array}{ccccccc} 1&0&0&0&0&0&0\\ 2&0&1&0&0&0&0\\ 0&2&0&0&1&0&0\\ 0&1&0&0&2&0&1\\ 0&0&0&1&0&2&0\\ 0&0&0&0&0&1&0\end {array} \right] \\ 
T(22,29) & =  \left[ \begin {array}{cccccccc} 0&1&0&0&0&0&0&0\\ 0&2&0&1&0&0&0&0\\ 1&0&2&0&0&1&0&0\\ 0&0&1&0&0&2&0&1\\ 0&0&0&0&1&0&2&0\\ 0&0&0&0&0&0&1&0\end {array} \right] \\ 
T(22,30) & =  \left[ \begin {array}{ccccccc} 0&1&0&0&0&0&0\\ 0&2&0&1&0&0&0\\ 1&0&2&0&0&1&0\\ 0&0&1&0&0&2&0\\ 0&0&0&0&1&0&2\\ 0&0&0&0&0&0&1\end {array} \right] \\ 
T(23,31) & =  \left[ \begin {array}{ccccccc} 1&0&0&0&0&0&0\\ 2&0&1&0&0&0&0\\ 1&0&2&0&1&0&0\\ 0&1&0&2&0&0&1\\ 0&0&0&1&0&0&2\\ 0&0&0&0&0&1&0\end {array} \right] \\ 
T(24,32) & =  \left[ \begin {array}{ccccccc} 1&0&0&0&0&0&0\\ 0&1&0&0&0&0&0\\ 0&2&0&1&0&0&0\\ 0&1&0&2&0&1&0\\ 0&0&1&0&2&0&0\\ 0&0&0&0&1&0&0\\ 0&0&0&0&0&0&1\end {array} \right] \\ 
T(25,24) & =  \left[ \begin {array}{ccccccc} 0&1&0&0&0&0&0\\ 2&0&1&0&0&0&0\\ 1&0&2&0&1&0&0\\ 0&0&1&0&2&0&1\\ 0&0&0&1&0&2&0\\ 0&0&0&0&0&1&0\end {array} \right] \\ 
T(25,25) & =  \left[ \begin {array}{cccccc} 0&1&0&0&0&0\\ 2&0&1&0&0&0\\ 1&0&2&0&1&0\\ 0&0&1&0&2&0\\ 0&0&0&1&0&2\\ 0&0&0&0&0&1\end {array} \right] \\ 
T(26,27) & =  \left[ \begin {array}{ccccccc} 1&0&0&0&0&0&0\\ 0&1&0&0&0&0&0\\ 2&0&0&1&0&0&0\\ 1&0&0&2&0&1&0\\ 0&0&1&0&2&0&1\\ 0&0&0&0&1&0&2\\ 0&0&0&0&0&0&1\end {array} \right] \\ 
T(27,28) & =  \left[ \begin {array}{ccccccc} 1&0&0&0&0&0&0\\ 0&1&0&0&0&0&0\\ 2&0&1&0&0&0&0\\ 0&2&0&0&1&0&0\\ 0&1&0&0&2&0&1\\ 0&0&0&1&0&2&0\\ 0&0&0&0&0&1&0\end {array} \right] \\ 
T(27,29) & =  \left[ \begin {array}{cccccccc} 0&1&0&0&0&0&0&0\\ 2&0&1&0&0&0&0&0\\ 0&2&0&1&0&0&0&0\\ 1&0&2&0&0&1&0&0\\ 0&0&1&0&0&2&0&1\\ 0&0&0&0&1&0&2&0\\ 0&0&0&0&0&0&1&0\end {array} \right] \\ 
T(27,30) & =  \left[ \begin {array}{ccccccc} 0&1&0&0&0&0&0\\ 2&0&1&0&0&0&0\\ 0&2&0&1&0&0&0\\ 1&0&2&0&0&1&0\\ 0&0&1&0&0&2&0\\ 0&0&0&0&1&0&2\\ 0&0&0&0&0&0&1\end {array} \right] \\ 
T(28,33) & =  \left[ \begin {array}{ccccccc} 1&0&0&0&0&0&0\\ 2&0&1&0&0&0&0\\ 0&2&0&1&0&0&0\\ 1&0&2&0&1&0&0\\ 0&1&0&2&0&0&1\\ 0&0&0&1&0&0&2\\ 0&0&0&0&0&1&0\end {array} \right] \\ 
T(28,34) & =  \left[ \begin {array}{cccccccc} 0&1&0&0&0&0&0&0\\ 0&2&0&1&0&0&0&0\\ 1&0&2&0&1&0&0&0\\ 0&1&0&2&0&1&0&0\\ 0&0&1&0&2&0&0&1\\ 0&0&0&0&1&0&0&2\\ 0&0&0&0&0&0&1&0\end {array} \right] \\ 
T(29,35) & =  \left[ \begin {array}{cccccccc} 1&0&0&0&0&0&0&0\\ 0&0&1&0&0&0&0&0\\ 0&0&2&0&1&0&0&0\\ 0&1&0&2&0&1&0&0\\ 0&0&1&0&2&0&1&0\\ 0&0&0&1&0&2&0&0\\ 0&0&0&0&0&1&0&0\\ 0&0&0&0&0&0&0&1\end {array} \right] \\ 
T(30,36) & =  \left[ \begin {array}{cccccccc} 0&1&0&0&0&0&0&0\\ 2&0&0&1&0&0&0&0\\ 1&0&0&2&0&1&0&0\\ 0&0&1&0&2&0&1&0\\ 0&0&0&1&0&2&0&1\\ 0&0&0&0&1&0&2&0\\ 0&0&0&0&0&0&1&0\end {array} \right] \\ 
T(30,37) & =  \left[ \begin {array}{ccccccc} 0&1&0&0&0&0&0\\ 2&0&0&1&0&0&0\\ 1&0&0&2&0&1&0\\ 0&0&1&0&2&0&1\\ 0&0&0&1&0&2&0\\ 0&0&0&0&1&0&2\\ 0&0&0&0&0&0&1\end {array} \right] \\ 
T(31,28) & =  \left[ \begin {array}{ccccccc} 1&0&0&0&0&0&0\\ 2&0&1&0&0&0&0\\ 0&2&0&0&1&0&0\\ 0&1&0&0&2&0&1\\ 0&0&0&1&0&2&0\\ 0&0&0&0&1&0&2\\ 0&0&0&0&0&1&0\end {array} \right] \\ 
T(31,29) & =  \left[ \begin {array}{cccccccc} 0&1&0&0&0&0&0&0\\ 0&2&0&1&0&0&0&0\\ 1&0&2&0&0&1&0&0\\ 0&0&1&0&0&2&0&1\\ 0&0&0&0&1&0&2&0\\ 0&0&0&0&0&1&0&2\\ 0&0&0&0&0&0&1&0\end {array} \right] \\ 
T(31,30) & =  \left[ \begin {array}{ccccccc} 0&1&0&0&0&0&0\\ 0&2&0&1&0&0&0\\ 1&0&2&0&0&1&0\\ 0&0&1&0&0&2&0\\ 0&0&0&0&1&0&2\\ 0&0&0&0&0&1&0\\ 0&0&0&0&0&0&1\end {array} \right] \\ 
T(32,31) & =  \left[ \begin {array}{ccccccc} 1&0&0&0&0&0&0\\ 2&0&1&0&0&0&0\\ 1&0&2&0&1&0&0\\ 0&1&0&2&0&0&1\\ 0&0&0&1&0&0&2\\ 0&0&0&0&0&1&0\\ 0&0&0&0&0&0&1\end {array} \right] \\ 
T(33,28) & =  \left[ \begin {array}{ccccccc} 1&0&0&0&0&0&0\\ 2&0&1&0&0&0&0\\ 0&2&0&0&1&0&0\\ 0&1&0&0&2&0&1\\ 0&0&0&1&0&2&0\\ 0&0&0&0&1&0&2\\ 0&0&0&0&0&1&0\end {array} \right] \\ 
T(33,29) & =  \left[ \begin {array}{cccccccc} 0&1&0&0&0&0&0&0\\ 0&2&0&1&0&0&0&0\\ 1&0&2&0&0&1&0&0\\ 0&0&1&0&0&2&0&1\\ 0&0&0&0&1&0&2&0\\ 0&0&0&0&0&1&0&2\\ 0&0&0&0&0&0&1&0\end {array} \right] \\ 
T(34,38) & =  \left[ \begin {array}{cccccccc} 1&0&0&0&0&0&0&0\\ 0&0&1&0&0&0&0&0\\ 0&0&2&0&1&0&0&0\\ 0&1&0&2&0&0&1&0\\ 0&0&0&1&0&0&2&0\\ 0&0&0&0&0&1&0&2\\ 0&0&0&0&0&0&1&0\\ 0&0&0&0&0&0&0&1\end {array} \right] \\ 
T(35,39) & =  \left[ \begin {array}{cccccccc} 1&0&0&0&0&0&0&0\\ 0&1&0&0&0&0&0&0\\ 2&0&0&1&0&0&0&0\\ 1&0&0&2&0&1&0&0\\ 0&0&1&0&2&0&0&1\\ 0&0&0&0&1&0&0&2\\ 0&0&0&0&0&0&1&0\\ 0&0&0&0&0&0&0&1\end {array} \right] \\ 
T(36,40) & =  \left[ \begin {array}{cccccccc} 1&0&0&0&0&0&0&0\\ 0&1&0&0&0&0&0&0\\ 2&0&1&0&0&0&0&0\\ 0&2&0&0&1&0&0&0\\ 0&1&0&0&2&0&1&0\\ 0&0&0&1&0&2&0&0\\ 0&0&0&0&0&1&0&0\\ 0&0&0&0&0&0&0&1\end {array} \right] \\ 
T(37,29) & =  \left[ \begin {array}{cccccccc} 0&1&0&0&0&0&0&0\\ 2&0&1&0&0&0&0&0\\ 0&2&0&1&0&0&0&0\\ 1&0&2&0&0&1&0&0\\ 0&0&1&0&0&2&0&1\\ 0&0&0&0&1&0&2&0\\ 0&0&0&0&0&0&1&0\end {array} \right] \\ 
T(37,30) & =  \left[ \begin {array}{ccccccc} 0&1&0&0&0&0&0\\ 2&0&1&0&0&0&0\\ 0&2&0&1&0&0&0\\ 1&0&2&0&0&1&0\\ 0&0&1&0&0&2&0\\ 0&0&0&0&1&0&2\\ 0&0&0&0&0&0&1\end {array} \right] \\ 
T(38,36) & =  \left[ \begin {array}{cccccccc} 1&0&0&0&0&0&0&0\\ 0&1&0&0&0&0&0&0\\ 2&0&0&1&0&0&0&0\\ 1&0&0&2&0&1&0&0\\ 0&0&1&0&2&0&1&0\\ 0&0&0&1&0&2&0&1\\ 0&0&0&0&1&0&2&0\\ 0&0&0&0&0&0&1&0\end {array} \right] \\ 
T(38,37) & =  \left[ \begin {array}{ccccccc} 1&0&0&0&0&0&0\\ 0&1&0&0&0&0&0\\ 2&0&0&1&0&0&0\\ 1&0&0&2&0&1&0\\ 0&0&1&0&2&0&1\\ 0&0&0&1&0&2&0\\ 0&0&0&0&1&0&2\\ 0&0&0&0&0&0&1\end {array} \right] \\ 
T(39,28) & =  \left[ \begin {array}{ccccccc} 1&0&0&0&0&0&0\\ 0&1&0&0&0&0&0\\ 2&0&1&0&0&0&0\\ 0&2&0&0&1&0&0\\ 0&1&0&0&2&0&1\\ 0&0&0&1&0&2&0\\ 0&0&0&0&1&0&2\\ 0&0&0&0&0&1&0\end {array} \right] \\ 
T(39,29) & =  \left[ \begin {array}{cccccccc} 0&1&0&0&0&0&0&0\\ 2&0&1&0&0&0&0&0\\ 0&2&0&1&0&0&0&0\\ 1&0&2&0&0&1&0&0\\ 0&0&1&0&0&2&0&1\\ 0&0&0&0&1&0&2&0\\ 0&0&0&0&0&1&0&2\\ 0&0&0&0&0&0&1&0\end {array} \right] \\ 
T(39,30) & =  \left[ \begin {array}{ccccccc} 0&1&0&0&0&0&0\\ 2&0&1&0&0&0&0\\ 0&2&0&1&0&0&0\\ 1&0&2&0&0&1&0\\ 0&0&1&0&0&2&0\\ 0&0&0&0&1&0&2\\ 0&0&0&0&0&1&0\\ 0&0&0&0&0&0&1\end {array} \right] \\ 
T(40,33) & =  \left[ \begin {array}{ccccccc} 1&0&0&0&0&0&0\\ 2&0&1&0&0&0&0\\ 0&2&0&1&0&0&0\\ 1&0&2&0&1&0&0\\ 0&1&0&2&0&0&1\\ 0&0&0&1&0&0&2\\ 0&0&0&0&0&1&0\\ 0&0&0&0&0&0&1\end {array} \right] \\ 
T(40,34) & =  \left[ \begin {array}{cccccccc} 0&1&0&0&0&0&0&0\\ 0&2&0&1&0&0&0&0\\ 1&0&2&0&1&0&0&0\\ 0&1&0&2&0&1&0&0\\ 0&0&1&0&2&0&0&1\\ 0&0&0&0&1&0&0&2\\ 0&0&0&0&0&0&1&0\\ 0&0&0&0&0&0&0&1\end {array} \right] \\ 
\end{align*}

The essential class is: [28, 29, 30, 33, 34, 35, 36, 37, 38, 39, 40].
The essential class is of positive type.
An example is the path [28, 34, 38, 36, 40, 34, 38, 37].
The essential class is not a simple loop.
This spectral range will include the interval $[.6172895104, .6202985760]$.
The minimum comes from the loop $[29, 35, 39, 29]$.
The maximum comes from the loop $[28, 33, 28]$.
These points will include points of local dimension [3.873239615, 3.883344935].
The Spectral Range is contained in the range $[.5095977274, .6753430785]$.
The minimum comes from the column sub-norm on the subset ${{3, 4}}$ of length 20. 
The maximum comes from the total column sup-norm of length 10. 
These points will have local dimension contained in [3.696560893, 4.281748469].

There are 4 additional maximal loops.

Maximal Loop Class: [25].
The maximal loop class is a simple loop.
It's spectral radius is an isolated points of .6202985760.
These points have local dimension 3.873239615.

Maximal Loop Class: [19].
The maximal loop class is a simple loop.
It's spectral radius is an isolated points of .6202985760.
These points have local dimension 3.873239615.

Maximal Loop Class: [6].
The maximal loop class is a simple loop.
It's spectral radius is an isolated points of .2500000000.
These points have local dimension 5.761680361.

Maximal Loop Class: [2].
The maximal loop class is a simple loop.
It's spectral radius is an isolated points of .2500000000.
These points have local dimension 5.761680361.

\section{Minimal polynomial $3 x-1$ with $d_i \in [0, 2/9, 4/9, 2/3]$} 
\label{sec:18}
 
Consider $\varrho$, the root of $3 x-1$ and the maps $S_i(x) = \varrho x  + d_i$ with $d_{0} = 0$, $d_{1} = 2/9$, $d_{2} = 4/9$, and $d_{3} = 2/3$.
The probabilities are uniform.
The reduced transition diagram has 5 reduced characteristic vectors.
The reduced characteristic vectors are:
\begin{itemize}
\item Reduced characteristic vector 1: $(1, (0))$ 
\item Reduced characteristic vector 2: $(2/3, (0))$ 
\item Reduced characteristic vector 3: $(1/3, (0, 2/3))$ 
\item Reduced characteristic vector 4: $(1/3, (1/3))$ 
\item Reduced characteristic vector 5: $(2/3, (1/3))$ 
\end{itemize}
See Figure \ref{fig:Pic27} for the transition diagram.
\begin{figure}[H]
\includegraphics[scale=0.5]{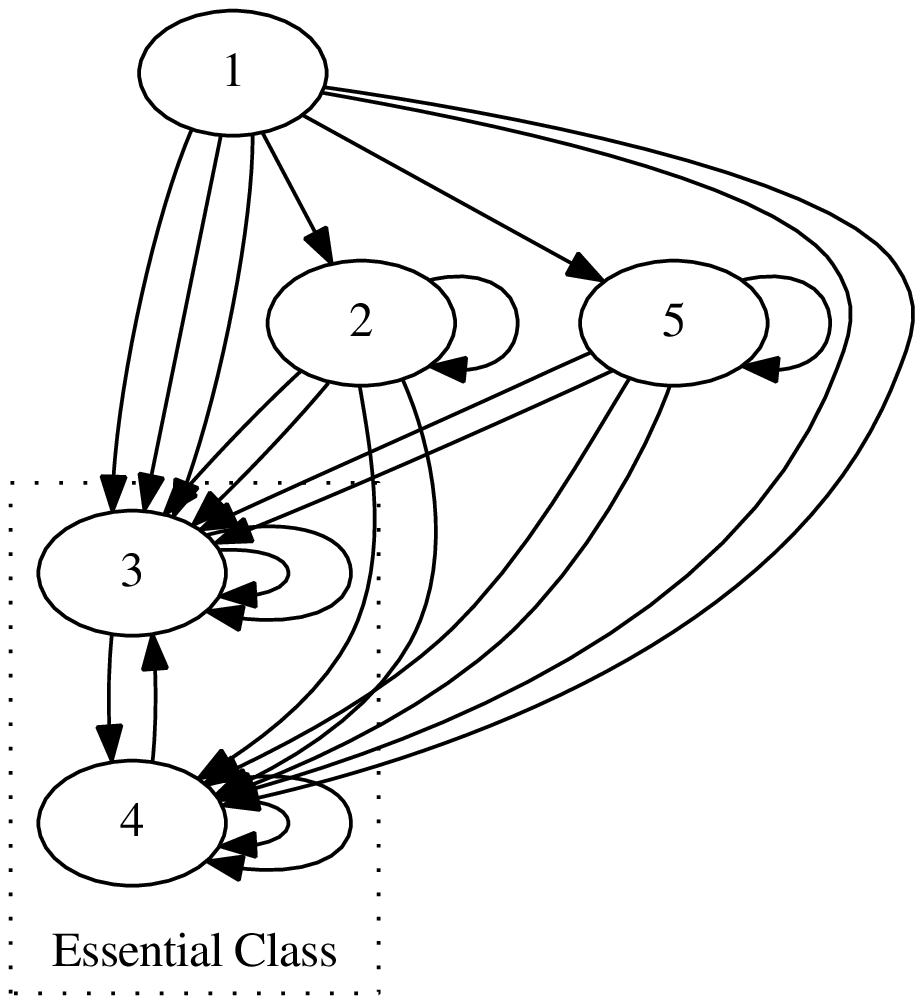}
\caption{$3 x-1$ with $d_i \in [0, 2/9, 4/9, 2/3]$, Full set + Essential class}
\label{fig:Pic27}
\end{figure}
This has transition matrices:
\begin{align*}
T(1,2) & =  \left[ \begin {array}{c} 1\end {array} \right] & 
 T(1,3) & =  \left[ \begin {array}{cc} 1&1\end {array} \right] \\ 
T(1,4) & =  \left[ \begin {array}{c} 1\end {array} \right] & 
 T(1,3) & =  \left[ \begin {array}{cc} 1&1\end {array} \right] \\ 
T(1,4) & =  \left[ \begin {array}{c} 1\end {array} \right] & 
 T(1,3) & =  \left[ \begin {array}{cc} 1&1\end {array} \right] \\ 
T(1,5) & =  \left[ \begin {array}{c} 1\end {array} \right] & 
 T(2,2) & =  \left[ \begin {array}{c} 1\end {array} \right] \\ 
T(2,3) & =  \left[ \begin {array}{cc} 1&1\end {array} \right] & 
 T(2,4) & =  \left[ \begin {array}{c} 1\end {array} \right] \\ 
T(2,3) & =  \left[ \begin {array}{cc} 1&1\end {array} \right] & 
 T(2,4) & =  \left[ \begin {array}{c} 1\end {array} \right] \\ 
T(3,3) & =  \left[ \begin {array}{cc} 1&0\\ 1&1\end {array} \right] & 
 T(3,4) & =  \left[ \begin {array}{c} 1\\ 1\end {array} \right] \\ 
T(3,3) & =  \left[ \begin {array}{cc} 1&1\\ 0&1\end {array} \right] & 
 T(4,4) & =  \left[ \begin {array}{c} 1\end {array} \right] \\ 
T(4,3) & =  \left[ \begin {array}{cc} 1&1\end {array} \right] & 
 T(4,4) & =  \left[ \begin {array}{c} 1\end {array} \right] \\ 
T(5,4) & =  \left[ \begin {array}{c} 1\end {array} \right] & 
 T(5,3) & =  \left[ \begin {array}{cc} 1&1\end {array} \right] \\ 
T(5,4) & =  \left[ \begin {array}{c} 1\end {array} \right] & 
 T(5,3) & =  \left[ \begin {array}{cc} 1&1\end {array} \right] \\ 
T(5,5) & =  \left[ \begin {array}{c} 1\end {array} \right] & 
 \end{align*}

The essential class is: [3, 4].
The essential class is of positive type.
An example is the path [3, 4].
The essential class is not a simple loop.
This spectral range will include the interval $[1., 1.618033989]$.
The minimum comes from the loop $[3, 3]$.
The maximum comes from the loop $[3, 3, 3]$.
These points will include points of local dimension [.8238416275, 1.261859507].
The Spectral Range is contained in the range $[1., 1.670277652]$.
The minimum comes from the total row sub-norm of length 5. 
The maximum comes from the total row sup-norm of length 5. 
These points will have local dimension contained in [.7949160034, 1.261859507].

There are 2 additional maximal loops.

Maximal Loop Class: [5].
The maximal loop class is a simple loop.
It's spectral radius is an isolated points of 1.
These points have local dimension 1.261859507.

Maximal Loop Class: [2].
The maximal loop class is a simple loop.
It's spectral radius is an isolated points of 1.
These points have local dimension 1.261859507.

\section{Minimal polynomial $3 x-1$ with $d_i \in [0, 1/6, 1/3, 1/2, 2/3]$} 
\label{sec:19}
 
Consider $\varrho$, the root of $3 x-1$ and the maps $S_i(x) = \varrho x  + d_i$ with $d_{0} = 0$, $d_{1} = 1/6$, $d_{2} = 1/3$, $d_{3} = 1/2$, and $d_{4} = 2/3$.
The probabilities are uniform.
The reduced transition diagram has 4 reduced characteristic vectors.
The reduced characteristic vectors are:
\begin{itemize}
\item Reduced characteristic vector 1: $(1, (0))$ 
\item Reduced characteristic vector 2: $(1/2, (0))$ 
\item Reduced characteristic vector 3: $(1/2, (0, 1/2))$ 
\item Reduced characteristic vector 4: $(1/2, (1/2))$ 
\end{itemize}
See Figure \ref{fig:Pic28} for the transition diagram.
\begin{figure}[H]
\includegraphics[scale=0.5]{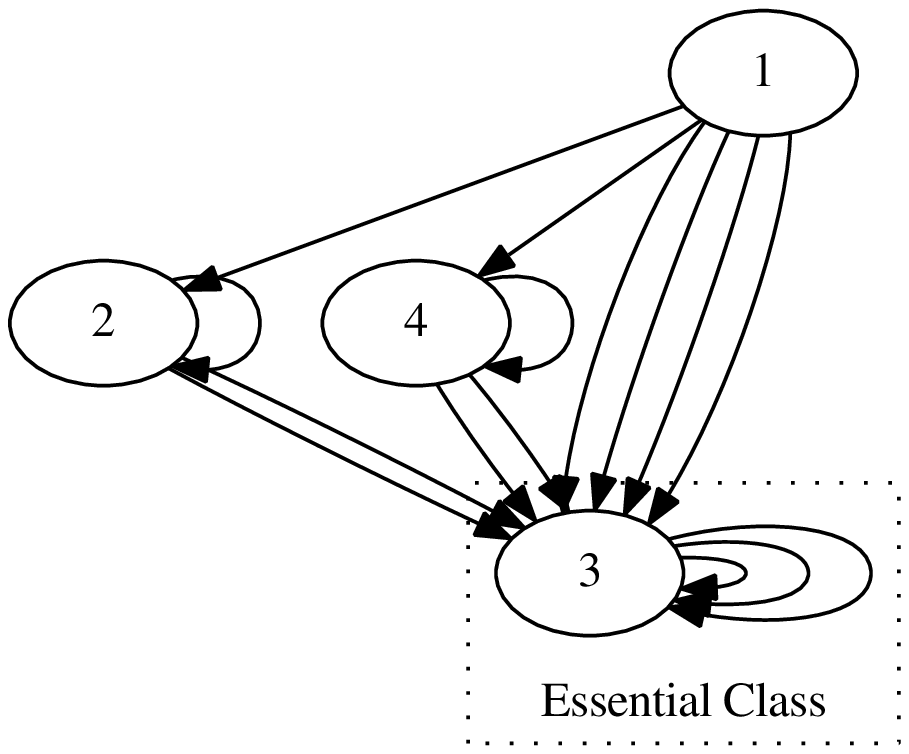}
\caption{$3 x-1$ with $d_i \in [0, 1/6, 1/3, 1/2, 2/3]$, Full set + Essential class}
\label{fig:Pic28}
\end{figure}
This has transition matrices:
\begin{align*}
T(1,2) & =  \left[ \begin {array}{c} 1\end {array} \right] & 
 T(1,3) & =  \left[ \begin {array}{cc} 1&1\end {array} \right] \\ 
T(1,3) & =  \left[ \begin {array}{cc} 1&1\end {array} \right] & 
 T(1,3) & =  \left[ \begin {array}{cc} 1&1\end {array} \right] \\ 
T(1,3) & =  \left[ \begin {array}{cc} 1&1\end {array} \right] & 
 T(1,4) & =  \left[ \begin {array}{c} 1\end {array} \right] \\ 
T(2,2) & =  \left[ \begin {array}{c} 1\end {array} \right] & 
 T(2,3) & =  \left[ \begin {array}{cc} 1&1\end {array} \right] \\ 
T(2,3) & =  \left[ \begin {array}{cc} 1&1\end {array} \right] & 
 T(3,3) & =  \left[ \begin {array}{cc} 1&0\\ 1&1\end {array} \right] \\ 
T(3,3) & =  \left[ \begin {array}{cc} 1&1\\ 1&1\end {array} \right] & 
 T(3,3) & =  \left[ \begin {array}{cc} 1&1\\ 0&1\end {array} \right] \\ 
T(4,3) & =  \left[ \begin {array}{cc} 1&1\end {array} \right] & 
 T(4,3) & =  \left[ \begin {array}{cc} 1&1\end {array} \right] \\ 
T(4,4) & =  \left[ \begin {array}{c} 1\end {array} \right] & 
 \end{align*}

The essential class is: [3].
The essential class is of positive type.
An example is the path [3, 3].
The essential class is not a simple loop.
This spectral range will include the interval $[1., 2.]$.
The minimum comes from the loop $[3, 3]$.
The maximum comes from the loop $[3, 3]$.
These points will include points of local dimension [.8340437666, 1.464973520].
The Spectral Range is contained in the range $[1., 2.000000000]$.
The minimum comes from the total row sub-norm of length 5. 
The maximum comes from the total row sup-norm of length 5. 
These points will have local dimension contained in [.8340437666, 1.464973520].

There are 2 additional maximal loops.

Maximal Loop Class: [4].
The maximal loop class is a simple loop.
It's spectral radius is an isolated points of 1.
These points have local dimension 1.464973520.

Maximal Loop Class: [2].
The maximal loop class is a simple loop.
It's spectral radius is an isolated points of 1.
These points have local dimension 1.464973520.

\section{Minimal polynomial $3 x-1$ with $d_i \in [0, 2/15, 4/15, 2/5, 8/15, 2/3]$} 
\label{sec:20}
 
Consider $\varrho$, the root of $3 x-1$ and the maps $S_i(x) = \varrho x  + d_i$ with $d_{0} = 0$, $d_{1} = 2/15$, $d_{2} = 4/15$, $d_{3} = 2/5$, $d_{4} = 8/15$, and $d_{5} = 2/3$.
The probabilities are uniform.
The reduced transition diagram has 7 reduced characteristic vectors.
The reduced characteristic vectors are:
\begin{itemize}
\item Reduced characteristic vector 1: $(1, (0))$ 
\item Reduced characteristic vector 2: $(2/5, (0))$ 
\item Reduced characteristic vector 3: $(2/5, (0, 2/5))$ 
\item Reduced characteristic vector 4: $(1/5, (0, 2/5, 4/5))$ 
\item Reduced characteristic vector 5: $(1/5, (1/5, 3/5))$ 
\item Reduced characteristic vector 6: $(2/5, (1/5, 3/5))$ 
\item Reduced characteristic vector 7: $(2/5, (3/5))$ 
\end{itemize}
See Figure \ref{fig:Pic29} for the transition diagram.
\begin{figure}[H]
\includegraphics[scale=0.5]{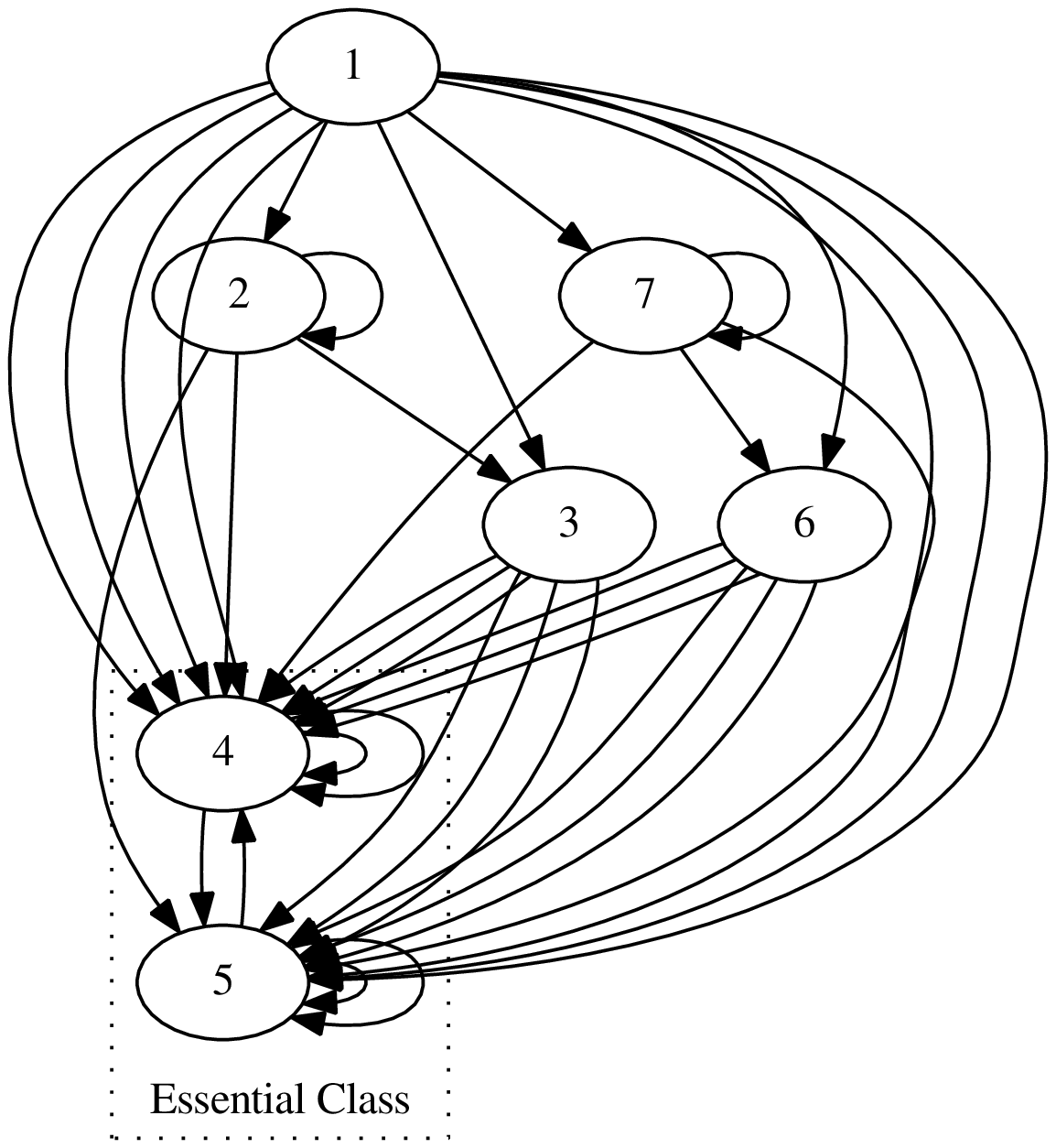}
\caption{$3 x-1$ with $d_i \in [0, 2/15, 4/15, 2/5, 8/15, 2/3]$, Full set + Essential class}
\label{fig:Pic29}
\end{figure}
This has transition matrices:
\begin{align*}
T(1,2) & =  \left[ \begin {array}{c} 1\end {array} \right] & 
 T(1,3) & =  \left[ \begin {array}{cc} 1&1\end {array} \right] \\ 
T(1,4) & =  \left[ \begin {array}{ccc} 1&1&1\end {array} \right] & 
 T(1,5) & =  \left[ \begin {array}{cc} 1&1\end {array} \right] \\ 
T(1,4) & =  \left[ \begin {array}{ccc} 1&1&1\end {array} \right] & 
 T(1,5) & =  \left[ \begin {array}{cc} 1&1\end {array} \right] \\ 
T(1,4) & =  \left[ \begin {array}{ccc} 1&1&1\end {array} \right] & 
 T(1,5) & =  \left[ \begin {array}{cc} 1&1\end {array} \right] \\ 
T(1,4) & =  \left[ \begin {array}{ccc} 1&1&1\end {array} \right] & 
 T(1,6) & =  \left[ \begin {array}{cc} 1&1\end {array} \right] \\ 
T(1,7) & =  \left[ \begin {array}{c} 1\end {array} \right] & 
 T(2,2) & =  \left[ \begin {array}{c} 1\end {array} \right] \\ 
T(2,3) & =  \left[ \begin {array}{cc} 1&1\end {array} \right] & 
 T(2,4) & =  \left[ \begin {array}{ccc} 1&1&1\end {array} \right] \\ 
T(2,5) & =  \left[ \begin {array}{cc} 1&1\end {array} \right] & 
 T(3,4) & =  \left[ \begin {array}{ccc} 1&0&0\\ 1&1&1\end {array} \right] \\ 
T(3,5) & =  \left[ \begin {array}{cc} 1&0\\ 1&1\end {array} \right] & 
 T(3,4) & =  \left[ \begin {array}{ccc} 1&1&0\\ 1&1&1\end {array} \right] \\ 
T(3,5) & =  \left[ \begin {array}{cc} 1&1\\ 1&1\end {array} \right] & 
 T(3,4) & =  \left[ \begin {array}{ccc} 1&1&1\\ 1&1&1\end {array} \right] \\ 
T(3,5) & =  \left[ \begin {array}{cc} 1&1\\ 1&1\end {array} \right] & 
 T(4,4) & =  \left[ \begin {array}{ccc} 1&0&0\\ 1&1&1\\ 0&1&1\end {array} \right] \\ 
T(4,5) & =  \left[ \begin {array}{cc} 1&0\\ 1&1\\ 0&1\end {array} \right] & 
 T(4,4) & =  \left[ \begin {array}{ccc} 1&1&0\\ 1&1&1\\ 0&0&1\end {array} \right] \\ 
T(5,5) & =  \left[ \begin {array}{cc} 1&1\\ 1&1\end {array} \right] & 
 T(5,4) & =  \left[ \begin {array}{ccc} 1&1&1\\ 1&1&1\end {array} \right] \\ 
T(5,5) & =  \left[ \begin {array}{cc} 1&1\\ 1&1\end {array} \right] & 
 T(6,5) & =  \left[ \begin {array}{cc} 1&1\\ 1&1\end {array} \right] \\ 
T(6,4) & =  \left[ \begin {array}{ccc} 1&1&1\\ 1&1&1\end {array} \right] & 
 T(6,5) & =  \left[ \begin {array}{cc} 1&1\\ 1&1\end {array} \right] \\ 
T(6,4) & =  \left[ \begin {array}{ccc} 1&1&1\\ 0&1&1\end {array} \right] & 
 T(6,5) & =  \left[ \begin {array}{cc} 1&1\\ 0&1\end {array} \right] \\ 
T(6,4) & =  \left[ \begin {array}{ccc} 1&1&1\\ 0&0&1\end {array} \right] & 
 T(7,5) & =  \left[ \begin {array}{cc} 1&1\end {array} \right] \\ 
T(7,4) & =  \left[ \begin {array}{ccc} 1&1&1\end {array} \right] & 
 T(7,6) & =  \left[ \begin {array}{cc} 1&1\end {array} \right] \\ 
T(7,7) & =  \left[ \begin {array}{c} 1\end {array} \right] & 
 \end{align*}

The essential class is: [4, 5].
The essential class is of positive type.
An example is the path [5, 4].
The essential class is not a simple loop.
This spectral range will include the interval $[2., 2.]$.
The minimum comes from the loop $[4, 4]$.
The maximum comes from the loop $[4, 4]$.
These points will include points of local dimension [.9999999996, .9999999996].
The Spectral Range is contained in the range $[2.000000000, 2.000000000]$.
The minimum comes from the total row sub-norm of length 5. 
The maximum comes from the total row sup-norm of length 5. 
These points will have local dimension contained in [.9999999996, .9999999996].

There are 2 additional maximal loops.

Maximal Loop Class: [7].
The maximal loop class is a simple loop.
It's spectral radius is an isolated points of 1.
These points have local dimension 1.630929753.

Maximal Loop Class: [2].
The maximal loop class is a simple loop.
It's spectral radius is an isolated points of 1.
These points have local dimension 1.630929753.

\end{document}